\documentclass[preprint,1p times,12pt]{elsarticle}

\journal{Physica D}

\usepackage{natbib}
\RequirePackage{caption}
\usepackage{epigraph}

\usepackage[utf8]{inputenc}
\usepackage[T1]{fontenc} 
\usepackage{hyperref}   
\usepackage{url}         
\usepackage{booktabs}  
\usepackage{amsfonts}      
\usepackage{nicefrac}       
\usepackage{microtype}      
\usepackage{xcolor}     
\usepackage{amsthm}

\usepackage{amsmath}
\usepackage{amsfonts}
\usepackage{amssymb}
\usepackage{amsthm}
\usepackage{cuted}
\usepackage{enumitem}
\usepackage{url}

\usepackage{algorithm}
\usepackage{algpseudocode}

    \newcommand{\commentOUT}[1]{}
    
\newcommand{\ex}{{\bf\sf E}}               

\newcommand{\bz}{{\bf z}}               

\newcommand{\ra}{\rightarrow}           

\newcommand{\diag}{{\rm diag}\,}
\newcommand{\vect}{{\rm vect}\,}

\def\ex{\mathbb{E}}

\newcommand{\Real}{\mathbb R}

\newcommand{\al}{\alpha}                
\newcommand{\g}{\lambda}                
\newcommand{\lam}{\lambda}               
\newcommand{\Lam}{\Lambda}               

\newtheorem{pro}{Proposition}
\newtheorem{thm}{Theorem}[section]
\newtheorem{lem}{Lemma}[section]
\newtheorem{rem}{Remark}[section]

\newtheorem{cor}{Corollary}

\newcommand{\tr}{{\rm Tr}}
\newcommand{\Id}{{\rm I}}

\begin{document}

\begin{frontmatter}



\title{Dynamical Behaviors of the Gradient Flows for In-Context Learning}


\author[IBM]{Songtao Lu}
\author[IBM]{Yingdong Lu}
\author[IBM]{Tomasz Nowicki} 

\affiliation[IBM]{organization={IBM T.J. Watson Research Center},
            addressline={1101 Kitchawan Rd}, 
            city={Yorktown Heights},
            postcode={10598}, 
            state={New York},
            country={U.S.A}}


\begin{abstract}
We derive the system of differential equations for the gradient flow characterizing the training process of linear in-context learning in full generality. Next, we explore the geometric structure of the gradient flows in two instances, including identifying its invariants, optimum, and saddle points. This understanding allows us to quantify the behavior of the two gradient flows under the full generality of parameters and data.
\end{abstract}

\end{frontmatter}

\epigraph{6accdae13eﬀ7i3l9n4o4qrr4s8t12ux}{\textit{Isaac Newton \\ \emph{Letter to Gottfried W. Leibniz}, 1677}}

\section{Introduction}
\label{sec:intro}

In-context learning (ICL) is a versatile capability to predict input-output dependence on test data without making any updates to the parameters which are obtained from training data. It is one of the key features that establishes transformer-based neural networks as the default machine learning model. For further discussions on the basic mechanism and main benefit of ICL in terms of avoiding expensive fine-tuning and enhancing learning capabilities, see, e.g.,~\cite{li2024fine},~\cite{lampinen2022can},~\cite{wang2024theoretical}, and~\cite{wei2022chain}.

There have been extensive studies on ICL from various theoretical aspects; see, e.g.~\cite{garg2022can}. We focus on an influential linear attention model on the gradient flow of its training dynamics that has been analyzed in~\cite{zhang2023trained}, and adapted in many follow-up studies, see, e.g.\cite{huang2023context},~\cite{kim2024transformers}, and~~\cite{yang2024context}. So far, these studies have focused on model cases and have demonstrated convergence of gradient flow, thus the training dynamics, for the range of initial conditions and parameters. 

\subsection{Our Goals}

Our first goal is to provide \emph{a complete characterization} of a system of differential equations of the gradient flow, using a rather elementary approach. This will serve as the cornerstone for a more comprehensive understanding of the gradient flow dynamics and, eventually, of the learning dynamics for~ICL. 

Due to the complicated nature of the system under consideration, the system of resulting differential equations is highly involved. To focus on revealing its structural properties, we conduct a detailed analysis of two instances that we believe to be equipped with all the necessary features of the system. 

First, we consider a system that generalizes the one in~\cite{zhang2023trained}, which ignores the contributions from some of the weight matrices. For this system of differential equations, we perform detailed calculations to characterize its invariant sets, saddle points, and equilibrium points explicitly, and quantify the behavior under \emph{all initial conditions and parameters}, not just the very restricted ones considered in the literature before. 
For references of related concepts and approaches in dynamical systems employed in this paper, see, e.g.,~\cite{palis2012geometric}. 

Next, we study a system consisting of all the weights but of low dimension. For the resulting system of four differential equations, we aim to provide \emph{an exhaustive qualitative description} of the gradient flow, formulae for its critical points, and establish their character.

\subsection{Our Contributions}

The main contributions of this paper include the following.

\begin{enumerate}
\item In Theorem~\ref{thm:differentialEquations}, we present a complete system of differential equations for the gradient flow that represents the ICL training process with the linear attention function. Compared to previous treatments, see, e.g.~\cite{zhang2023trained} our arguments are a conceptually straightforward approach. This simplicity is a critical enabler for characterizing the dynamics in full generality. This may allow further analysis to be conducted to include all weight matrices. 
\item
In Theorem~\ref{thm:critical_points}, we provide a more detailed and general analysis of the structural and stability properties of a simplified system that appeared in~\cite{zhang2023trained} and many follow-up research works. More specifically, we are able to characterize the invariant set for all the initial conditions and parameters, identify all the critical points, and quantify the behavior in their neighborhoods. Together with the gradient descent character of dynamics, this qualitatively describes the global behavior of the system. The analysis is conducted for the full spectrum of parameters with any initial conditions, and we are able to classify different behaviors as the parameters and initial conditions vary, see Remark~\ref{rem:global}. 
\item
In Theorems~\ref{thm:1DCriticalPoints} and~\ref{thm:1DCharacter}, we illustrate the characterization of the system by a detailed analysis of an example of low dimension. More specifically, we are able to identify the basic structure of the invariant manifolds, the critical points, and their qualitative behavior. We note qualitative similarities between this system and the simplified one.
\end{enumerate}

\subsection{Previous Works}

Motivated by the exceptional numerical performance of transformers in natural language processing, a crucial first step is to understand the principles underlying their ability to extract latent features from data samples. One particularly remarkable feature is their few-shot learning capability, where pre-trained models adapt to additional context provided in the input without requiring weight updates. This has spurred a line of research into the interpretability of transformer-based models, with studies exploring transformers as decision-makers \cite{lin2023transformers}, max-margin token selectors \cite{ataee2023max}, approximations of Kalman filters \cite{goel2023can}, and nearest-neighbor learners \cite{li2024one}, among other frameworks.

Building on the exploration of how transformers learn in-context representations, recent studies reveal that transformer structures operate as learning algorithms for prediction, highlighting their capacity for implicit fine-tuning of pre-trained models during inference \cite{oswald23}. For instance, \cite{ahn2023transformers} presents a comprehensive characterization of the global optimum of a single-layer linear transformer trained on i.i.d. Gaussian data. The findings demonstrate that the transformer architecture functions as a preconditioned gradient descent for linear regression models, bridging its architecture with optimization principles. Moreover, \cite{akurek2023what} examines the statistical properties of transformers, showing that trained in-context learners align closely with predictors obtained via gradient descent, ridge regression, or least squares regression. The specific alignment depends on factors such as the depth of the transformer and the noise characteristics of the dataset. These insights collectively illustrate the versatility of transformers in implementing diverse machine learning algorithms, furthering our understanding of their role as adaptive predictive mechanisms.

Even though these results provide valuable insights into the expressiveness of learned transformers across various scenarios, a more fundamental question remains unanswered: \emph{Is gradient descent on transformer model parameters sufficient to find optimal solutions to the loss function?} Initial studies have begun exploring the training dynamics of transformers for specific input settings. For example, \cite{tian2023scan} demonstrates that a self-attention mechanism within a one-layer transformer augmented with a decoder layer functions as a discriminative scanning algorithm, assigning attention scores based on patterns of co-occurrence between key and query tokens during training. This analysis was later extended to multi-layer transformers in \cite{tian2023joma}, revealing the sparse nature of attention weights in hierarchical structures. While these foundational works offer an understanding of parameter convergence during training, they fall short of determining whether this interpretable learning mechanism aligns with the optimality of minimizing the loss function from an optimization perspective. Bridging this gap is crucial for comprehending the role of transformer modules in pre-trained language models.
 
One of the earliest studies on transformer optimization dynamics focuses on the one-layer linear transformer architecture \cite{zhang2023trained}. It examines the gradient flow of model updates and concludes that gradient descent can achieve the global optimality of model parameters despite the nonconvexity of the loss function, provided certain conditions are met. These conditions include specific parameter initialization distributions and the exclusion of a subset of parameters during optimization. Building on this result, subsequent research explores the impact of incorporating a multi-layer perceptron layer into the transformer block, demonstrating that the infinite-dimensional attention landscape is benign and that mean-field dynamics ensure global convergence rates.  Additionally, a reduced parameter space of attention models is studied in \cite{huang2023context} using cross-entropy loss, where data is assumed to be generated by orthogonal vectors. This work also addresses nonlinear attention models, showing that gradient descent achieves global optimality at a sublinear convergence rate. More recent research \cite{yang2024context} investigates the training dynamics of one-layer multi-head transformers, establishing that the training loss under gradient descent converges to the global minimum at a linear rate for certain nonlinear regression tasks \cite{kim2024transformers}. Notably, the multi-head soft attention mechanism exhibits a ``task allocation'' phenomenon, effectively distributing attention heads across tasks in multi-task linear regression.  Finally, \cite{chen2024training} maps gradient flow dynamics in parameter space to spectral dynamics in the eigenspace of data features. This analysis reveals that each attention head specializes in optimally handling a single task in multi-task settings, even when the loss function is nonconvex. These findings collectively highlight the inherent special and benign structures of the loss functions built upon transformer architectures, which makes gradient descent capable of finding optimal solutions under various structural and data assumptions.

Moving beyond Gaussian data, a more practical Markov sequence is considered in \cite{makkuva2024local}, showing that the gradient flow of transformer parameters trained on next-token prediction loss can converge to either global or local minima, depending on initialization and the properties of the Markovian data. However, these results apply only to binary inputs. A more realistic $n$-gram Markov chain model is introduced in \cite{chen2024training}, which helps reveal the ``induction head'' mechanism of a two-attention-layer transformer by analyzing the convergence of gradient flow with respect to a cross-entropy ICL loss during training. Although the contributions of each building block of this transformer module—namely, the first-attention layer, feed-forward layer, and second-attention layer—are characterized, the training dynamics are analyzed unrealistically, where the model parameters of each block are updated entirely before switching to the next block, solely for the convenience of theoretical analysis.

Beyond single-layer models, under a set of conditions on the problem settings—including restrictive data distribution assumptions and specific distributions for model parameter initializations—\cite{gatmiry2024can} demonstrates that the looped attention model can converge to a near-global minimum of the loss function. This is achieved by characterizing the gradient flow of model updates for a sufficiently large number of samples. Despite considering more complex transformer architectures, the fundamental question of characterizing the training dynamics of model parameters in general settings remains unanswered. Therefore, it is highly motivated to investigate the gradient flow dynamics of the original linear attention model as discussed in \cite{zhang2023trained}.

\subsection{Organization of the Paper}

The preliminary results and notation are presented in Sec.~\ref{sec:prelim}. The complete set of differential equations is presented and derived in Sec.~\ref{sec:differential_equations}. The analysis of a simplified system is presented in Sec.~\ref{sec:analysis_uU}. The analysis of a low-dimension system is presented in Sec.~\ref{sec:analysis_1D}. The paper is concluded in Sec.~\ref{sec:conclusions} with a summary of the main results and future research directions.

\section{Preliminaries and Notations}
\label{sec:prelim}

In this section, we will provide the basic setup of the optimization in training of ICL.

For $k=1,2,\ldots,N$, $x_{k}$ represents an independent and identically distributed $d$-variate normal random variable with mean zero and $d\times d$ covariance matrix $\Lam$, that is, they follow the distribution of $N(0,\Lam)$, and $x_k^{(i)}$ represents the $i$-th component of the $d$-dimensional vector. $x_{q}$ is assume to follow the same distribution $N(0,\Lam)$. $w$ follows a standard $d$-variate normal, $N(0, I_d)$. For each $i=1,2,\ldots, d$, $w^{(i)}$ represents the $i$-th component of the $d$-dimensional vector,
\begin{align*}
x_k=&\left(\begin{array}{c}x_k^{(1)} \\ x_k^{(2)}\\ \cdot \\ \cdot \\ \cdot \\x_k^{(d)}
\end{array}\right), \quad w=\left(\begin{array}{c}w^{(1)} \\ w^{(2)}\\ \cdot \\ \cdot \\ \cdot \\w^{(d)}
\end{array}\right).
\end{align*}
A $(d+1)\times (N+1)$ matrix $E$ then can be defined as, 
\begin{align*}
E=\left( \begin{array}{ccccc} x_{1}&x_{2}& \ldots &x_{N}& x_{q}\\ y_{1}&y_{2}& \ldots &y_{N}&  0 \end{array}\right)
\end{align*}
with $y_k:=w^\top x_k$, $k=1,2,\ldots,N$. A $(d+1)^2\times(d+1)^2$ matrix $H$ is defined as $H:= X\otimes (EE^\top)$ with $X=(0, x_q;x_q^\top,0)$, where $\otimes$ denotes the Kronecker product.
Therefore, 
\begin{align*}
H=\left(\begin{array}{cc} 0 & x_q\otimes (EE^\top) \\x_q^\top\otimes (EE^\top) & 0 \end{array}\right).
\end{align*}

For random $(d+1)^2\times (d+1)^2$ matrix $H$ and scalar $A=\sum_{m=1}^d w^{(m)} x_q^{(m)}$, the loss function to be minimized $L(\xi):\Real^{(d+1)^2}\ra \Real$ is defined as,
\begin{align}
\label{eqn:lost_function}
L(\xi)=\ex[\mathbf{L}(\xi)]=\frac{1}{4}\ex[(\xi^\top H \xi-A)^2]
\end{align}
variable $\xi$ is the vectorization of the matrix, $\xi=Vec(U,\bz;Z^\top,v)$, 
with $U$ being a $d\times d$ matrix ($U=(U_{ij})^d_{i,j=1}$ as usual), $\bz$ and $Z$ are $d$ dimensional vectors and $v$ a scalar.

\begin{rem}

Here, we provide some background information on how the above setup of the problem is related to training in in-context learning. More detailed background can be found in~\cite{zhang2023trained} and references therein.

In-Context learning models train on the sequences of input-output pairs. These sequences are termed \emph{prompts}, and appear in the form of $E$ with $x_i$ being the input and $y_i$ the output for $i=1,2,\ldots, N$, and the learning task is to predict the outcome for the input $x_{query}$.

Attention function is assumed to take the form of
\begin{align*}
f_A(E: W^K, W^Q, W^V, W^P) = E+ W^PW^VE\cdot \hbox{softmax}((W^KE)^TW^QE),
\end{align*}
where weight matrices include: 1) key weight matrix $W^K$ of dimension $d\times(d+1)$; 2) query weight matrix $W^Q$ of dimension $d\times(d+1)$; 3) value weight matrix $W^V$ of dimension $d\times(d+1)$; and 4)project matrix $W^P$ of dimension $(d+1)\times d$. This is further simplified to be the linear attention model of $f_A$.
\begin{align*}
f_{LA}(E: W^K, W^Q, W^V, W^P) = E+ W^{PV}E\cdot (E)^TW^{KQ}E). 
\end{align*}
Thus, $f_A(E: W^K, W^Q, W^V, W^P)$ has the same dimension as $E$, and its $(N+1)\times (N+1)$ element represents the predicted value. In the training process, it is assumed that the inputs $x_n, n=1$ and the outputs $x_n, n=1$ are drawn according to the prescribed distributions. The neural network minimizes the sample average version of the loss function~\eqref{eqn:lost_function},
\begin{align*}
\frac{1}{B} \sum_{i=1}^B \left((f_A(E^i))_{(N+1)\times (N+1)} - f^i(x^i_{N+1})\right)^2.
\end{align*}
\end{rem}
In the ICL training process of, the loss function is minimized through various gradient descent algorithm implementations, which are discretized gradient flows. Therefore, to shed light on the behavior of the training process, we investigate the gradient flow associated with the loss function minimization.

\section{System of Differential Equations Representing the Gradient Flows}
\label{sec:differential_equations}

In this section, we will derive the complete system of differential equations for the gradient flow. The basic approach we take is to write the loss function in summation form, instead of the matrix form that has been used in most of the previous literature. Although they appear to be less concise, their calculations are straightforward. The main elements of the calculations involved are exchange of the orders of summations, taking expectations sequentially, and higher (fourth and sixth) moments of multivariate normal variables. 

\subsection{Summation Form of the Loss Function}

The key term in the loss function, $\xi^\top H \xi$, represents the prediction ${\hat y}_{query}$. Here we present an alternative expression in summation form instead of matrix form. Note that it is consistent with the original expression of the prediction, see equation (3.6) in~\cite{zhang2023trained}, more precisely,
\begin{align}
\label{eqn:prediction_alt}
\xi^\top H \xi={\hat y}_{query}=\frac{1}{N}\left[((u_{12})^\top, u_{-1})(EE^\top)\left(\begin{array}{c} U_{11} \\ (u_{21})^\top\end{array}\right) x_{query}\right],
\end{align}
with $((u_{12})^\top, u_{-1})$ represents $W^{PV}$, and $\left(\begin{array}{c} U_{11} \\ (u_{21})^\top\end{array}\right)$ represents $W^{KQ}$.

It is evident to see that in our setup, the variable $(U,\bz, Z,v)$ represents $(U,u_{12}, u_{21}, u_{-1})$ in~\eqref{eqn:prediction_alt}. Therefore,
\begin{lem}
\label{lem:predictor_I_tendor}
\begin{align}
{\hat y}_{query}=&\frac{1}{N}\sum_{n=1}^N\sum_{i=1}^d \sum_{j=1}^d\sum_{k=1}^d(z_i  U_{jk}+z_iZ_k w_j+vw_iU_{jk}+vZ_kw_iw_j)x^{(i)}_n x_n^{(j)}x^{(k)}_q \nonumber \\ &+\frac{1}{N}\sum_{i=1}^d\sum_{j=1}^d \sum_{k=1}^d (z_iU_{jk}) x_q^{(j)}x_q^{(i)}  x_q^{(k)}.\label{eqn:predictor_I_tendor}
\end{align}
\end{lem}
\begin{proof}
\begin{align*}
&((u_{12})^\top, u_{-1})(EE^\top)
\\=& ((\bz)^\top, u_{-1})(EE^\top) 
\\=&\left(\sum_{n=1}^{N,q}\sum_{i=1}^d z_i x_n^{(1)}x_n^{(i)}+v\sum_{n=1}^Ny_n x_n^{(1)},  \ldots, \sum_{n=1}^{N,q}\sum_{i=1}^d z_i x_n^{(d)}x_n^{(i)}+v\sum_{n=1}^Ny_n x_n^{(d)},\right.
\\ & \quad
\left.\sum_{n=1}^N\sum_{i=1}^d z_i y_n x_n^{(i)}+v\sum_{n=1}^Ny_n y_n\right)
\\=&\underbrace{\left(\sum_{n=1}^{N,q}\sum_{i=1}^d z_i x_n^{(1)}x_n^{(i)},  \ldots, \sum_{n=1}^{N,q}\sum_{i=1}^d z_i x_n^{(d)}x_n^{(i)},\sum_{n=1}^N\sum_{i=1}^d z_i y_n x_n^{(i)}\right)}_{I}
\\ & + \underbrace{v\left(\sum_{n=1}^Ny_n x_n^{(1)},  \ldots, \sum_{n=1}^Ny_n x_n^{(d)},\sum_{n=1}^Ny_n y_n\right)}_{II}.
\end{align*}
Continue with the calculations, 
\begin{align*}
&I\left(\begin{array}{c} U  \\ Z^\top\end{array}\right) x_{q} 
\\=&\left(\sum_{n=1}^{N,q}\sum_{i=1}^d z_i x_n^{(1)}x_n^{(i)},  \ldots, \sum_{n=1}^{N,q}\sum_{i=1}^d z_i x_n^{(d)}x_n^{(i)},\sum_{n=1}^N\sum_{i=1}^d z_i y_n x_n^{(i)}\right)\left(\begin{array}{c} U  \\ Z^\top\end{array}\right) x_{q} 
\\ =& \sum_{n=1}^{N,q}\sum_{i=1}^d\sum_{j=1}^d \sum_{k=1}^d z_i x_n^{(j)}x_n^{(i)} U_{jk} x_q^{(k)}+ \sum_{n=1}^N\sum_{i=1}^d \sum_{k=1}^d z_iZ_k y_n x_n^{(i)}x^{(k)}_q.
\\ =& \sum_{n=1}^{N}\sum_{i=1}^d\sum_{j=1}^d \sum_{k=1}^d z_i x_n^{(j)}x_n^{(i)} U_{jk} x_q^{(k)}+ \sum_{n=1}^N\sum_{i=1}^d \sum_{k=1}^d z_iZ_k y_n x_n^{(i)}x^{(k)}_q+\sum_{i=1}^d\sum_{j=1}^d \sum_{k=1}^d z_i x_q^{(j)}x_q^{(i)} U_{jk} x_q^{(k)}.
\end{align*}
And
\begin{align*}
&II\left(\begin{array}{c} U  \\ Z^\top\end{array}\right) x_{q} 
\\=&v\left(\sum_{n=1}^Ny_n x_n^{(1)},  \ldots, \sum_{n=1}^Ny_n x_n^{(d)},\sum_{n=1}^Ny_n y_n\right)\left(\begin{array}{c} U  \\ Z^\top\end{array}\right) x_{q}
\\=&v\sum_{n=1}^N\sum_{j=1}^d \sum_{k=1}^dy_n x_n^{(j)}U_{jk} x_q^{(k)}+ v\sum_{n=1}^N\sum_{k=1}^dZ_ky_n y_nx^{(k)}_q.
\end{align*}
Therefore,
\begin{align}
&N{\hat y}_{query}\nonumber \\=& (I+II)\left(\begin{array}{c} U  \nonumber \\ Z^\top\end{array}\right) x_{q} 
\\=&\sum_{n=1}^N\sum_{i=1}^d\sum_{j=1}^d \sum_{k=1}^d z_i x_n^{(j)}x_n^{(i)} U_{jk} x_q^{(k)}+ \sum_{n=1}^N\sum_{i=1}^d \sum_{k=1}^d z_iZ_k y_n x_n^{(i)}x^{(k)}_q\nonumber \\&+v\sum_{n=1}^N\sum_{j=1}^d \sum_{k=1}^dy_n x_n^{(j)}U_{jk} x_q^{(k)}+ v\sum_{n=1}^N\sum_{k=1}^dZ_ky_n y_nx^{(k)}_q\nonumber \\&+\sum_{i=1}^d\sum_{j=1}^d \sum_{k=1}^d z_i x_q^{(j)}x_q^{(i)} U_{jk} x_q^{(k)}.\label{eqn:predictor_I}
\end{align}
Therefore, the desired expression in~\eqref{eqn:predictor_I_tendor} follows after we plug in the linear expression of the $y$ variable in~\eqref{eqn:predictor_I}.
\end{proof}

\begin{rem}
To see that~\eqref{eqn:prediction_alt} is $(1/2) \xi^\top H\xi$ intuitively, just consider the special case of $d=n=1$. Then,~\eqref{eqn:prediction_alt} can be written as,
\begin{align*}
&(u_{12}, u_{-1})(EE^\top)\left(\begin{array}{c} U_{11} \\ u_{21}\end{array}\right) x_{query}\\ =& (u_{12}, u_{-1})(x_{query}\otimes(EE^\top))\left(\begin{array}{c} U_{11} \\ u_{21}\end{array}\right) \\ =& \frac12(u_{12}, u_{-1}, U_{11}, u_{21})\left(\begin{array}{cc} 0& x_{query}\otimes(EE^\top) \\ x_{query}\otimes(EE^\top) &0\end{array}\right)\left(\begin{array}{c} u_{12}\\ u_{-1}\\U_{11} \\ u_{21}\end{array}\right)
\end{align*}
where $(EE^\top)$ is a $2\times 2$ matrix, and $x_{query}$ is a scalar.
\end{rem}

\begin{rem}
\label{rem:simplied}
Under the assumption that $u_{12}=u_{21}=0$, the predictor, and consequently the loss function and gradient flows, takes a simplified form. Note that this is the system studied in~\cite{zhang2023trained}, where convergence results are obtained for certain initial conditions and parameters.
\begin{align*}
((u_{12})^\top, u_{-1})(EE^\top)= & u_{-1}((EE^\top)_{d+1,1}, (EE^\top)_{d+1,2} \ldots, (EE^\top)_{d+1,d+1})\\=&u_{-1}\left(\sum_{n=1}^Ny_n x_n^{(1)}, \sum_{n=1}^Ny_n x_n^{(2)} \ldots, \sum_{n=1}^Ny_n x_n^{(d)},\sum_{n=1}^Ny_n y_n\right).
\end{align*}
Therefore, 
\begin{align}
\label{eqn:predictor_simplified}
{\hat y}_{query}=\frac{1}{N}\sum_{i=1}^{d}\sum_{j=1}^{d}\sum_{n=1}^{N} vU_{ij} x_q^{(i)}y_nx_n^{(j)},
\end{align}
with $v$ represents $u_{-1}$ and $U$ represents $U_{11}$. 
\end{rem}

\subsection{Gradient Calculations, through an Index Approach}
\label{sec:derivative_calculations}

The format of the loss function in~\eqref{eqn:lost_function}, in conjunction with the expression of the predictor~\eqref{eqn:predictor_I_tendor}, allow us to derive explicit expressions for the derivative of the loss function with respect to all the variables, namely the scalar $v$, the vectors $z,Z$, and the matrix $U$ which are presented in this section. 

For expression of the scalar $\frac{\partial L}{\partial v}$, straightforward calculations from equations~\eqref{eqn:lost_function} and ~\eqref{eqn:predictor_I_tendor} produce,
\begin{align*}
\frac{\partial L}{\partial v}= &\ex\left[\frac{1}{N}\sum_{n=1}^N\sum_{i=1}^d \sum_{j=1}^d\sum_{k=1}^d(z_i  U_{jk}+z_iZ_k w_j+vw_iU_{jk}+vZ_kw_iw_j)x^{(i)}_n x_n^{(j)}x^{(k)}_q \right.\\ &\left. +\frac{1}{N}\sum_{i=1}^d\sum_{j=1}^d \sum_{k=1}^d (z_iU_{jk}) x_q^{(j)}x_q^{(i)}  x_q^{(k)}-\sum_{m=1}^d w_mx_q^{(m)}\right]\\  &\times \left(\frac{1}{N}\sum_{n=1}^N\sum_{i=1}^d \sum_{j=1}^d\sum_{k=1}^d(w_iU_{jk}+Z_kw_iw_j)x^{(i)}_n x_n^{(j)}x^{(k)}_q\right). 
\end{align*}
Examine each term, we see that, 
\begin{lem}
\label{lem:dldv}
\begin{align}
\frac{\partial L}{\partial v} = &  \frac{1}{N}v\tr(U\Lam U^\top \Lam) \tr(\Lam) + \frac{N+1}{N} v \tr(U\Lam U^\top \Lam^2)+\frac{N+2}{N}(z^\top\Lam U\Lam Z)\tr(\Lam) \nonumber
\\ &+\frac{N+3}{N} z^\top\Lam^2 U\Lam Z\nonumber
 +\left[\frac{2(N+3)}{N}\tr(\Lam^2)
+ \tr^2(\Lam)\right]vZ^\top \Lam Z  \\ & +\frac{1}{N} [z^\top \Lam U^\top\Lam Z + z^\top \Lam Z\tr(U\Lam)]\tr(\Lam)
-\tr(\Lam U\Lam).\label{eqn:dldv}
\end{align}
\end{lem}
The proof can be found in Sec.\ref{sec:Dv}. The proofs of lemmas~\ref{lem:dldz},~\ref{lem:dldZ}, and~\ref{lem:dldU} utilize the same arguments, therefore, they are collected in the supplement. 
\begin{rem}
The terms $\frac{1}{N}v\tr(U\Lam U^\top \Lam) \tr(\Lam) + \frac{N+1}{N} v \tr(U\Lam U^\top \Lam^2)$ are the same cubic terms relate to $v$ and $U$ as those in the case of $z=Z=0$; The term $-\tr(\Lam U\Lam)$ is the linear term; all the other terms provide extra cubic terms for other combinations involving $z$ and $Z$.
\end{rem}

For the expression of the vector $\frac{\partial L}{\partial \bz}$, equations~\eqref{eqn:lost_function} and ~\eqref{eqn:predictor_I_tendor} give us, for any $\ell=1,2,\ldots, d$, 
\begin{align*}
\frac{\partial L}{\partial z_\ell}= &\ex\left[\frac{1}{N}\sum_{n=1}^N\sum_{i=1}^d \sum_{j=1}^d\sum_{k=1}^d(z_i  U_{jk}+z_iZ_k w_j+vw_iU_{jk}+vZ_kw_iw_j)x^{(i)}_n x_n^{(j)}x^{(k)}_q \right.\\ &\left. +\frac{1}{N}\sum_{i=1}^d\sum_{j=1}^d \sum_{k=1}^d (z_iU_{jk}) x_q^{(j)}x_q^{(i)}  x_q^{(k)}-\sum_{m=1}^d w_mx_q^{(m)}\right]\\  &\times \left(\frac{1}{N}\sum_{n=1}^N \sum_{j=1}^d\sum_{k=1}^d(U_{jk}+Z_kw_j)x^{(\ell)}_n x_n^{(j)}x^{(k)}_q +\frac{1}{N}\sum_{j=1}^d \sum_{k=1}^d U_{jk} x_q^{(j)}x_q^{(\ell)}  x_q^{(k)}\right).
\end{align*}
Therefore, 
\begin{lem}
\label{lem:dldz}
\begin{align}
\left(\frac{\partial L}{\partial z}\right)_\ell = & 
\frac{1}{N}(z^\top\Lam)_\ell\tr(U\Lam U^\top \Lam)+ \frac{N+2}{N} (z^\top\Lam U\Lam U^\top \Lam)_\ell \nonumber
\\ &+\frac{1}{N}(z^\top\Lam U \Lam)_\ell \tr(U\Lam)+ \frac{1}{N}(z^\top\Lam U \Lam U\Lam)_\ell \nonumber
\\&+\frac{1}{N} (z^\top \Lam)_\ell\tr(\Lam) (Z^\top\Lam Z)+\frac{N+1}{N}(z^\top \Lam^2)_\ell (Z^\top\Lam Z) \nonumber
\\&+\frac{2}{N} 
(vZ^\top\Lam U^\top \Lam)_\ell \tr(\Lam)+\frac{N+1}{N} (vZ^\top\Lam U^\top\Lam^2)_\ell \nonumber
\\ &+\frac{N+1}{N} ( vZ^\top\Lam U^\top\Lam^2)_\ell\nonumber
\\ &+
\frac{1}{N}(vZ^\top\Lam)_\ell \tr(\Lam) \tr (U\Lam)+\frac{2}{N}(vZ^\top\Lam U\Lam)_\ell \tr(\Lam)  \nonumber
\\&+\frac{1}{N}(z^\top \Lam U\Lam U^\top\Lam)_\ell+\frac{1}{N}(z^\top \Lam U^\top\Lam U^\top\Lam)_\ell+ \frac{1}{N}(z^T\Lam U^\top\Lam^\top)_\ell \tr(U\Lam) \nonumber
\\&+\frac{1}{N^2}[2(z^\top \Lam U\Lam)_\ell \tr(U\Lam) + 2(z^\top \Lam U \Lam U\Lam)_\ell+ 2(z^\top \Lam U \Lam U^\top \Lam)_\ell \nonumber
\\ 
&+ 2(z^\top \Lam U^\top \Lam)_\ell \tr(U\Lam) + 2(z^\top \Lam U^\top  \Lam U\Lam)_\ell+2(z^\top \Lam U^\top \Lam U^\top\Lam)_\ell \nonumber
\\ 
&+ (z^\top \Lam)_\ell \tr(U\Lam U\Lam ) +(z^\top \Lam)_\ell \tr(U\Lam U^\top \Lam ) +(z^\top \Lam)_\ell \tr(U\Lam)\tr(U\Lam)] \nonumber
\\ &-(Z^\top\Lam^2)_\ell.
\label{eqn:dldz}
\end{align}
\end{lem}

For the expression of the vector $\frac{\partial L}{\partial Z}$, equations~\eqref{eqn:lost_function} and ~\eqref{eqn:predictor_I_tendor} give us, for any $\ell=1,2,\ldots, d$, 
\begin{align*}
\frac{\partial L}{\partial Z_\ell}= &\ex\left[\frac{1}{N}\sum_{n=1}^N\sum_{i=1}^d \sum_{j=1}^d\sum_{k=1}^d(z_i  U_{jk}+z_iZ_k w_j+vw_iU_{jk}+vZ_kw_iw_j)x^{(i)}_n x_n^{(j)}x^{(k)}_q \right.
\\ &\left. +\frac{1}{N}\sum_{i=1}^d\sum_{j=1}^d \sum_{k=1}^d (z_iU_{jk}) x_q^{(j)}x_q^{(i)}  x_q^{(k)}-\sum_{m=1}^d w_mx_q^{(m)}\right]
\\  &\times \left(\frac{1}{N}\sum_{n=1}^N \sum_{i=1}^d\sum_{j=1}^d(z_iw_j+vw_iw_j)x^{(i)}_n x_n^{(j)}x^{(\ell)}_q \right). 
\end{align*}
Therefore, 
\begin{lem}
\label{lem:dldZ}
\begin{align}
\left(\frac{\partial L}{\partial Z}\right)_\ell = & 
\frac{2}{N} (vz^T\Lam^2U\Lam)_\ell+(vz^\top \Lam U\Lam)_\ell \tr(\Lam)\nonumber
    \\ &+ \frac{1}{N}(z^\top\Lam z Z^\top\Lam)_\ell\tr(\Lam)+ \frac{N+1}{N} (z^\top \Lam^2 z Z^\top\Lam)_\ell \nonumber
    \\& +\frac{2}{N}(vz^\top\Lam U\Lam)_\ell \tr(\Lam)+ \frac{N+1}{N}(vz^\top\Lam^2U\Lam)_\ell \nonumber
    \\ &+\frac{N+4}{N}(v^2 Z^\top\Lam)_\ell (\tr(\Lam))^2+2\frac{N+1}{N}(v^2Z^\top \Lam)_\ell \tr(\Lam^2) \nonumber
    \\&+ \frac{1}{N}(vz^\top \Lam U^\top \Lam)_\ell\tr(\Lam)+ \frac{1}{N}(vz^\top\Lam)_\ell\tr(U\Lam)\tr(\Lam)-(z^\top\Lam^2)_\ell.\label{eqn:dldZ}
\end{align}
\end{lem}

For the expression of the matrix $\frac{\partial L}{\partial \bz}$, equations~\eqref{eqn:lost_function} and ~\eqref{eqn:predictor_I_tendor} give us, for $\ell, p=1,2,\ldots d$, we have, 
\begin{align*}
\frac{\partial L}{\partial U_{\ell p}}= &\ex\left[\frac{1}{N}\sum_{n=1}^N\sum_{i=1}^d \sum_{j=1}^d\sum_{k=1}^d(z_i  U_{jk}+z_iZ_k w_j+vw_iU_{jk}+vZ_kw_iw_j)x^{(i)}_n x_n^{(j)}x^{(k)}_q \right.\\ &\left. +\frac{1}{N}\sum_{i=1}^d\sum_{j=1}^d \sum_{k=1}^d (z_iU_{jk}) x_q^{(j)}x_q^{(i)}  x_q^{(k)}-\sum_{m=1}^d w_mx_q^{(m)}\right]\\  &\times \frac{1}{N}\left(\sum_{n=1}^N \sum_{i=1}^d(z_i+vw_i)x^{(i)}_n x^{(\ell)}_n x_q^{(p)}+ \sum_{i=1}^dz_i x_q^{(i)}  x_q^{(\ell)}x_q^{(p)}\right).
\end{align*}
Therefore, 
\begin{lem}
\label{lem:dldU}
\begin{align}
\frac{\partial L}{\partial U_{\ell p}}= &\frac{1}{N}(z^\top \Lam z)(\Lam U\Lam)_{\ell p}+ \frac{N+3}{N}(z^\top\Lam)_\ell (z^\top \Lam U\Lam)_p \nonumber
\\
 &+\frac{1}{N}(z^\top\Lam U\Lam z)(\Lam)_{\ell P}+\frac{1}{N}(z^\top\Lam U\Lam)_\ell (z^\top \Lam)_p  \nonumber
 \\ &+
 \frac{N+1}{N}v\tr(\Lam) (z^\top \Lam)_\ell (Z^\top\Lam)_p+ \frac{N+3}{N}v(z^\top \Lam^2)_\ell (Z^\top \Lam)_p \nonumber
 \\&+
 \frac{1}{N} v^2(\Lam U\Lam)_{\ell p}\tr(\Lam)+\frac{N+1}{N} v^2 (\Lam ^2U\Lam)_{\ell p} \nonumber
 \\ &+
 \frac{1}{N}v [(z^\top\Lam Z) \Lam_{\ell p} + 2(Z^\top \Lam )_\ell(z^\top \Lam)_p]\tr(\Lam) \nonumber
\\&+ \frac{1}{N}(z^\top\Lam)_\ell(z^\top\Lam U^\top\Lam)_p+ \frac{1}{N}(z^\top\Lam)_\ell(z^\top\Lam)_p\tr(U\Lam) \nonumber
\\&+\frac{1}{N^2}[(z^\top \Lam U\Lam z)\Lam_{\ell p}+ 2(z^\top \Lam U\Lam)_\ell (z^\top \Lam)_p+ 2(z^\top \Lam)_\ell(z^\top \Lam U\Lam)_p \nonumber
\\&+ (z^\top \Lam U^\top \Lam z) \Lam_{\ell p} + 2(z^\top \Lam)_\ell(z^\top \Lam U^\top\Lam)_p \nonumber
\\&+ (z^\top \Lam z) \tr(U\Lam) \Lam_{\ell p}+  (z^\top \Lam z)(\Lam U^\top \Lam)_{\ell p} + (z^\top \Lam z)(\Lam U \Lam)_{\ell p}\nonumber
\\ &+ 2(z^\top \Lam U^\top \Lam)_\ell (z^\top \Lam)_p + 2(z^\top \Lam)_\ell \tr(U \Lam) (z^\top \Lam)_p] \nonumber
\\ &-v(\Lam^2)_{\ell p}.\label{eqn:dldU}
\end{align}
\end{lem}

\subsection{Complete System of Differential Equations}

Combining the above derivations, we have
\begin{thm}
\label{thm:differentialEquations}
The gradient flow for the problem of minimizing the loss function given in~\eqref{eqn:lost_function} is described by a system of differential equations, $({\dot U}, {\dot v} , {\dot Z}, {\dot z}) =-(\frac{\partial L}{\partial U}, \frac{\partial L}{\partial v}, \frac{\partial L}{\partial Z}, \frac{\partial L}{\partial \bz})$ with the expressions given in Lemmas~\ref{lem:dldv},~\ref{lem:dldz},~\ref{lem:dldZ}, and~\ref{lem:dldU}.
\end{thm}

\section{Analysis of the Simplified System {$Z=z=0$}}
\label{sec:analysis_uU}

In this section, we analyze the special case of the gradient flow obtained in Theorem~\ref{thm:differentialEquations}, where the variables $z$ and $Z$ are assumed to be zero vectors, thus termed the simplified system. This system corresponds to the gradient flow of the training dynamics considered in~\cite{zhang2023trained}, where a set of initial states and parameters is identified which are sufficient for exponential convergence to critical points. Based on the analysis provided in this section, we discover that
the invariant manifold takes different shapes depending on the initial values, and the behavior of the gradient flow, which includes the critical points and their stability patterns, is quite different. These differences, which we characterize qualitatively and quantitatively, have not been explicitly demonstrated in previous research. 

The simplified system is defined as the following system of $d^2+1$ ordinary differential equations with initial values $v_0$ and $U_0$:
\begin{align}
    \dot{U}&=v\Lam(-v\Gamma U+\Id)\Lam \label{eqn: dot U}\\
    \dot{v}&=\tr[\Lam(-v\Gamma U+\Id)\Lam U^\top], \label{eqn: dot u}
\end{align}
where $\Gamma=(1+\frac{1}{N})\Lam + \frac{1}{N}\tr[\Lam]\Id$. Here, $\tr$ denotes the trace operator and $\Id$ is the identity matrix. Note that $\Gamma$ is also positive definite, symmetric, and commutes with $\Lam$, that is, $\Lam\Gamma=\Gamma\Lam$.

\subsection{Invariant Manifold}
Each trajectory satisfies the following equation.
\begin{lem}
For any $t\ge 0$, 
\begin{align}
v^2(t)&=\tr[UU^\top](t)+[v_0^2-\tr[U_0U_0^\top]].\label{eqn: codim  1 foliation}
\end{align}
\end{lem}
\begin{proof}
Taking the transposition of~\eqref{eqn: dot U} and using basic properties of the trace operator
and the symmetry of $\Lam$ and $\Gamma$, 
we get:
\begin{align}    (\dot{U})^\top=\dot{(U^\top)}&=v\Lam(-v U^\top\Gamma+\Id)\Lam, \label{eqn: dot UT}\\
    \dot{v}&=\tr[U \Lam (-v U^\top\Gamma+\Id)\Lam]. \label{eqn: dot uT}
\end{align}
Multiplying~\eqref{eqn: dot U} on the right by $U^\top$, \eqref{eqn: dot u} by $v$ and~\eqref{eqn: dot UT} on the left by $U$, \eqref{eqn: dot uT} by $v$, we have:
\begin{align*}
    \dot{U}U^\top&=v\Lam(-v\Gamma U+\Id)\Lam U^\top, \quad
    v\dot{v}=\tr[v\Lam(-v\Gamma U+\Id)\Lam U^\top]=\tr[\dot{U}U^\top],
    \\
    U\dot{U^T}&=vU\Lam(-v U^\top\Gamma U+\Id)\Lam,
    \quad 
    v\dot{v}=\tr[vU \Lam (-v U^\top\Gamma+\Id)\Lam]=\tr[U\dot{U^\top}]. 
\end{align*}
Hence, 
\begin{align}
\frac{\partial v^2}{\partial t}&=\tr[\dot{U}U^\top+U\dot{U^\top]}=
\tr\left[\frac{\partial UU^\top}{\partial t}\right]
=\frac{\partial\tr[UU^\top]}{\partial t}.\nonumber
\end{align}
The equation~\eqref{eqn: codim  1 foliation} thus follows. 
\end{proof}

\begin{rem}
Equation~\eqref{eqn: codim  1 foliation} produces a codimension one foliation with invariant leaves, on which all trajectories live. 
As we shall see later, the curve (one  dimensional manifold) defined by
$v\Gamma U=\Id$ 
consists of critical points of~\eqref{eqn: dot U} and ~\eqref{eqn: dot u}. Here, the matrix $U$ is symmetric because $\Gamma$ is so.
It is easy to verify that, together with~\eqref{eqn: codim  1 foliation}, $v\Gamma U=\Id$ defines a family of critical points, at most two on each leaf of the foliation, possible attractors of the trajectories on the foliation.
\end{rem}

\subsection{Diagonalization of the Equations and Boundedness}

To further facilitate our analysis, we will consider a diagonalized representation of equations~\eqref{eqn: dot U} and~\eqref{eqn: dot V}, taking advantage of the diagonizability of the covariance matrix $\Lam$ and $\Gamma$. 

Denote $E$ as the orthonormal (unitary) matrix $E^\top E=\Id$ of eigenvectors of $\Lam$, $\Lam E= D_\Lam E$ or $E^\top \Lam E=D_\Lam$, where $D_\Lam$ is the diagonal matrix of eigenvalues of $\Lam$.  Also, due to commutativity $E^\top \Gamma E= D_\Gamma$, where $D_\Gamma=(1+\frac{1}{N})D_\Lam+\frac{1}{N}\tr[D_\Lam]\Id$ is the diagonal matrix of eigenvalues of $\Gamma$. Define $V=E^\top U E$, immediately, we can see that, 
\begin{align}
    \dot{V}&=uD_\Lam(-v D_\Gamma V +\Id)D_\Lam,
    \label{eqn: dot V} \\
    \dot{v}&=\tr[D_\Lam(-v D_\Gamma V+\Id)D_\Lam V^\top]. 
    \label{eqn: dot v}
 \end{align}
\begin{rem}
If $V(0)$ is a diagonal matrix, then from~\eqref{eqn: dot V} we get a diagonal solution $V(t)$ and in this case the system of~\eqref{eqn: dot V} and~\eqref{eqn: dot v} is further reduced to be of $d+1$ dimensional.
\end{rem}
Define $Z:=v D_\Gamma V$, then $Z^\top=v V^\top D_\Gamma$ and $\dot{Z}=\dot{v}D_\Gamma V + v D_\Gamma \dot{V}$. 
Therefore,
\begin{align*}
\sum_{ij}Z_{ij}^2=\tr[ZZ^\top]&=v^2\tr[D_\Gamma VV^\top D_\Gamma]=v^2\sum_{ij}\gamma_iV_{ij}^2\gamma_i,
\\
 v^2\sum_{ij}V_{ij}^2=   v^2\tr[VV^\top]&=\tr[D_\Gamma^{-1}ZZ^\top D_\Gamma^{-1}]=
 \sum_{ij}\gamma_i^{-1}Z_{ij}^2\gamma_j^{-1}.
\end{align*}
\begin{pro}
The variable $v^2$ is bounded iff every $V_{ij}$ is bounded iff every $Z_{ij}$ is bounded.
\end{pro}
\begin{proof}
It follows from $v^2=\tr[VV^\top]+{\rm  
const}=\sum_{ij}V_{ij}^2+{\rm  const}$ and from the fact that $\Lambda$ is nonnegative definite, so all $\lambda_i\ge 0$ and $\gamma_i=(1+\frac{1}{N})\lambda_i+\frac{1}{N}\sum_i\lambda_i\ge \frac{1}{N}\sum_i\lambda_i$ are separated from 0 unless all $\lambda_i=0$ and from $\infty$, since all $\lambda_i$ are finite.
\end{proof}

Write that the matrices $D_\Lam={\diag\,}(\lambda_i), D_\Gamma={\diag\,}(\gamma_i)$, $\lambda_i, \gamma_i>0$  are diagonal and positive definite, that is, the diagonals consist of positive entries. Note that  $Z_{ij}=v\gamma_i V_{ij}$, and 
\begin{align*}
    \dot{V}_{ij}&=v\lambda_i(-v \gamma_i V_{ij} +\Id_{ij})\lambda_j
    =\left\{
    \begin{array}{ll}
    -v\lambda_i\lambda_j Z_{ij}&j\not=i\\
    v\lambda_i^2(1-Z_{ii})& j=i    
    \end{array}
    \right. 
    \\
    \dot{v}&=\sum_l\sum_k \lambda_l(-v \gamma_l V_{lk}+\Id_{lk})\lambda_k V_{lk}\\
    &=\sum_{l}\sum_{k\not=l}\lambda_l(-v \gamma_l V_{lk})\lambda_k V_{lk}
    +\sum_{l}\lambda_l(1-v\gamma_l V_{ll})\lambda_l V_{ll}\\
    &=\sum_{l}\sum_{k\not=l}\lambda_l\lambda_k(-Z_{lk})\frac{1}{v\gamma_l} Z_{lk}
    +\sum_{l}\lambda_l^2(1-Z_{ll})\frac{1}{v\gamma_l}Z_{ll}
 \end{align*}
 Therefore, the gradient flow can be presented as the follow,
 \begin{align}
 \label{eqn:gf_coordinate-wise}
 \left\{\begin{array}{ll}
    \dot{(V^2_{ij})}=-2\lambda_i\lambda_j\gamma_i^{-1}\cdot Z_{ij}^2< 0 \text{ unless }Z_{ij}=0    & j\not=i 
    \\
    \dot{(V^2_{ii})}=2\lambda_i^2\gamma_i^{-1}(1-Z_{ii})Z_{ii}=
    2\lambda_i^2\gamma_i^{-1}
    \left(\frac{1}{4}-(Z_{ii}-\frac{1}{2})^2\right) &i=j
\\
\dot{(v^2)}=
    -2\sum_{l}\sum_{k\not=l}\lambda_l\lambda_k\gamma_l^{-1} Z_{lk}^2
    +2\sum_{l}\lambda_l^2\gamma_l^{-1}\left(\frac{1}{4}-(Z_{ll}-\frac{1}{2})^2)\right). &
\end{array}\right.
\end{align}
This allows us to obtain characterization of its asymptotic behavior.
\begin{pro}
For any $i\not=j$, $Z_{ij}(t)\to 0$, as $t\to\infty$, thus, $v(t)V_{ij}(t)\to 0$. In particular, if $v\not\to 0$, $V_{ij}(t)\to 0$ For all $i\not=j$. 
\end{pro}
\begin{proof}
As $\lambda, \gamma>0$,  $V_{ij}^2(t)$ is positive and non-increasing, hence it has a finite limit. 
Therefore, $\dot{V^2_{ij}}\ra 0$, 
which means $Z_{ij}^2\to 0$ and $v(t)V(t)_{ij}\to 0$.
\end{proof}
\begin{pro} On the trajectories, $V^2_{ii}(t)$, $i=1,2,\ldots, d$ and $v^2(t)$ are bounded. The bounds depend on initial conditions, but at the limit we have  $V_{ii}^2\le |\kappa|+\max(|\kappa|^{-1}, \gamma_i^2+\gamma_i^{-2})$, 
$v^2\le (d+1)|\kappa|+d|\kappa|^{-1}+\sum_i(\gamma_i^2+\gamma_i^{-2})$.
\end{pro}
 \begin{proof}
If $V^2_{ii}\le |\kappa|+|\kappa|^{-1}$, we are done. 
If $v^2\le  \gamma_i^2+\gamma_i^{-2}$, as $v^2=\sum_j V_{ii}^2+\kappa$, we have
$V_{ii}^2\le v^2+|\kappa|\le \gamma_i^2+\gamma_i^{-2}+|\kappa|$.
Otherwise, we have  $V_{ii}^2>|\kappa|+|\kappa|^{-1}$ and $v^2>\gamma_i^2+\gamma_i^{-2}$,  then 
$Z_{ii}^2=v^2\gamma_i^2V_{ii}^2>(\gamma_i^4+1)( |\kappa|+|\kappa|^{-1})>2$. Therefore, $\frac{\partial V^2_{ii}}{\partial t}< 0$, and $V^2_{ii}$ is decreasing until it reaches either boundary conditions. 
\end{proof}

\subsection{Critical Points}

\subsubsection{Identify the Critical Points}

Denote the right-hand side of the ordinary differential equations by $F(V,v)=(F_{ij}(V,v), f(V,v))$, thus $\dot{V}=F$, $\dot{v}=f$, and coordinate-wise:
\begin{equation*}
\begin{array}{rlll}
    \dot{V}_{ii}&=F_{ii}(V_{ii},v)&=v\lambda_i^2(1-v\gamma_i V_{ii})\\
    \dot{V}_{ij}&=F_{ij}(V_{ij},v)&=-v^2\lambda_i\lambda_j\gamma_iV_{ij}&i\not= j\\
    \dot{v}&=f(V,v)&=-v\sum_{i\not=j}\lambda_i\lambda_j\gamma_i V_{ij}^2 +\sum_i\lambda_i^2 V_{ii}-v\sum_i\lambda_i^2 \gamma_i V_{ii}^2.
\end{array}
\end{equation*}
Denote $(W,w)$ as the critical points, satisfying $(F(W,w),f(W,w))=(0,0)$.

\vskip 0.5cm
\noindent
\textbf{Case $v\not=0$, the points $(A,a)$}\\
For any $i\not=j$, $F_{ij}=0$ holds only when $A_{ij}=0$. $F_{ii}=0$ when $a\gamma_i A_{ii}=1$ which implies $f=0$ as well. This, together with the invariance condition,  defines a set of (two) points on any invariant leaf, symmetric by the sign change. We note that all $A_{ii}$ have the same sign as $a$. Moreover, with $\mu:=\sum_i\gamma_i^{-2}$, we have, 
\begin{align*}
\left\{\begin{array}{ll}    a=\pm\sqrt{\frac{1}{2}\left(\kappa+\sqrt{\kappa^2+4\mu}\right)}, &
\\
A_{ii}=\frac{1}{\gamma_ia}, & i=1,2,\ldots d
\\
A_{ij}=0, & i\not=j.
\end{array} \right.
\end{align*}

\vskip 0.5cm
\noindent
\textbf{Case $v=0$, the points $(B,b)$}\\
In this case $b=w=0$ and $F(B,0)=0$.  We then have $f(B,0)=0$ iff $\sum_{i}\lambda_i^2B_{ii}=0$. We remark that this defines a codimension two linear subspace of the $d^2+1$ dimensional state space or, in the case when we consider only the diagonal subsystem of $V_{ij}=0$ for $i\not=j$ a codimension two subspace of the $d+1$ dimensional state space $(V_{ii},v)$. In fact, by the invariance of the hyperboloid $v^2=\sum V_{ij}^2+\kappa$ ($\kappa$ being a constant), this defines a co-dimension one submanifold respectively of $d^2$ and $d$ dimensional manifolds. 

\begin{pro}
On any invariant leaf, all critical points lie on the same linear subspace.
\end{pro}
\begin{proof}
Define the codimension one linear subspace by the equation
$\alpha v -\beta \sum_i \lambda_i^2 V_{ii}=0$.
Clearly, the critical points $(B,b)$ (the case $v=0$ all lie on this subspace (regardless of the values of $V_{i\not=j}$. We want to determine $\alpha$ and $\beta$ so that the points $(A,a)$ (case $v\not=0$) lie on this hyperplane.
Define $\nu:=\sum_i\frac{\lambda_i^2}{\gamma_i}$. Then 
\begin{align*}
\alpha a -\beta \sum_i\lambda_i^2 A_{ii}&=0,
\end{align*}
leads to 
\begin{align*}
\beta&= \frac{\alpha}{2\nu}(\kappa+\sqrt{\kappa^2+4\mu}),
\end{align*}
which provides the required equation. 
\end{proof}

\subsubsection{Linearization at the Critical Points}

The linearization provides the character of the critical points. 
At the point $(W,w)$ it is given by the variables
$(\Delta, \delta)=(V,v)-(W,w)$ and the equations:
\begin{equation*}
    \left(
    \begin{array}{c}
         \dot{\Delta} \\
          \dot{\delta}
    \end{array}
    \right)=
    \frac{\partial(F,f)}{\partial(V,v)}_{\left|V=W, v=w\right.}\cdot\left(
    \begin{array}{c}
         \Delta\\
         \delta
    \end{array}
    \right).
\end{equation*}
\begin{pro}\label{pro:linearization} At critical points, the linearization takes the following form, 
\begin{itemize}
\item{\bf Case $(W,w)=(A,a)$}, then $v=w=b\not=0$ and
$A_{i\not=j}=0$ and $a\gamma_i A_{ii}=1$.
We have
\begin{align*}
    \dot{\Delta}_{ii}&=-\lambda_i^2a^2\gamma_i\Delta_{ii}-\lambda_i^2\delta\\
    \dot{\Delta}_{ij}&=-a^2\lambda_i\lambda_j \gamma_i\Delta_{ij}&i\not=j\\
    \dot{\delta}&=-\sum_i(\lambda_i^2\Delta_{ii}+\lambda_i^2\gamma_iA_{ii}^2\delta)
\end{align*}
\item
{\bf Case $(W,w)=(B,b)$}, then $v=w=b=0$ and 
 $\sum_i\lambda_i^2B_{ii}=0$.
 We have
\begin{align*}
    \dot{\Delta}_{ii}&=\lambda_i^2\delta\\
    \dot{\Delta}_{ij}&=0 &i\not=j\\
    \dot{\delta}&=\sum_i\lambda_i^2 \Delta_{ii}-(\sum_{ij}\lambda_i\lambda_j\gamma_iB_{ij}^2)\delta
\end{align*}
\end{itemize}
\end{pro}
\begin{proof}
From the coordinate-wise expression~\eqref{eqn:gf_coordinate-wise}, we have 
\begin{align*}
    \frac{\partial F_{ij}}{\partial V_{ij}}&=-v^2\lambda_i\lambda_j\gamma_i&\quad\text{any}\quad i,j\\
    \frac{\partial F_{ij}}{\partial V_{kl}}&=0&(i,j)\not=(k,l)\\
    \frac{\partial F_{ii}}{\partial v}&=\lambda_i^2-2v\lambda_i^2\gamma_iV_{ii}\\
    \frac{\partial F_{ij}}{\partial v}&=-2v\lambda_i\lambda_j\gamma_iV_{ij}& i\not=j\\
    \frac{\partial f}{\partial V_{ii}}&=\lambda_i^2-2v\lambda_i^2\gamma_iV_{ii}\\
    \frac{\partial f}{\partial V_{ij}}&= -2v\lambda_i\lambda_j\gamma_iV_{ij}& i\not=j\\
    \frac{\partial f}{\partial v}&=-\sum_{i,j}\lambda_i\lambda_j\gamma_iV_{ij}^2 &\quad\text{any}\quad i,j\, .
\end{align*}
So that with a vertical vector: $U=(\vect_k(V_{kk}),v,\vect_{i\not=j}(V_{ij}))$, 
where $\vect_i(a_i)$ is a vertical vector of all $(a_i)$ entries we get a ${\partial F}/{\partial U}$ matrix of dimension $d^2+1\times d^2+1$ represented by the blocks of dimensions:
\begin{align*}
    \text{Dimensions of }\frac{\partial F}{\partial U}
    \text{  are  }
    \left(
    \begin{array}{ccc}
    d\times d&d\times 1& d\times d(d-1)\\
    1\times d &1\times 1 & 1\times d(d-1)\\
    d(d-1)\times d & d(d-1)\times 1& d(d-1)\times d(d-1)
    \end{array}    \right)
   \end{align*}

\begin{align*}
    \qquad \frac{\partial F}{\partial U}=&
    \left(
    \begin{array}{ccc}
     \frac{\partial \dot{V}_{kk}}{\partial V_{kk}}
     &\frac{\partial \dot{V}_{kk}}{\partial v}
     &\frac{\partial \dot{V}_{kk}}{\partial {V}_{ij}}\\
     \frac{\partial \dot{v}}{\partial V_{kk}}
     &\frac{\partial \dot{v}}{\partial v}
     &\frac{\partial \dot{v}}{\partial {V}_{ij}}\\
     \frac{\partial \dot{V}_{ij}}{\partial V_{kk}}\
     &\frac{\partial \dot{V}_{ij}}{\partial v}
     &\frac{\partial \dot{V}_{ij}}{\partial {V}_{ij}}     
    \end{array}    \right)\\
    =&
    \left(
    \begin{array}{lll}
    -v^2\diag_k(\lambda_k^2\gamma_k)
    &\vect_k(\lambda_k^2)-2vS_k
    & 0
    \\[0.1in]
    \vect_k(\lambda_k^2)^\top-2vS_k^\top
    &-\sum_{ij}\lambda_i\lambda_j V_{ij}^2
    &-2vS_{i\neq j}^\top
    \\[0.1in]
    0
    & -2vS_{i\neq j}
    &-v^2\diag_{i\not=j}(\lambda_i\lambda_j\gamma_i)
    \end{array}
    \right)    
\end{align*}
where $S_k=\vect_k(\lambda_k^2\gamma_k V_{kk})$ and $S_{i\neq j}:= \vect_{i\not=j}(\lambda_i\lambda_j\gamma_iV_{ij})$.
\end{proof}

\subsubsection{Calculation of Eigenvalues}

In order to establish the character of a critical point, we investigate the eigenvalues of the linearization matrix. 
Let $(\Delta,\delta)$ be represented by the following $d^2+1$ vector $(\Delta_{ii}, \delta, \Delta_{i\not=j})$.
The $(d^2+1)\times (d^2+1)$ linearization  matrix is composed by the following blocks: a diagonal $d\times d$ matrix ${\diag}(t_i)$, a $d\times 1$ column $b_i$, a
$d \times (d^d-d)$ zero matrix, a $1\times 1$ scalar  $t_0$, 
a $1\times (d^2-d)$ row $e_{ij}$, a $(d^2-d)\times d$ zero matrix, a $(d^2-d)\times 1$ column $g_{ij}$, a diagonal $(d^2-d)\times (d^2-d)$ matrix ${\diag}(h_{ij})$:
\begin{equation*}
    \left(
    \begin{array}{ccc}
    {\diag}(t_i) & b_i & 0\\
    c_i& t_0 & e_{ij}  \\
    0 & g_{ij} &{\diag} (h_{ij})
    \end{array}
    \right)
\end{equation*}
where for $i=1\dots d$, $t_i=\partial F_{ii}/\partial V_{ii}$, $b_i=\partial F_{ii}/\partial v$, $c_i=\partial f/\partial V_{ii}$ and $t_0=\partial f/\partial v$. Additionally for $i\not=j$: 
$e_{ij}=\partial f/\partial V_{ij}$, $g_{ij}=\partial F_{ij}/\partial v$ and $h_{ij}=\partial F_{ij}/\partial V_{ij}$.

\vskip 0.5cm
\noindent
\textbf{Case $(W,w)=(A,a)$}

In this case, we have 
$\diag(t_i)=\diag( -a^2\lambda_i^2\gamma_i)$, 
$b_i=c_i=-\lambda_i^2$, $t_0=-\sum_i\lambda_i^2\gamma_iA_{ii}^2=
-\frac{1}{a^2}\sum_i\frac{\lambda_i^2}{\gamma_i}$, 
$e_{ij}=g_{ij}=0$, $\diag(h_{ij})=\diag(-a^2\lambda_i\lambda_j\gamma_i)$.
The characteristic polynomial is 
\begin{align*}
  P(\mu)&=   \left|
    \begin{array}{ccc}
   {\diag_i}(\mu+a^2\lambda_i^2\gamma_i) & -\lambda_i^2 & 0\\
    -\lambda_i^2 &\mu+\frac{1}{a^2}\sum_i\frac{\lambda_i^2}{\gamma_i}  & 0  \\
    0 & 0 &{\diag_{ij}} (\mu+a^2\lambda_i\lambda_j\gamma_i)
    \end{array}
    \right|\\
    &=P_A(\mu)\cdot \prod_{i\not=j}(\mu+a^2\lambda_i\lambda_j\gamma_i),
    \end{align*}
    where
    \begin{align*}
    P_A(\mu)&= \left|
    \begin{array}{cc}
   {\diag_i}(\mu+a^2\lambda_i^2\gamma_i) & -\lambda_i^2 \\
    -\lambda_i^2 &\mu+\frac{1}{a^2}\sum_i\frac{\lambda_i^2}{\gamma_i} 
    \end{array}
    \right|.
\end{align*}

\vskip 0.5cm
\noindent
\textbf{Case $(W,w)=(B,0)$}

In this case, we have $\diag(t_i)=0$, $b_i=c_=\lambda_i^2$, $t_0=-\sum_{ij}\lambda_i\lambda_j\gamma_i B_{ij}^2$, 
$e_{ij}=g_{ij}=0$, $\diag(h_{ij})=0$.
And the characteristic polynomial is given by 
\begin{align*}
  P(\mu)&=   \left|
    \begin{array}{ccc}
   {\diag_i}(\mu) & \lambda_i^2 & 0\\
    \lambda_i^2 &\mu+\sum_{ij}\lambda_i\lambda_j\gamma_i B_{ij}^2  & 0  \\
    0 & 0 &{\diag_{ij}} (\mu)
    \end{array}       
    \right|=P_B(\mu)\cdot \mu^{d^2-d},
    \end{align*}
    where
    \begin{align*}
P_B(\mu)&=\left|\begin{array}{cc}
   {\diag_i}(\mu) & \lambda_i^2 \\
    \lambda_i^2 &\mu+\sum_{ij}\lambda_i\lambda_j\gamma_i B_{ij}^2  
    \end{array}
    \right|.
\end{align*}

In both cases the determinants $P_B$ and $P_A$ have a similar structure
\begin{equation*}
R_d= \left|
    \begin{array}{cc}
   {\diag_i}(t_i) & b_i \\
    b_i &t_0 
    \end{array}
    \right|. 
\end{equation*}
Suppose that the index $i$ runs from $d$ down to 1. 
\begin{lem}\label{lem:det R} $R_d$ satisfies the recursion, $R_d= t_d R_{d-1}-b_d^2\prod_{k=1}^{d-1} t_k$. Therefore, 
\begin{align*}
R_d&=\prod_{k=1}^d t_k\cdot(t_0-\sum_{j=1}^d\frac{b_j^2}{t_j})\,.
\end{align*}
\end{lem}

\noindent
We follow the convention that the empty product is one and the empty sum is zero. 
\begin{proof}
    We expand the determinant with respect to the first column, which has only two nonzero entries: $t_d$ in the first row and $b_d$ in the $(d+1)$-th row, obtaining a combination of two smaller $d\times d$ determinants $t_d R_{d-1}$ and $(-1)^d b_d Q_{d}$. The first row of $Q_d$ has only one non-zero entry, $b_d$ in the $d$th column, expanding $Q_{d}$ with respect to this column we get $(-1)^{d-1}b_d\cdot|\diag(t_i)_{i=d-1\dots 1}|$. Combining them together, we obtain the recursive formula, and the expression follows immediately.
    
\end{proof}
\subsubsection{Stability Behavior at Critical Points}
\begin{thm}
\label{thm:critical_points}
The linearization matrix at critical point is symmetric, and thus all its eigenvalues are real. 
\begin{itemize}
\item
The critical points of type $(W,w)=(A,a)$ are attractors, the eigenvalues of the
linearization matrix are non-positive. There is an eigenvalue equal to zero which corresponds to the eigen-direction orthogonal to the invariant leaf. A higher multiplicity of the eigenvalue zero is only possible when some $\lambda_j=0$.
\item
The critical points of type $(W,w)=(B,0)$ are degenerate saddles. There are $d^2-d$ zero eigenvalues associated with the variables $\Delta_{i\not=j}$ defined in Proposition~\ref{pro:linearization} and $d-1$ eigenvalues $0$  from $P_B$ and the two remaining eigenvalues are real, non-zero and have different signs, as they product is negative. 
\end{itemize}
\end{thm}
\begin{proof}
The symmetry of the linearization matrix is obvious. Apply Lemma~\ref{lem:det R} to the two cases.

\vskip 0.5cm
\noindent
\textbf{Case $(W,w)=(A,a)$}.
There are $d^2-d$ negative eigenvalues $\mu=-a^2\lambda_j\lambda_j\gamma_i$ associated with variables $V_{i\not=j}$. Meanwhile, the degree $d+1$ factor of the characteristic polynomial associated with the variables $v$ and $V_{ii}$ takes the form, 
\begin{align*}
P_A(\mu)&=\prod_{k=1}^d(\mu+a^2\lambda_k^2\gamma_k)
\cdot\left(\mu+\frac{1}{a^2}\sum_{j=1}^d\frac{\lambda_j^2}{\gamma_j}
-\sum_{j=1}^d\frac{\lambda_j^4}{\mu+a^2\lambda_j^2\gamma_j}\right)\\
&=\prod_{k=1}^d(\mu+a^2\lambda_k^2\gamma_k)
\cdot\left(\mu+\sum_{j=1}^d
\frac{\mu\lambda_j^2}{a^2\gamma_j(\mu+a^2\lambda_j^2\gamma_j)}\right)
\end{align*}
By symmetry, all its zeros are real.
As $\gamma_i>0$,  for $\mu>0$ we have $P_A(\mu)>0$, so all its real zeros, or the eigenvalues of the linearization matrix are non-positive.  
We note that $\mu=0$ is one of the zeros of $P_A$ as well, and it corresponds to the eigenvalue $0$ with the eigen-direction orthogonal to the invariant surface $v^2=\tr (VV^\top)+\kappa$. A higher multiplicity of the eigenvalue $\mu=0$ is only possible when some $\lambda_j=0$.

\vskip 0.5cm
\noindent
\textbf{Case $(W,w)=(B,0)$}.
From the characteristic polynomial,
\begin{align*}
P_B(\mu)&=\mu^d\left(\mu+\sum_{ij}\lambda_i\lambda_j\gamma_i B_{ij}^2 -\sum_{j=1}^d\frac{\lambda_i^4}{\mu}\right)\\
&=\mu^{d-1}\left(\mu^2+\sum_{ij}\lambda_i\lambda_j\gamma_i B_{ij}^2\mu-\sum_{j=1}^d\lambda_i^4\right)\,,
\end{align*}
we conclude that the points $B$ are degenerate saddle points. 
There are $d^2-d$ zero eigenvalues associated with the variables $\Delta_{i\not=j}$ and $d-1$ eigenvalues $0$ from $P_B$ and the two remaining eigenvalues are real, non-zero and have different signs, as their product is negative. 
It is worth mentioning that the negative eigenvalue is stronger than the positive one, since their sum is negative as well. 
\end{proof}

\begin{rem}
\label{rem:global}
Because the system~\eqref{eqn: dot U} and~\eqref{eqn: dot u} is a gradient flow, almost all trajectories converge to attractors of type $(A,a)$ except the trajectories lying on the separatrices of saddle points of type $(B,0)$ which form a finite union of smooth manifolds of positive codimension, which implies that they form a null set in the sense of Lebesgue measure. 
\end{rem}

\begin{rem}
    \label{rem: simplified grad flow}
    The simplified system~\eqref{eqn: dot U} and~\eqref{eqn: dot u} can be viewed as a gradient flow of a simplified, more concise cost function $L$ presented below. 
    This function $L$ is different from both cost functions in the article~\cite{zhang2023trained}
where we have a very non-intuitive
$L=\ex\left(\mathcal {U^\top H U}-w x\right)^2$ and (c.f. A.11) $\tilde{\ell}=\tr[v\Lambda(v\Gamma V/2-\Id)\Lambda V^\top]$. All three expressions provide the same gradient flow.
\end{rem}
Let $\lambda_i, \gamma_i>0$, $i=1,\dots,d$ ($\gamma_i=\frac{1}{N}\lambda_i+(1+\frac{1}{N}\sum_i\lambda_i)$). Consider the following function $L:\Real^d\times \Real\times \Real^{d^2-d}\to \Real$, $(V_{ii},v,V_{ij})\mapsto L(V_{ii},v,V_{ij})$, where $i,j\in\{1,\dots,d\}, i\not=j$:
\begin{align*}
    L(V_{ii},v, V_{ij})&=\frac{1}{2}\left(
    \sum_{i\not=j}\frac{\lambda_i\lambda_j}{\gamma_i}(-v\gamma_i V_{ij})^2+
    \sum_i \frac{\lambda_i^2}{\gamma_i}(1-v\gamma_i V_{ii})^2
    \right)\\
    &=
    \frac{1}{2}\left(
    \sum_{i,j}\lambda_i\lambda_j\gamma_iv^2 V_{ij}^2+
    \sum_i \frac{\lambda_i^2}{\gamma_i}-2\sum_i\lambda_i^2\gamma_iv V_{ii}
    \right)\,,\\
    \frac{\partial L}{\partial V_{ii}}&=
    -v\lambda_i^2(1-v\gamma_i V_{ii})\,,\\
    \frac{\partial L}{\partial v}&=
    \sum_{i\not=j}v\lambda_i\lambda_j\gamma_i V_{ij}^2-
    \sum_i \lambda_i^2(1-v\gamma_i V_{ii})V_{ii}\\
    &=
    \sum_{i,j}v\lambda_i\lambda_j\gamma_i V_{ij}^2-
    \sum_i \lambda_i^2\gamma_iV_{ii}\,,
    \\
    \frac{\partial L}{\partial V_{ij}}&=v^2\lambda_i\lambda_j\gamma_i V_{ij} &i\not=j\\\
\end{align*}

\subsection{Invariant Sets and Their Dependence on $\kappa$}
The subspace $V_{ij}=0$, $i\not= j$ is invariant. There $\dot{V}_{ij}=0$ and the trajectories starting there cannot move outside this subspace. We observe that unless $v=0$ this subspace is attracting trajectories, as a result of $\lambda, \gamma>0$ the coordinates $V_{ij}$, $i\not=j$ exhibit a negative feedback behavior.

Each codimension-one manifold (rotational hyperboloid) $v^2=\sum_{ij} V_{ij}+\kappa$ is invariant for any constant $\kappa$, that is, if a trajectory starts on such manifold, with $\kappa$ determined by the initial conditions, it will not leave this manifold.
All such manifolds form an invariant \emph{lamination} of the space. The shape of invariant manifolds,  leaves of the lamination, depends on $\kappa$, as illustrated on Figure~\ref{fig:kappa shape}
\begin{figure}[h]
\center
\includegraphics[scale=0.156]{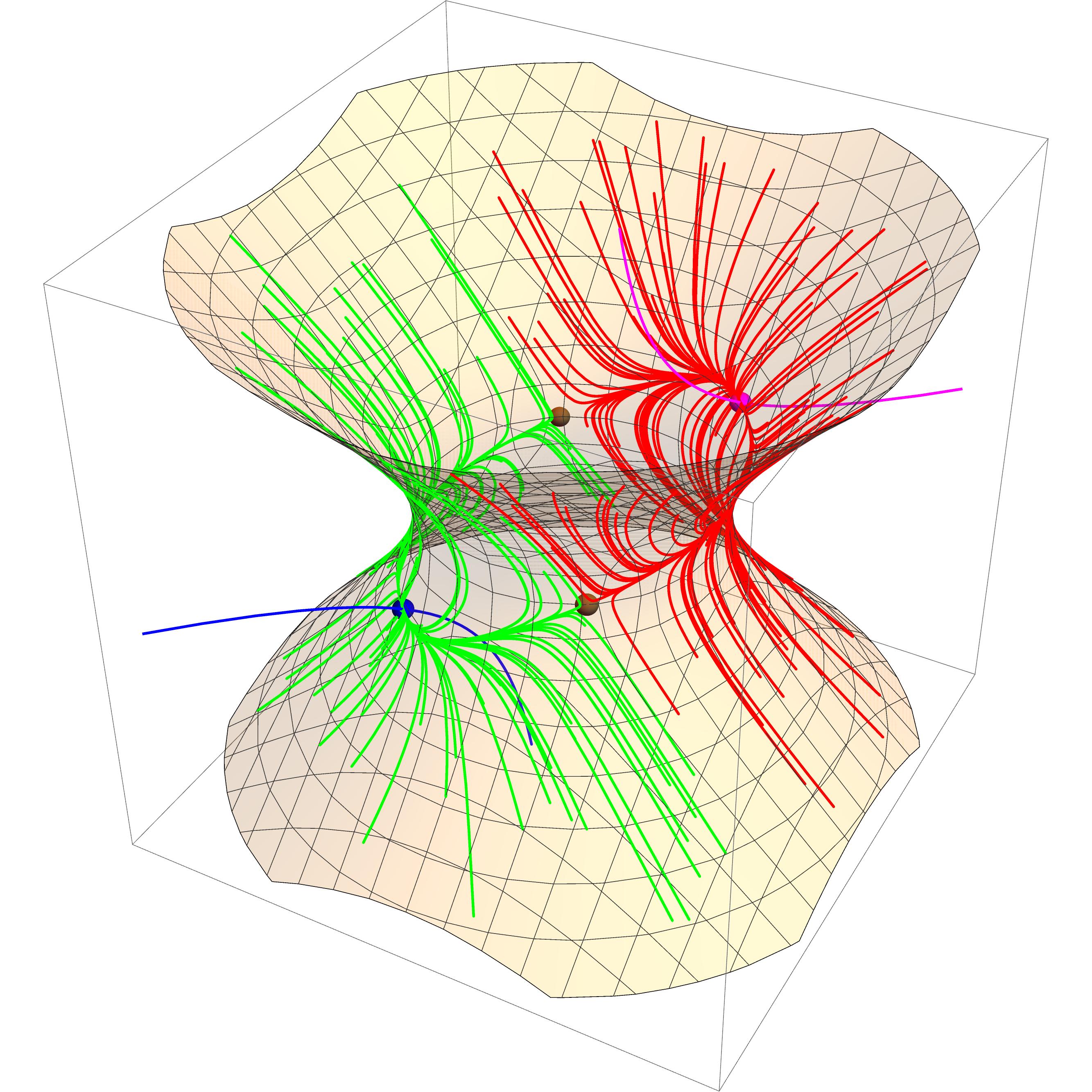}
\includegraphics[scale=0.156]{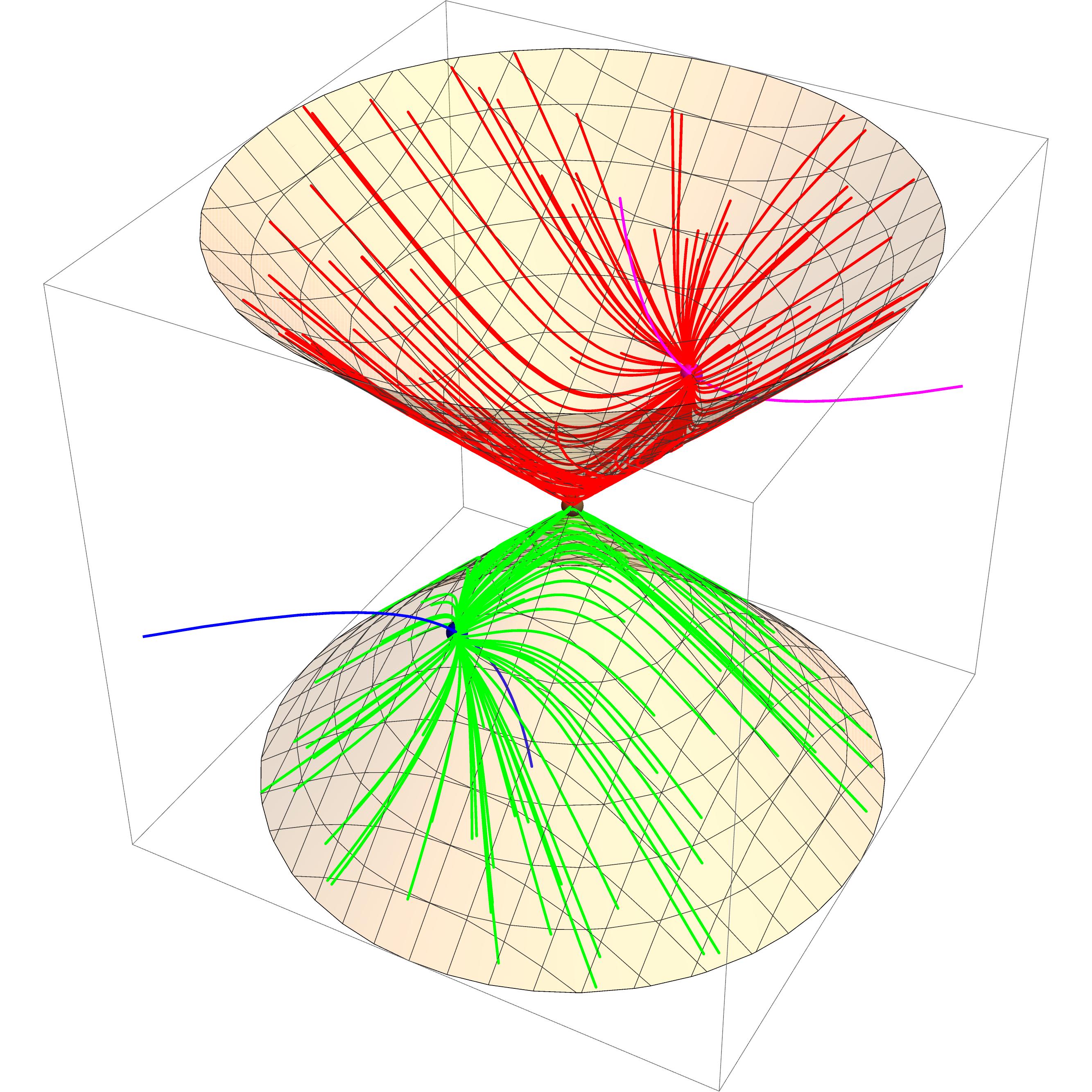}
\includegraphics[scale=0.156]{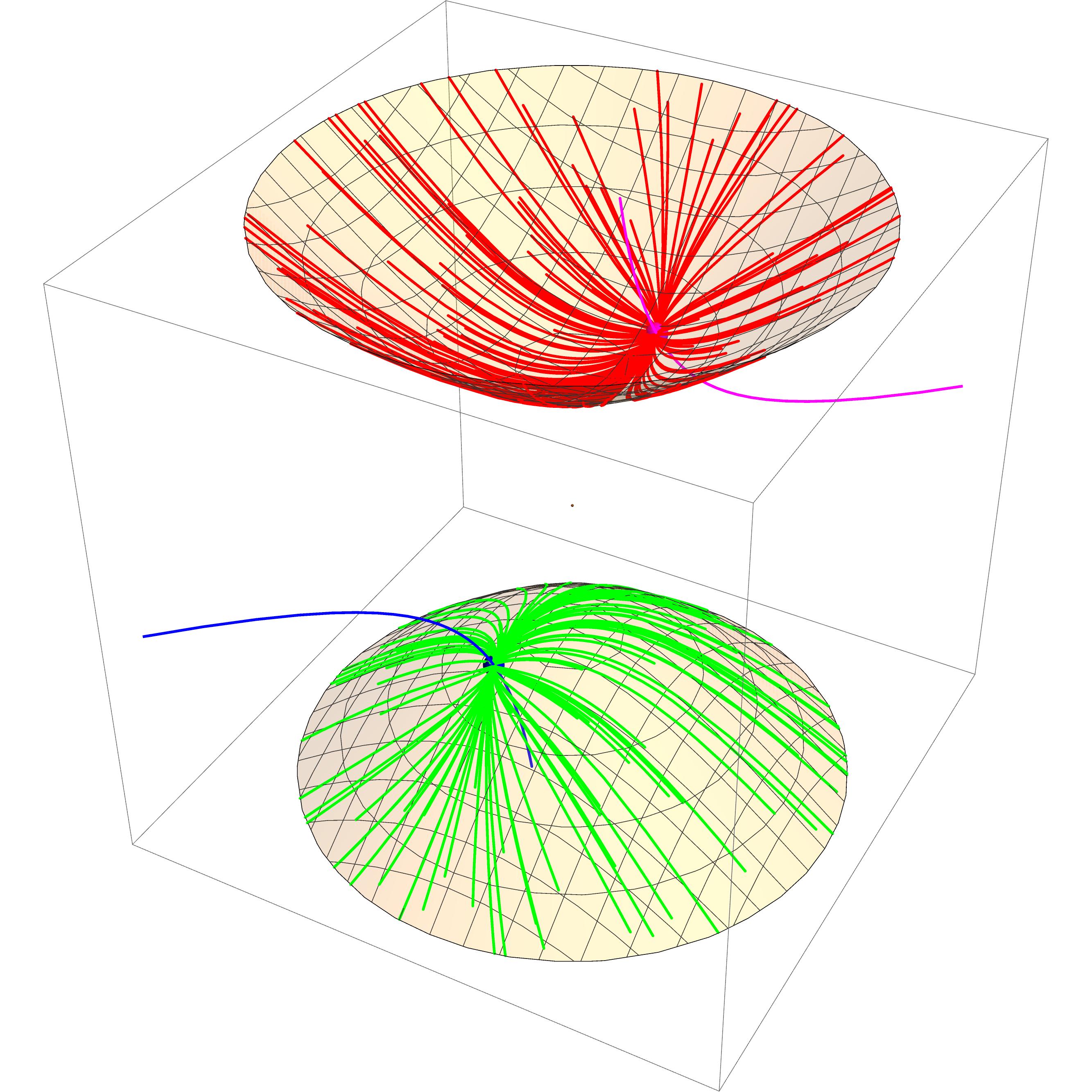}
    \caption{
    From left to right: $\kappa<0$, $\kappa=0$, $\kappa>0$.
    }
    \label{fig:kappa shape}
\end{figure}
\begin{itemize}
    \item[$\kappa<0$,] 
 the manifold is connected, all values of $v$ are allowed and for $d>1$ the spheres $\sum V_{ij}^2=v^2 -\kappa\ge -\kappa>0$ are connected.
\item[$\kappa=0$,]  
 the hyperboloid becomes a cone, with the unique point B at 0.
\item[$\kappa>0$,]  the manifold is disconnected, as in this case we have $v^2\ge \kappa>0$, that is, $v$ cannot lie in some interval around 0. We note that in this case the points B do not exist (to be more precise, the linear subspace B is disjoint from the invariant hyperboloid), so the only critical points are the two points A.
\end{itemize}
In all cases, the two components of the curve $v\Gamma U=\Id$ consisting of the points $A$ intersect each leaf of invariant lamination, defined by $\kappa$, exactly at two points, symmetric with respect to the point $0$. At $A$ we have $V_{ij}=0$, $i\not=j$ and $V_{ii}=1/v\gamma_i$ so the intersection points are given by
$v=v(A,\kappa)=\pm\sqrt{\frac{\kappa+\sqrt{\kappa^2+4\sum1/\gamma_i^2}}{2}}$. We note that $v(A,\kappa)$ is well defined for every $\kappa$.

For each $\kappa$, the critical points A and B lie on a ($\kappa$ dependent) codimension one linear hyperplane $\alpha\sum \lambda_i^2 V_{ii}+\beta v=0$, where $\alpha, \beta$ satisfy $\alpha\sum\lambda_i^2/\gamma_i+\beta v^2(A,\kappa)=0$.

\section{Analysis of the System of Equations When $d=1$}
\label{sec:analysis_1D}

When $d=1$, $v, U, z, Z$ are all scalars. We will also use the scalar $\g>0$ to replace $\Lam$. Therefore, omitting the constant term, the loss function takes the form
\[
L=-\lambda^2 (vU+zZ)+\frac{\lambda^3}{2}\left[ \alpha (v^2U^2+z^2Z^2)+\beta v^2Z^2+\delta  U^2z^2+2\gamma vUzZ\right]\,,
\]
where $\alpha =\frac{(N+2)}{N}>1$, $\beta=\frac{(3N+6)}{N}>3$, $\gamma=\frac{(2N+7)}{N}>2$ and $\delta=\frac{(N+8)}{N}+\frac{15}{N^2}>1$, but $2\gamma<\beta+\delta$. 
The gradient of $L$ equals
\begin{align}\label{equ:gradL}
\begin{array}{rcl}
\frac{\partial L}{\partial U} &=& -\g^2 v + \g^3(\alpha v^2U +\delta z^2U+ \gamma zZ),
\\
\frac{\partial L}{\partial z} &=&-\g^2 Z+ \g^3(\alpha zZ^2+\delta zU^2+\gamma vZU),
\\
\frac{\partial L}{\partial Z} &=& -\g^2 z  +\g^3(\alpha  z^2Z+ \beta v^2Z+ \gamma vzU)\,,
\\
\frac{\partial L}{\partial v} &=& -\g^2 U  +\g^3(\alpha  vU^2+ \beta vZ^2+ \gamma vzZU)\,.
\end{array}
\end{align}
Denote $T=(U,z,Z,v)$, the evolution equation is then given by $\frac{\partial T}{\partial t}=-\frac{\partial L}{\partial T}$. Without changing the dynamics we can rescale the variables  
 $S=(\tilde{U},\tilde{z}, \tilde{Z}, \tilde{v})=\sqrt{\lambda}(U,z,Z,v)$ and the function $\tilde{L}=\lambda^{-3/2} L$ (or rescaling the time $\lambda^{3/2} \partial t=\partial s$ with $\dot{S}=\frac{\partial S}{\partial s}$), thus obtain, skipping the tilde's, the following expression for the new function $L$ 
 \[
 L = - (Uv+zZ)+\frac{1}{2}( \alpha (v^2U^2+z^2Z^2)+\beta v^2Z^2+\delta z^2U^2)+\gamma vUzZ
 \]
 and its gradient flow:
\begin{align}
\begin{array}{rcl}\label{equ:gradflowL}
{\dot U} &=&  v - \alpha  v^2U - \delta z^2U-  \gamma vzZ,
\\
{\dot z} &=& Z- \alpha  zZ^2-\delta zU^2-\gamma vZU, 
\\
{\dot Z} &=&  z  -\alpha  z^2Z- \beta v^2Z- \gamma vzU,
\\
{\dot v} &=& U - \alpha  vU^2  - \beta vZ^2- \gamma zZU\,.
\end{array}
\end{align}
\begin{rem}
Writing $P=\alpha (vU+zZ-\frac{1}{\alpha })^2\ge 0$ and $R=\beta (vZ)^2+\delta (zU)^2+2(\gamma-\alpha)(vZ)(zU)\ge 0$ as $(\gamma-\alpha)^2<\beta\delta$,  we have  
\begin{align*}
2 L&=P-\frac{1}{\alpha }+R\ge -\frac{1}{\alpha }\,.
\end{align*}
Therefore, $L$ is bounded from below and its level sets provide bounds for the following aggregates of variables: $vU, zZ, zU$ and $vZ$.
\end{rem}

\subsection{Symmetry and Invariance}
There is duality in the equations in the following sense: \emph{whatever can be said about the pair $(v,U)$ is by duality true for the pair $(Z, z)$}.  In particular, if $S(s)=(U, z,Z, v)$ is a trajectory of~\eqref{equ:gradflowL} then so is $(z, U, v, Z)$. 

We have
\begin{align}
    \frac{\partial(v^2-U^2)}{\partial s}&=
    2(-\beta v^2Z^2+\delta z^2U^2)=
    \frac{\partial(Z^2-z^2)}{\partial s}
    \label{eqn:v2-U2}
    \\
    \frac{\partial(v^2-Z^2)}{\partial s}&=
    2(vU-zZ-\alpha(v^2U^2-z^2Z^2))=\frac{\partial(U^2-z^2)}{\partial s} ,
    \label{eqn:v2-Z2}    
\end{align}
therefore each trajectory is confined to one of the invariant co-dimension one manifolds $K_\kappa=\{S=(U,z,Z,v): -U^2+z^2-Z^2+v^2=\kappa\}$, the invariant lamination,  where the parameter $\kappa\in \Real$ is established by the initial condition of the trajectory. However, the duality changes the sign of $\kappa$. 

The manifold $K_\kappa$ is a union of one parameter family of tori:
$T_{\rho, \mu}=\{S:v^2+z^2=\mu,\  U^2+Z^2=\rho\}$, $\mu, \rho\ge 0$ and $\mu=\rho+\kappa$. For each $\kappa$, $K_\kappa$  is connected. Each torus is connected, and each two tori $T_{\rho, \mu}$, $T_{\sigma, \tau}$ are connected by the path of tori: $T_{(1-p)\rho+p\sigma, (1-p)\mu+p\tau}\subset K_\kappa$, $p\in[0,1]$. Exceptionally, when $\kappa=0$ one member of this family of tori is a point $(0,0,0,0)$, it happens when $\rho=0=\mu$.

\emph{Note that the connectedness argument does not work with three variables, as the 0-dimensional  torus is not connected.}

The condition $(Z,z)=(0,0)$ is invariant and yields a system for $(v,U)$, $\dot{U}=v-\alpha U v^2, \dot{v}=U-\alpha v U^2$, with a critical point $(0,0)$ and a critical curve $\alpha vU=1$, which contains two symmetric points on the invariant manifold $v^2-U^2=\kappa$.
By duality, the condition $(v,U)=(0,0)$ is invariant as well and yields an analogous system $\dot{z}=Z-\alpha z Z^2, \dot{Z}=z-\alpha z Z^2$ with analogous critical points.

The condition: both $v=Z$ and $z=U$, leading to $\kappa=0$, is invariant as it yields to the system
$\dot{Z}=U-\alpha U^2Z-\beta Z^3-\gamma U^2Z, 
\dot{U}=Z-\alpha U Z^2-\delta  U^3-\gamma UZ^2$ and the same equations hold for the variables $(z,v)$. 
Here again, the point $(0,0,0,0)$ is critical and the curve (in coordinates $(U,Z)$ or a two-dimensional surface in coordinates $S$) $1=(\alpha+\gamma+\sqrt{\beta\delta})UZ$ consists of critical points with the additional constraint $\beta Z^4=\delta U^4$. 

The system~\eqref{equ:gradflowL} and the parameter $\kappa$ remain the same when both variables $(v,U)$ change the sign; a similar statement holds for the change of signs for $(z,Z)$, and therefore for the change of signs of all four variables. That means that if $S(s)=(U,z,Z,v)$ is a trajectory on $K_\kappa$, then so are $(-U, z,Z, -v)$, $(U,-z,-Z,v)$, and $(-U,-z, -Z, -v)$.

\subsection{Critical (or Critical)  Points}
At critical points, where $\frac{\partial S}{\partial s}=(0,0,0,0)$, we have, by \eqref{eqn:v2-U2}, \eqref{eqn:v2-Z2} and the invariance of the manifolds $K_\kappa$

\begin{align}\label{equ:crit cond}
\begin{array}{rcl}
\beta v^2Z^2 &=& \delta z^2U^2
\\
vU- zZ  &=& \alpha (v^2 U^2-z^2 Z^2)=\alpha(vU-zZ)(vU+zZ)
\\
\kappa  &=&  v^2 -U^2 -Z^2+z^2
\\
0 &=&  U - \alpha  vU^2  - \beta vZ^2- \gamma zZU\, .
\end{array}
\end{align}
Where for the fourth equation, closing the system, may be  chosen from any equation of the original $\dot{S}=0$. The fourth equation 
\begin{thm}
        In case $\kappa=0$, the point $O=(0,0,0,0)$ is critical. For all $\kappa$'s the critical points can be of two types:
    \begin{enumerate}
        \item[$A$:] 
        Four points  with two variables equal zero, 
        \begin{itemize}
            \item 
        two points $(U,0,0,v)$ with $1=\alpha v U>0$ and 
        \\ 
        $U^2 = (\sqrt{\kappa^2+4/\alpha^2}-\kappa)/2$, 
        $v^2 = (\sqrt{\kappa^2+4/\alpha^2}+\kappa)/2$
        \item two points $(0,z,Z,0)$ with $1=\alpha z Z>0$ and 
        \\ 
        $Z^2 = (\sqrt{\kappa^2+4/\alpha^2}-\kappa)/2$, 
        $z^2 = (\sqrt{\kappa^2+4/\alpha^2}+\kappa)/2$\,.
        \end{itemize}
        \item[$B$:] Four points  with all variables non zero, satisfying $vU=zZ=\rho >0$, where  $\rho=(\alpha+\sqrt{\beta\delta}+\gamma)^{-1}$ 
        \begin{itemize}
        \item \  
        $\kappa=v^2\left(1+\sqrt{\frac{\beta}{\delta}}\right) - U^2\left(1+\sqrt{\frac{\delta}{\beta}}\right)$ \qquad and 
        \item \ 
        $\kappa=z^2\left(1+\sqrt{\frac{\delta}{\beta}}\right)-Z^2\left(1+\sqrt{\frac{\beta}{\delta}}\right)$.
        \end{itemize}
    \end{enumerate}
     \label{thm:1DCriticalPoints}
\end{thm}
\begin{rem}
We have $L(A)=-\frac{1}{2\alpha}$, which is minimal. Moreover $L(A)<L(B)=-\rho=-\frac{1}{\alpha+\sqrt{\beta \delta}+\gamma}<0=L(O)$
\end{rem}
\begin{proof}[Proof of Theorem~\ref{thm:1DCriticalPoints}] 
At critical points, the right-hand side of~\eqref{equ:gradflowL} equations equals zero. 

Clearly $O=(0,0,0,0)$ is a critical point, and $\kappa=0$. When three variables are equal to zero, so is the fourth one. 

When one variable is zero, then, by the first condition of \eqref{equ:crit cond}, at least another is also zero.
Suppose that, at a critical point, not all the variables are zero. 
If, say, $z=0$, then $Z=0$ or $v=0$.
If $z,v=0$, then $Z=0$ and then also $U=0$. 
The case $z,Z=0$ leads to $U(1-\alpha vU)=0$. If $U=0$, then again $v=0$, back at the point $O$. 

If $\alpha vU=1$, then, as $\kappa = v^2-U^2$, we can solve for $v^2,U^2$, and get $$U^2=(\sqrt{\kappa^2+4/\alpha^2}-\kappa)/2, \qquad v^2=(\sqrt{\kappa^2+4/\alpha^2}+\kappa)/2, $$ 
which gives two solutions for $v,U$ since they have the same sign.
Similar reasoning applies to all variables, and we get four possible critical points of type $A$ with exactly two variables equal to zero.

\medskip
Suppose now that at a critical point, no variable equals zero.

We combine the first ($\beta v^2Z^2 = \delta z^2U^2$) and the second ($vU- zZ =\alpha(vU-zZ)(vU+zZ)$) conditions to get an expression for $Z^2, z^2$ and $zZ$ in terms of $v, U$, and we have the following two cases:
\begin{itemize}
\item[($a$)] $vU- zZ \neq 0$, then, $1=\al(vU+zZ)$,  
\begin{align*}
\qquad& 0\not=\alpha z Z=1-\alpha v U
\quad\text{and}\quad \frac{\delta}{\beta}U^2\frac{z^2}{v^2}=Z^2=\frac{1}{\alpha^2z^2}(1-\alpha v U)^2 
\\
\text{so}\quad&z^2=\sqrt{\frac{\beta}{\delta}}\sqrt{\frac{v^2}{U^2}}\frac{1}{\alpha}\sqrt{(1-\alpha vU)^2},\quad
\end{align*}
\item[($b$)]$vU- zZ = 0$, thus,
\begin{align*}
& zZ=vU\quad\text{and}\quad \frac{Z^2}{U^2}=\frac{v^2}{z^2}=\frac{\delta}{\beta}\frac{U^2}{Z^2}, 
\\
\text{so}\quad&
Z^2 =U^2 \sqrt{\frac{\delta}{\beta}}\quad\text{and}\quad
z^2=v^2\sqrt{\frac{\beta}{\delta}}.
\end{align*}
\end{itemize}

\medskip
\noindent
The case $(a)$ applied to the fourth condition in~\eqref{equ:crit cond} leads to
\begin{align*}
     0&=U(1-\alpha v U)-\beta v \frac{1}{\alpha^2z^2}(1-\alpha vU)^2-\frac{\gamma}{\alpha} (1-\alpha vU)U,
     \\
     0&=U(1-\alpha v U)\frac{1}{\alpha}(\alpha-\frac{\beta v (1-\alpha v U)}{U\alpha z^2 } -\gamma)\\
     &=\frac{1}{\alpha}U(1-\alpha v U)(\alpha-\frac{\beta v(1-\alpha vU)}{U\alpha }\frac{\sqrt{\delta U^2}\alpha}{\sqrt{\beta v^2}\sqrt{(1-\alpha v U)^2}}-\gamma)\\
     &
    =U(\frac{1}{\alpha}- v U) (\alpha\mp\sqrt{\beta\delta} -\gamma)\,.
\end{align*}
As $U(1-\alpha v U)=\alpha UzZ\not=0$, we would have 
$\alpha \mp \sqrt{\beta\delta}-\gamma=0$, which cannot be satisfied for positive $N$ (in fact for $N>-5$). Therefore, condition $1=\alpha(zZ+vU)$ does not yield any critical points. 

In the case $(b)$, when $zZ=vU$ and also 
$\frac{Z^2}{U^2}=\frac{v^2}{z^2}=\sqrt{\frac{\delta}{\beta}}$, we have 
\begin{align*}
     \kappa&=v^2\left(1+\sqrt{\frac{\beta}{\delta}}\right) - U^2\left(1+\sqrt{\frac{\delta}{\beta}}\right)  
     \\
     0&=U\left(1-\alpha v U -\beta v U\sqrt{\frac{\delta}{\beta}} - \gamma vU\right)\quad\text{ or, as  }U\not=0
     \\
     1&=vU(\alpha+\sqrt{{\delta}{\beta}}+\gamma)=\frac{v U}{\rho}\,
\end{align*}
which can be converted into quadratic equation which has one solution for $v^2>0$ and provides $U^2$ as well.
Similarly $\kappa=-Z^2\left(1+\sqrt{\frac{\beta}{\delta}}\right) + z^2\left(1+\sqrt{\frac{\delta}{\beta}}\right)$ and $zZ=\rho$. As $v$ and $U$ have the same sign and $z$ and $Z$ have the same sign, we have exactly four critical points of type $B$. 

We note that for $\kappa=0$ we have $v=U$ and $z=Z$.  

The values of $L$ follow by simple calculations. 
 \end{proof}    

\begin{cor}
    For critical points of type $B$,\\ using $\rho$ and $\mu^2=\left(1+\sqrt{\frac{\delta}{\beta}}\right)\left(1+\sqrt{\frac{\beta}{\delta}}\right)=\frac{(\sqrt{\beta}+\sqrt{\delta})^2}{\sqrt{\beta\delta}}$, we have: 
\begin{align*}
    U^2&=\frac{\sqrt{\beta}}{2(\sqrt{\beta}+\sqrt{\delta})}(\sqrt{\kappa^2+4\rho^2\mu^2}-\kappa)
    \\
    z^2&=\frac{\sqrt{\beta}}{2(\sqrt{\beta}+\sqrt{\delta})}(\sqrt{\kappa^2+4\rho^2\mu^2}+\kappa)
    \\
    Z^2&=\frac{\sqrt{\delta}}{2(\sqrt{\beta}+\sqrt{\delta})}(\sqrt{\kappa^2+4\rho^2\mu^2}-\kappa)
    \\
        v^2&=\frac{\sqrt{\delta}}{2(\sqrt{\beta}+\sqrt{\delta})}(\sqrt{\kappa^2+4\rho^2\mu^2}+\kappa)
    \,.
\end{align*}
\end{cor}
\begin{proof}
Substituting $U=\rho/v$ and solving for $v^2$, we get  
\begin{align*}
   \kappa&=v^2\left(1+\sqrt{\frac{\beta}{\delta}}\right)-\frac{\rho^2}{v^2}\left(1+\sqrt{\frac{\delta}{\beta}}\right) \\
   v^2&=\frac{\sqrt{\delta}}{2(\sqrt{\beta}+\sqrt{\delta})}\left(\sqrt{\kappa^2+4\rho^2\mu^2}+\kappa\right)
\end{align*}
where we choose the positive root for $v^2$. Analogous calculations apply to the remaining variables. 
\end{proof}

\begin{rem}\label{rem:CritPointsSignature}
    The critical points can be symbolically labeled according to the signs of the variables. For type $A$ we have $(+0\,0+), (-0\,0-),(0++0), (0--0)$ and for type $B$: $(++++), (+--+), (-++-),(----)$  
\end{rem}
\subsection{Stability of the critical points}
The character of the critical points depends on the eigenvalues of the Hessian $H(S)=\frac{\partial \dot{S}}{\partial S}=- 
\frac{\partial^2 L}{\partial S^2}$ at these points,
\[
\small
-\left(
\begin{array}{cccc}
 \alpha  v^2 +\delta  z^2 & \gamma  v Z+2 \delta  U z & \gamma  v z & 2 \alpha  U v+\gamma  z Z-1 \\
 \gamma  v Z+2 \delta  U z & \delta  U^2+\alpha  Z^2 & \gamma  U v+2 \alpha  z Z-1 & \gamma  U Z \\
 \gamma  v z & \gamma  U v+2 \alpha  z Z-1 & \beta  v^2+\alpha  z^2 & \gamma  U z+2 \beta  v Z \\
 2 \alpha  U v+\gamma  z Z-1 & \gamma  U Z & \gamma  U z+2 \beta  v Z & \alpha  U^2+\beta  Z^2 \\
\end{array}
\right)
\]

At the points where $vU=zZ=\rho$, the vector ${\rm grad} (\kappa)=2 (-U, z, -Z, v)$ lies in the kernel of Hessian, in particular, it happens at all the critical points. This corresponds to the fact that each leaf of the foliation $K_\kappa$ is invariant. If the remaining eigenvalues are negative, the point is an attractor, if they are positive, it is a repellor, if the signs are different, then we have a saddle.

\begin{thm}
\label{thm:1DCharacter}
 At all critical points, the dynamic has a neutral direction corresponding to the invariant foliation $K_\kappa$. 
    The critical points of type $A$ are attractors in $K_\kappa$. The critical points of type $B$ are saddle points with two eigenvalues negative (their stable submanifolds are two dimensional) and one positive eigenvalue (corresponding to the repelling direction). For $\kappa=0$, the special critical point $(0,0,0,0)$ is a saddle.
\end{thm}
We recall that at points $A$ the value of $L$ is minimal with the value~$-\frac{1}{2\alpha}$.
\begin{proof}\ 
\hskip-2cm
\begin{itemize}
\item[Point $O$:]
At the special point $(0,0,0,0)$ the Hessian $H$ becomes: 
\[
H(0,0,0,0)=\left(\begin{array}{cccc}
  0&0&0&1\\
  0&0&1&0\\
  0&1&0&0\\
  1&0&1&0
\end{array}
\right)
\]
with the eigenvalues $(1,1,-1,-1)$ and the orthogonal system of eigenvectors: 
\[
\left(\begin{array}{cccc}
  1&0&1&0\\
  0&1&0&1\\
  0&1&0&-1\\
  1&0&-1&0
\end{array}
\right)\,.
\]
We note that the point $O$ is singular in the manifold $K_0$.

For the remaining critical pots of both types we have:
\item[Type $A$:]
At the points $(U,0,0,v)$ (resp. $(0,z,Z,0)$) with $\alpha v U=1$ ($\alpha zZ=1$) and $\kappa= v^2-U^2$ ($\kappa=z^2-Z^2$), $H$ becomes:
\begin{equation*}\hspace*{-1cm}
\left(\begin{array}{cccc}
  -\alpha v^2&0&0&1-2\alpha U v\\
  0&-\delta U^2&1-\gamma U v&0\\
  0&1-\gamma U v&-\beta v^2&0\\
  1-2\alpha U v&0&0&-\alpha U^2
\end{array}
\right)=
\left(\begin{array}{cccc}
  -\alpha v^2&0&0&-1\\
  0&-\frac{\delta}{\alpha^2 v^2}&1-\frac{\gamma}{\alpha}&0\\
  0&1-\frac{\gamma}{\alpha}&-\beta v^2&0\\
  -1&0&0&-\frac{1}{\alpha v^2}
\end{array}
\right).
\end{equation*}
The external block provides the eigenvalues 0 and $-\left(\frac{1}{\alpha v^2}+\alpha v^2\right)\le -2$. The determinant of the internal block is independent of the value of $v$, it is positive for $N\ge 0$:
\[
\frac{\beta\delta-(\gamma-\alpha)^2}{\alpha^2}>
\frac{(3(N+2)(N+8)-(N+5)^2}{(N+2)^2}=\frac{2N^2+20N+23}{(N+2)^2}>0\,.
\] 
The trace is negative, so that both eigenvalues are negative.
The eigendirection for the zero eigenvalue is given by $(v,- \frac{1}{\alpha v},0,0)$. 
Clearly, the same holds for the critical points $(0,z,Z,0)$.
In conclusion, all the critical points of type $A$ are attracting on their invariant manifold.
\item[Type $B$:]
For the points with no zero variables the trace of $H=H(U,z,Z,v)$ is negative. As it is equal to the sum of the eigenvalues, at least one of them is negative. Consider the sum of the principal minors of rank 3. It corresponds to the sum of the four products of three different eigenvalues, of which only one product is nonzero. The calculation of the minors is elementary but tedious. After using conditions
$ zZ=vU=\rho=\frac{1}{\alpha+\sqrt{\beta\delta}+\gamma}$ and $\frac{Z^2}{U^2}=\frac{v^2}{z^2}=\frac{\delta}{\beta}\frac{U^2}{Z^2}$ the first principal minor (which excludes the fourth row and column) reads
\[
\frac{\sqrt{\beta } \sqrt{\delta } v^2 \left(-4 \alpha ^2+\alpha  \left(3 \delta -3 \sqrt{\beta } \sqrt{\delta }\right)+4 \left(\sqrt{\beta } \sqrt{\delta }+\gamma \right)^2\right)}{\left(\alpha +\sqrt{\beta } \sqrt{\delta }+\gamma \right)^2}\,,
\]
where the expression $-4 \alpha ^2+\alpha  (3 \delta -3 \sqrt{\beta } \sqrt{\delta })+4 (\sqrt{\beta } \sqrt{\delta }+\gamma )^2$ is decreasing in $N$, (which defines the parameters) and positive for the limiting values $\alpha =1, \beta=3, \delta=1,\gamma=2$. That proves that this minor is positive. 

By symmetry of the equations, the remaining three minors behave in an analogous way and are also positive. Thus, the product of three nonzero eigenvalues is positive, and hence two are negative and one is positive.  
In conclusion, all four critical points of type $B$ are hyperbolic saddles on each invariant leaf $K_\kappa$ with two-dimensional stable and one-dimensional unstable directions. 
\end{itemize}

\end{proof}
\begin{figure}[h]
    \centering
    \includegraphics[width=0.27\linewidth]{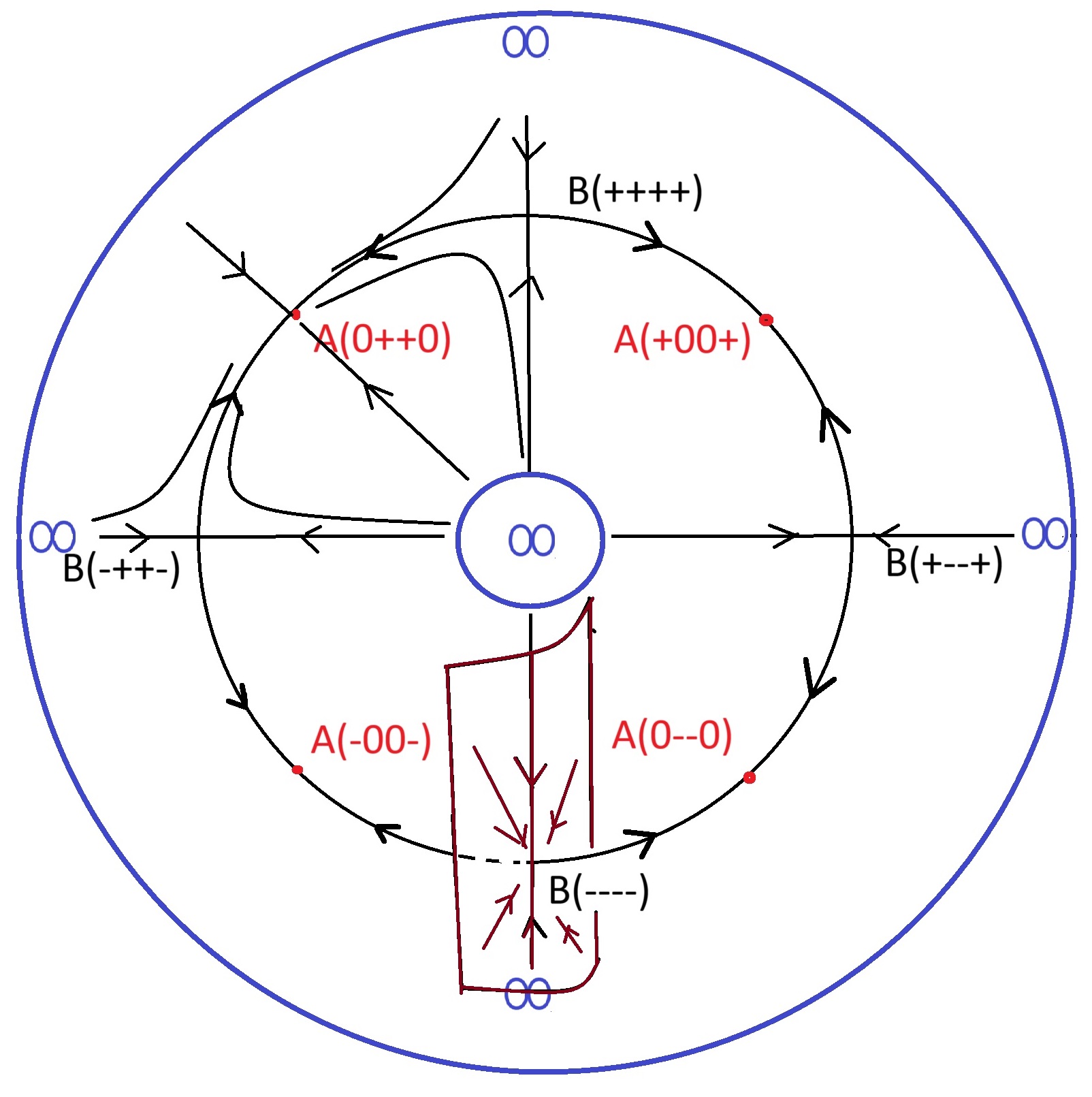}
    \includegraphics[width=0.47\linewidth]{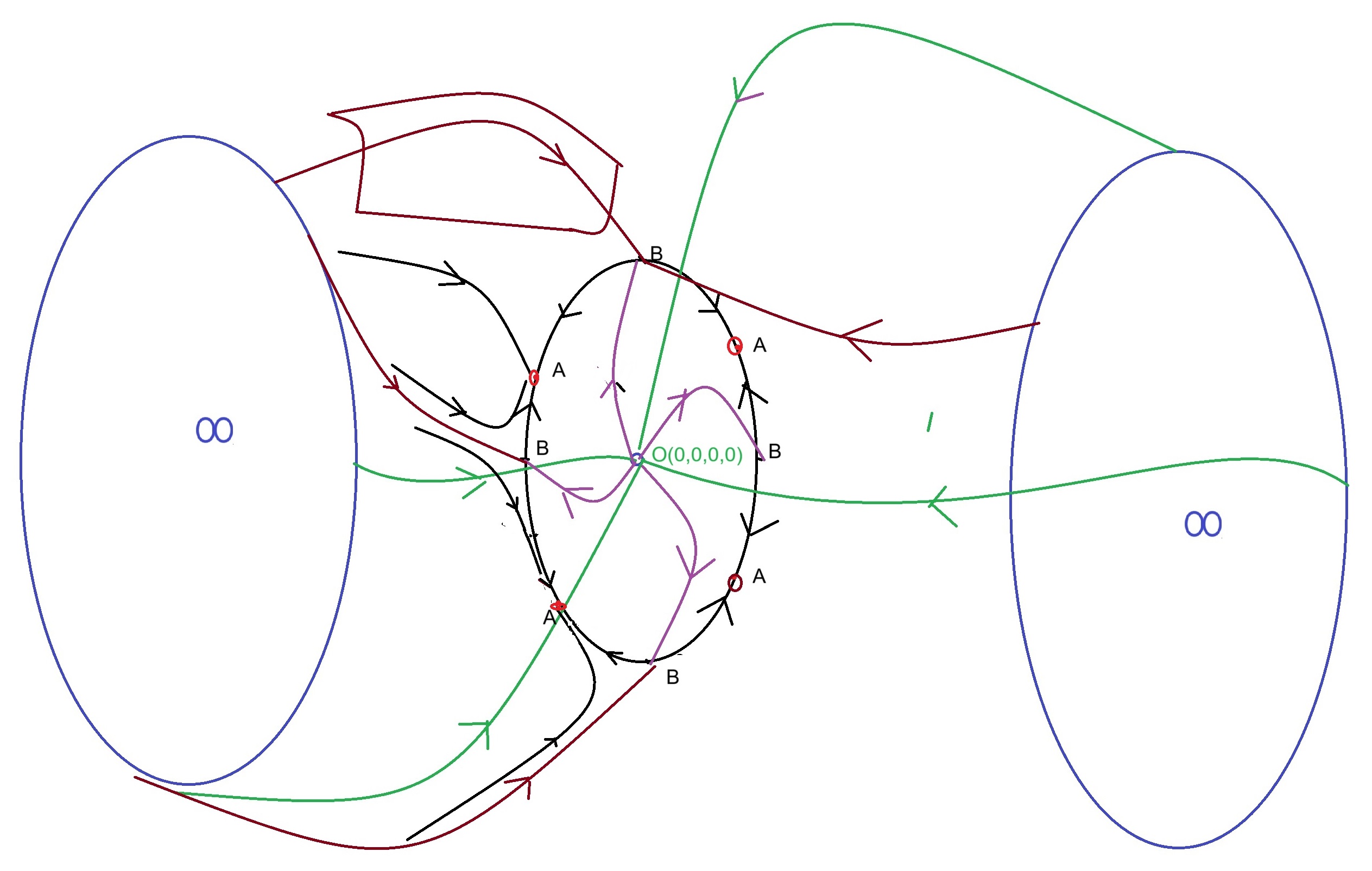}   
    \caption{Some trajectories when $\kappa\not=0$ and when $\kappa=0$.}
    \label{fig:UzZv}
\end{figure}
\begin{rem}
    The relative positions of the points and manifolds are as follows. We use the convention from Remark~\ref{rem:CritPointsSignature}
    On each 3-dimensional leaf $K_\kappa$, the unstable manifold (curve) of the saddle point $B_\kappa(++++)$ joins on one side the attractor $A(+00+)$ and on the other side the attractor $A_\kappa(0++0)$. The two-dimensional (relative to $K_\kappa)$  stable manifold of $B_\kappa(++++)$ separates the bassins of attraction of $A_\kappa(+00+)$ and $A_\kappa(0++0)$. Similarly, the stable manifold of $B_\kappa(+--+)$ separates the bassins of attraction of $A_\kappa(+00+)$ and $A_\kappa(0--0)$. The stable manifolds of points $B_\kappa$ do not intersect due to the hyperboloid topology of $K_\kappa$. The case $\kappa=0$ is exceptional. The unstable manifolds (curves) of $O=(0,0,0,0)$ join four points $B$, and the stable manifolds of $O$ are coming from infinity, they are intersections of two-dimensional stable manifolds of two points $B$ with a pair of two equal signs, for example $B(++++)$ and $B(+--+)$, if all signs are different there is no intersection, for example for $B(+--+)$ and $B(-++-)$. All four stable manifolds of $B$'s meet in the sigular point $O$.  The trajectories close to these stable curves of $O$ first go close to $O$, then turn and go close to a $B$ point along with its stable manifold, and then turn again and go to a point $A$.   
\end{rem}

\begin{figure}[h]
    \centering
\includegraphics[width=0.37\linewidth]{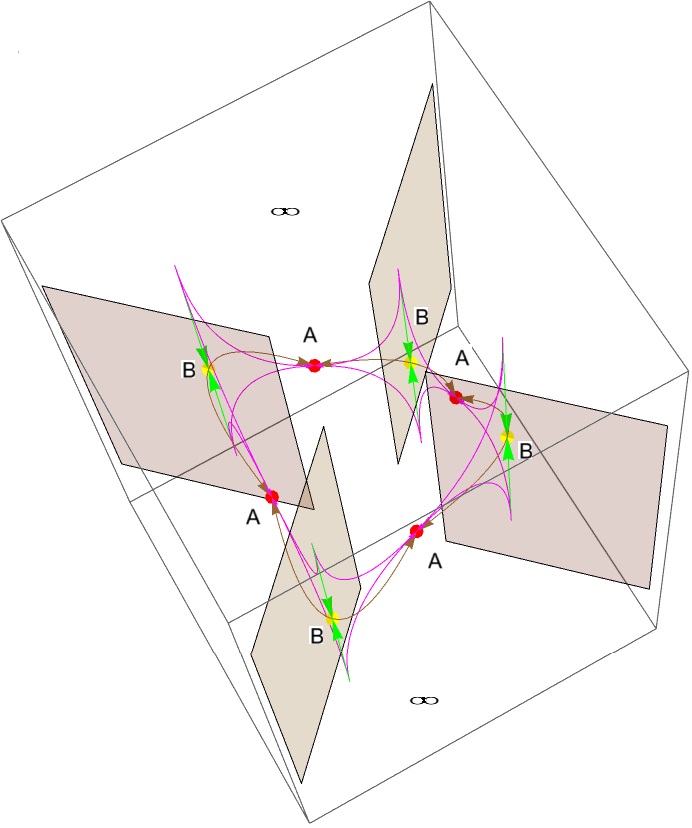}
\includegraphics[width=0.33\linewidth]{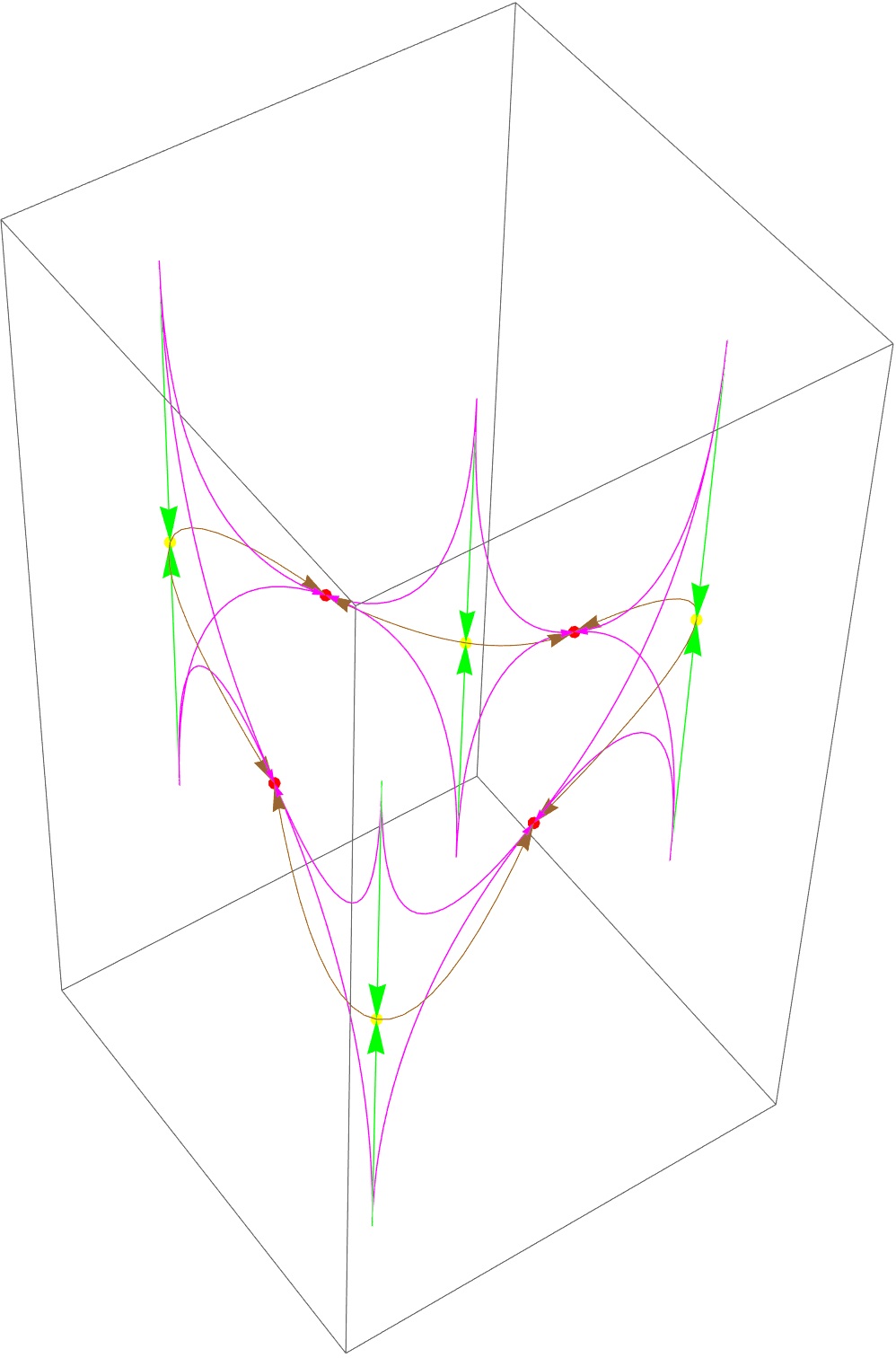}
    \caption{
    A 3D view on the trajectories $\kappa\not=0$. The 2-dimensional surfaces, the stable manifolds of points $B$ do not intersect and delineate
    the basins of attraction of points $A$.}
    \label{fig:UzZv1}
    
\end{figure}

    \begin{rem}
        At every point $S=(U,z,Z,v)$ of the invariant manifold $K_\kappa$, there is a two-dimensional affine manifold $L_\kappa(S) \subset K_\kappa$. 
        \[
        L_\kappa(S)=\{S+t A_1+s A_2| t,s\in\Real\}\,,
        \]
        where vectors $A_i$, $i=1,2$,  have coefficients $(a_i,b_i,c_i,d_i)$ that satisfy $a_iU-b_iz+c_iZ-d_iv=0$ and $a_i^2-b_i^2+c_i^2-d_i^2=0$, so that the two lines $S+t A_1$ and $S+sA_2$ lie on $K_\kappa$. In addition, $a_1a_2-b_1b_2+c_1c_2-d_1d_2=0$ for $L_\kappa(S)$ to lie entirely in $K_\kappa$. For example, with arbitrary $\sigma_i, \tau_i\in\{-1,1\}$ we can choose $a_i=\sigma_i(z-\tau_i Z)$, $b_i=v+\sigma_i U$, $c_i=\tau_i(v+\sigma_i U)$ and $d_i=-(z-\tau_i Z)$, keeping in mind that $A_i$ needs to be linearly independent, for instance $a_1=a_2$, $b_1=b_2$ but $c_1=-c_2$ and $d_1=-d_2$.  
    \end{rem}
    \begin{rem}
        When $\kappa\neq 0$, then "inside" $K_\kappa$ there is a three-dimensional hole. For example, if $\kappa>0$, we have $z^2+v^2\ge \kappa$ and the solid open $4D$ cylinder $C_\kappa=\{S: z^2+v^2<\kappa, Z,U\in\Real\}$ is disjoint from $K_\kappa$, however, for any $\epsilon>0$ the cylinder $C_{\kappa+\epsilon}$ intersects $K_\kappa$ (in a bounded set, as the hole "grows" when $z^2+Z^2$ increases). The closure of $C_\kappa$ touches the boundary of $K_\kappa$ at the circle $\{z^2+v^2=\kappa, U=0=Z\}$.  When $\kappa=0$ the two cylinders for positive and negative neighboring $\kappa$'s collapse to the lines that intersect at the singular boundary point $O=(0,0,0,0)$. 
    \end{rem}
\section{Conclusions}
\label{sec:conclusions}
The gradient flow relative to the ICL training dynamics exhibits a highly involved behavior. There are no unique minima of the cost function, and the descent can slow down at the intermediate saddle points, which are plentiful forming a manifold of lower dimension. Moreover, the minima depend strongly on the starting points as the trajectories stay on the invariant manifolds, which laminate the space into separate regimes of behavior. Our two simplified cases, one with $Z=0=z$ and one with $d=1$ already show the need for intensive calculations and quite an insight to understand the qualitative character of convergence and can answer only locally, close to the minima, the question about the rate of convergence. Meanwhile, while the saddles should not matter too much practically since stochastic gradient descent algorithms (a category in which the practical implementations of ICL certainly fall) avoid them almost surely, see, e.g.~\cite{liu2024almost}, their existence made a global universal estimate of such a rate prohibitively difficult.

\appendix
\section{Calculations of the Derivatives}
\subsection{Derivative with respect to $v$}
\label{sec:Dv}

Recall that,
\begin{align*}
\frac{\partial L}{\partial v}= &\ex\left[\frac{1}{N}\sum_{n=1}^N\sum_{i=1}^d \sum_{j=1}^d\sum_{k=1}^d(z_i  U_{jk}+z_iZ_k w_j+vw_iU_{jk}+vZ_kw_iw_j)x^{(i)}_n x_n^{(j)}x^{(k)}_q \right.
\\ &\left. +\frac{1}{N}\sum_{i=1}^d\sum_{j=1}^d \sum_{k=1}^d (z_iU_{jk}) x_q^{(j)}x_q^{(i)}  x_q^{(k)}-\sum_{m=1}^d w_mx_q^{(m)}\right]\\  &\times \left(\frac{1}{N}\sum_{n=1}^N\sum_{i=1}^d \sum_{j=1}^d\sum_{k=1}^d(w_iU_{jk}+Z_kw_iw_j)x^{(i)}_n x_n^{(j)}x^{(k)}_q\right).
\end{align*}
For the five terms that involves $\frac{1}{N}\sum_{n=1}^N\sum_{i=1}^d \sum_{j=1}^d\sum_{k=1}^dw_iU_{jk}x^{(i)}_n x_n^{(j)}x^{(k)}_q$, we have,
\begin{align}
&\ex\left(\sum_{n=1}^N\sum_{i=1}^d \sum_{j=1}^d\sum_{k=1}^dz_i  U_{jk}x^{(i)}_n x_n^{(j)}x^{(k)}_q \right)\left(\sum_{n=1}^N\sum_{i=1}^d \sum_{j=1}^d\sum_{k=1}^d(w_iU_{jk})x^{(i)}_n x_n^{(j)}x^{(k)}_q\right)=0 \label{eqn:AA}
\\&\ex\left(\sum_{n=1}^N\sum_{i=1}^d \sum_{j=1}^d\sum_{k=1}^dz_iZ_k w_jx^{(i)}_n x_n^{(j)}x^{(k)}_q \right)\left(\sum_{n=1}^N\sum_{i=1}^d \sum_{j=1}^d\sum_{k=1}^d(w_iU_{jk})x^{(i)}_n x_n^{(j)}x^{(k)}_q\right)\nonumber \\ =&N(z^\top \Lam U\Lam Z)\tr(\Lam)+ N(N+1)z^\top\Lam^2 U\Lam Z. \label{eqn:AB}\\
&\ex\left(\sum_{n=1}^N\sum_{i=1}^d \sum_{j=1}^d\sum_{k=1}^d(vw_iU_{jk}) w_jx^{(i)}_n x_n^{(j)}x^{(k)}_q \right)\left(\sum_{n=1}^N\sum_{i=1}^d \sum_{j=1}^d\sum_{k=1}^d(w_iU_{jk})x^{(i)}_n x_n^{(j)}x^{(k)}_q\right)\nonumber\\ =&Nv\tr (U\Lam U^\top \Lam )\tr(\Lam) + N(N+1) v \tr(U\Lam U^\top \Lam^2).\label{eqn:AC}
\\&\ex\left(\sum_{n=1}^N\sum_{i=1}^d \sum_{j=1}^d\sum_{k=1}^d(vZ_kw_iw_j) w_jx^{(i)}_n x_n^{(j)}x^{(k)}_q \right)\left(\sum_{n=1}^N\sum_{i=1}^d \sum_{j=1}^d\sum_{k=1}^d(w_iU_{jk})x^{(i)}_n x_n^{(j)}x^{(k)}_q\right) =0, \label{eqn:AD}
\\&\ex\left(\sum_{i=1}^d \sum_{j=1}^d\sum_{k=1}^d(z_iU_{jk}x^{(i)}_q x_q^{(j)}x^{(k)}_q \right)\left(\sum_{n=1}^N\sum_{i=1}^d \sum_{j=1}^d\sum_{k=1}^d(w_iU_{jk})x^{(i)}_n x_n^{(j)}x^{(k)}_q\right) =0, \label{eqn:AE}
\end{align}

The basic approach for evaluating other terms in the derivative calculations are the following. First, if there are odd number of appearances of $w$ in a term, then by symmetric, the term is zero. Therefore, we can immediately see that 
Equations~\eqref{eqn:AA} and~\eqref{eqn:AE} hold. Similarly, for equation~\eqref{eqn:AD}, we have,
\begin{align*}
&\ex\left(\sum_{n=1}^N\sum_{i=1}^d \sum_{j=1}^d\sum_{k=1}^d(vZ_kw_iw_j) w_jx^{(i)}_n x_n^{(j)}x^{(k)}_q \right)\left(\sum_{n=1}^N\sum_{i=1}^d \sum_{j=1}^d\sum_{k=1}^d(w_iU_{jk})x^{(i)}_n x_n^{(j)}x^{(k)}_q\right)\\ =&\ex\left(\sum_{n=1}^N\sum_{i=1}^d \sum_{j=1}^d\sum_{k=1}^d\sum_{n'=1}^N\sum_{i'=1}^d \sum_{j'=1}^d\sum_{k'=1}^d(vZ_kw_iw_j) w_jx^{(i)}_n x_n^{(j)}x^{(k)}_q (w_{i'}U_{j'k'})x^{(i')}_{n'} x_{n'}^{(j')}x^{(k')}_q\right)\\ =&0.
\end{align*}
If there are two appearance  $\ex[w_iw_j]=\delta_{ij}$, with $\delta_{ij}$ being the Kronecker symbol, can help us remove them; if there are four of them, then we will apply Isserlis' theorem, see, e.g.~\cite{isserlis1918formula} on higher moments of multivariate normal variables for simplification. After that, we will have a summation of a product of six $x_n$ or $x_q$ terms. There are only three possibilities:  (I) four $x_n$ terms and two $x_q$; (II) two $x_n$ terms and four $x_q$; (III) six $x_q$ terms. In cases (I), first take expectation with respect to $x_q$, then for the four $x_n$ terms, apply Isserlis' theorem for those that have common index $n$, otherwise directly apply expectation. Similar approach can be applied for case (II).  Apply Isserlis' theorem for the sixth moment in case (III). 

\subsubsection*{Derivation of equation \eqref{eqn:AB}}
\begin{align*}
&\ex\left(\sum_{n=1}^N\sum_{i=1}^d \sum_{j=1}^d\sum_{k=1}^dz_iZ_k w_jx^{(i)}_n x_n^{(j)}x^{(k)}_q \right)\left(\sum_{n=1}^N\sum_{i=1}^d \sum_{j=1}^d\sum_{k=1}^d(w_iU_{jk})x^{(i)}_n x_n^{(j)}x^{(k)}_q\right)
\\ \stackrel{(a)}{=}&
\ex\left(\sum_{n=1}^N\sum_{i=1}^d \sum_{j=1}^d\sum_{k=1}^d\sum_{n'=1}^N\sum_{i'=1}^d \sum_{j'=1}^d\sum_{k'=1}^dz_iZ_k w_jx^{(i)}_n x_n^{(j)}x^{(k)}_q (w_{i'}U_{j'k'})x^{(i')}_{n'} x_{n'}^{(j')}x^{(k')}_q\right)
\\ \stackrel{(b)}{=}&
\ex\left(\sum_{n=1}^N\sum_{i=1}^d \sum_{j=1}^d\sum_{k=1}^d\sum_{n'=1}^N\sum_{j'=1}^d\sum_{k'=1}^dz_iZ_k x^{(i)}_n x_n^{(j)}x^{(k)}_q U_{j'k'}x^{(j)}_{n'} x_{n'}^{(j')}x^{(k')}_q\right)
\\ \stackrel{(c)}{=}&
\ex\left(\sum_{n=1}^N\sum_{i=1}^d \sum_{j=1}^d\sum_{k=1}^d\sum_{n'=1}^N\sum_{j'=1}^d\sum_{k'=1}^dz_iZ_k U_{j'k'}x^{(i)}_n x_n^{(j)}x^{(j)}_{n'} x_{n'}^{(j')}\Lam_{kk'}\right)
\\ \stackrel{(d)}{=}&
\ex\left(\sum_{n=1}^N\sum_{i=1}^d \sum_{j=1}^d\sum_{k=1}^d\sum_{j'=1}^d\sum_{k'=1}^dz_iZ_k U_{j'k'}x^{(i)}_n x_n^{(j)}x^{(j)}_{n} x_{n}^{(j')}\Lam_{kk'}\right)
\\ &+
\ex\left(\sum_{n=1}^N\sum_{i=1}^d \sum_{j=1}^d\sum_{k=1}^d\sum_{n'=1,n/\neq n}^N\sum_{j'=1}^d\sum_{k'=1}^dz_iZ_k U_{j'k'}x^{(i)}_n x_n^{(j)}x^{(j)}_{n'} x_{n'}^{(j')}\Lam_{kk'}\right)
\\\stackrel{(e)}{=}& N \sum_{i=1}^d \sum_{j=1}^d\sum_{k=1}^d\sum_{j'=1}^d\sum_{k'=1}^dz_iZ_k U_{j'k'} (2\Lam_{ij}\Lam_{jj'}+ \Lam_{ij'}\Lam_{jj})\Lam_{kk'} 
\\ &+ N(N-1) \sum_{i=1}^d \sum_{j=1}^d\sum_{k=1}^d\sum_{j'=1}^d\sum_{k'=1}^dz_iZ_k U_{j'k'}\Lam_{ij}\Lam_{jj'}\Lam_{kk'}
\\=& N \sum_{i=1}^d \sum_{j=1}^d\sum_{k=1}^d\sum_{j'=1}^d\sum_{k'=1}^dz_iZ_k U_{j'k'}  \Lam_{ij'}\Lam_{jj}\Lam_{kk'} 
\\ &+ N(N+1) \sum_{i=1}^d \sum_{j=1}^d\sum_{k=1}^d\sum_{j'=1}^d\sum_{k'=1}^dz_iZ_k U_{j'k'}\Lam_{ij}\Lam_{jj'}\Lam_{kk'}\\ =&N(z^\top \Lam U\Lam Z)\tr(\Lam)+ N(N+1)z^\top\Lam^2 U\Lam Z,
\end{align*}
where (a) is the result of combining all the summations together; (b) follows from taking expectations with respect to $w$ variables; (c) is the result of take expectations with respect to $x_q$; (d) results from separating independent terms; Isserlis' theorem for the fourth moments is applied in (e); and the last two steps are simplifications.

\subsubsection*{Derivation of equation \eqref{eqn:AC}}

Following the same procedure as above, we have, 
\begin{align*}
&\ex\left(\sum_{n=1}^N\sum_{i=1}^d \sum_{j=1}^d\sum_{k=1}^d(vw_iU_{jk}) x^{(i)}_n x_n^{(j)}x^{(k)}_q \right)\left(\sum_{n=1}^N\sum_{i=1}^d \sum_{j=1}^d\sum_{k=1}^d(w_iU_{jk})x^{(i)}_n x_n^{(j)}x^{(k)}_q\right)\\ =&\ex\left(\sum_{n=1}^N\sum_{i=1}^d \sum_{j=1}^d\sum_{k=1}^d(vw_iU_{jk}) x^{(i)}_n x_n^{(j)}x^{(k)}_q \sum_{n'=1}^N\sum_{i'=1}^d \sum_{j'=1}^d\sum_{k'=1}^d(w_{i'}U_{j'k'})x^{(i')}_{n'} x_{n'}^{(j')}x^{(k')}_q\right)\\
=&\ex\left(\sum_{n=1}^N\sum_{i=1}^d \sum_{j=1}^d\sum_{k=1}^d(vU_{jk}) x^{(i)}_n x_n^{(j)}x^{(k)}_q \sum_{n'=1}^N \sum_{j'=1}^d\sum_{k'=1}^d(U_{j'k'})x^{(i)}_{n'} x_{n'}^{(j')}x^{(k')}_q\right)
\\
=&\ex\left(\sum_{n=1}^N\sum_{i=1}^d \sum_{j=1}^d\sum_{k=1}^d\sum_{n'=1}^N \sum_{j'=1}^d\sum_{k'=1}^d(vU_{jk}U_{j'k'}) x^{(i)}_n x_n^{(j)} x^{(i)}_{n'} x_{n'}^{(j')}\Lam_{kk'}\right)
\\
=&\ex\left(\sum_{n=1}^N\sum_{i=1}^d \sum_{j=1}^d\sum_{k=1}^d\sum_{j'=1}^d\sum_{k'=1}^d(vU_{jk}U_{j'k'}) x^{(i)}_n x_n^{(j)} x^{(i)}_{n} x_{n}^{(j')}\Lam_{kk'}\right)
\\
&+\ex\left(\sum_{n=1}^N\sum_{i=1}^d \sum_{j=1}^d\sum_{k=1}^d\sum_{n'=1,n'\neq n}^N \sum_{j'=1}^d\sum_{k'=1}^d(vU_{jk}U_{j'k'}) x^{(i)}_n x_n^{(j)} x^{(i)}_{n'} x_{n'}^{(j')}\Lam_{kk'}\right)
\\=& N\sum_{i=1}^d \sum_{j=1}^d\sum_{k=1}^d\sum_{j'=1}^d\sum_{k'=1}^d(vU_{jk}U_{j'k'})(2\Lam_{ij}\Lam_{ij'} + \Lam_{ii}\Lam_{jj'})\Lam_{kk'} 
\\ &+ N(N-1)\sum_{i=1}^d \sum_{j=1}^d\sum_{k=1}^d\sum_{j'=1}^d\sum_{k'=1}^d(vU_{jk}U_{j'k'})\Lam_{ij}\Lam_{ij'} \Lam_{kk'}
\\=& N\sum_{i=1}^d \sum_{j=1}^d\sum_{k=1}^d\sum_{j'=1}^d\sum_{k'=1}^d(vU_{jk}U_{j'k'})\Lam_{ii}\Lam_{jj'}\Lam_{kk'} 
\\ &+ N(N+1)\sum_{i=1}^d \sum_{j=1}^d\sum_{k=1}^d\sum_{j'=1}^d\sum_{k'=1}^d(vU_{jk}U_{j'k'})\Lam_{ij}\Lam_{ij'} \Lam_{kk'}
\\ =&Nv\tr (U\Lam U^\top \Lam )\tr(\Lam) + N(N+1) v \tr(U\Lam U^\top \Lam^2).
\end{align*}

Next, for the five terms that involves $\left(\sum_{n=1}^N\sum_{i=1}^d \sum_{j=1}^d\sum_{k=1}^d(Z_kw_iw_j)x^{(i)}_n x_n^{(j)}x^{(k)}_q\right)$, first, we have,

\begin{align*}
&\ex\left(\sum_{n=1}^N\sum_{i=1}^d \sum_{j=1}^d\sum_{k=1}^d(z_i  U_{jk}) x^{(i)}_n x_n^{(j)}x^{(k)}_q \right)\left(\sum_{n=1}^N\sum_{i=1}^d \sum_{j=1}^d\sum_{k=1}^d(Z_kw_iw_j)x^{(i)}_n x_n^{(j)}x^{(k)}_q\right)\\=&
\ex\left(\sum_{n=1}^N\sum_{i=1}^d \sum_{j=1}^d\sum_{k=1}^dz_i  U_{jk}x^{(i)}_n x_n^{(j)}x^{(k)}_q \right)\left(\sum_{n'=1}^N\sum_{i'=1}^d \sum_{j'=1}^d\sum_{k'=1}^d(Z_{k'}w_{i'}w_{j'})x^{(i')}_{n'} x_{n'}^{(j')}x^{(k')}_q\right)
\\=&
\ex\left(\sum_{n=1}^N\sum_{i=1}^d \sum_{j=1}^d\sum_{k=1}^d\sum_{n'=1}^N\sum_{i'=1}^d \sum_{j'=1}^d\sum_{k'=1}^dz_i  U_{jk}x^{(i)}_n x_n^{(j)}x^{(k)}_q (Z_{k'}w_{i'}w_{j'})x^{(i')}_{n'} x_{n'}^{(j')}x^{(k')}_q\right)
\\=&
\ex\left(\sum_{n=1}^N\sum_{i=1}^d \sum_{j=1}^d\sum_{k=1}^d\sum_{n'=1}^N\sum_{i'=1}^d \sum_{k'=1}^dz_i  U_{jk}Z_{k'}x^{(i)}_n x_n^{(j)}x^{(k)}_q x^{(i')}_{n'} x_{n'}^{(i')}x^{(k')}_q\right)
\\=&
\ex\left(\sum_{n=1}^N\sum_{i=1}^d \sum_{j=1}^d\sum_{k=1}^d\sum_{n'=1}^N\sum_{i'=1}^d\sum_{k'=1}^d z_i  U_{jk}Z_{k'}\Lam_{kk'}x^{(i)}_n x_n^{(j)} x^{(i')}_{n'} x_{n'}^{(i')}\right)
\\ =& \ex\left(\sum_{n=1}^N\sum_{i=1}^d \sum_{j=1}^d\sum_{k=1}^d\sum_{i'=1}^d\sum_{k'=1}^d z_i  U_{jk}Z_{k'}\Lam_{kk'}x^{(i)}_n x_n^{(j)} x^{(i')}_{n} x_{n}^{(i')}\right)+
\\ &\quad + \ex\left(\sum_{n=1}^N\sum_{i=1}^d \sum_{j=1}^d\sum_{k=1}^d\sum_{n'=1, n'\neq n}^N\sum_{i'=1}^d\sum_{k'=1}^d z_i  U_{jk}Z_{k'}\Lam_{kk'}x^{(i)}_n x_n^{(j)} x^{(i')}_{n'} x_{n'}^{(i')}\right)
\\ =&N \sum_{i=1}^d \sum_{j=1}^d\sum_{k=1}^d\sum_{i'=1}^d\sum_{k'=1}^d
z_i  U_{jk}Z_{k'}\Lam_{kk'}(\Lam_{ij}\Lam_{i'i'}+2\Lam_{ii'}\Lam_{ji'})+
\\ &\quad +N(N-1)\sum_{i=1}^d \sum_{j=1}^d\sum_{k=1}^d\sum_{i'=1}^d\sum_{k'=1}^d
z_i  U_{jk}Z_{k'}\Lam_{kk'}\Lam_{ij}\Lam_{i'i'}
\\ =&2N \sum_{i=1}^d \sum_{j=1}^d\sum_{k=1}^d\sum_{i'=1}^d\sum_{k'=1}^d
z_i  U_{jk}Z_{k'}\Lam_{kk'}\Lam_{ii'}\Lam_{ji'}
+N^2\sum_{i=1}^d \sum_{j=1}^d\sum_{k=1}^d\sum_{i'=1}^d\sum_{k'=1}^d
z_i  U_{jk}Z_{k'}\Lam_{kk'}\Lam_{ij}\Lam_{i'i'}
\\ =&2Nz^\top\Lam^2U\Lam Z+N^2z^\top \Lam U\Lam Z\tr(\Lam).
\end{align*}
Second, we have, 
\begin{align*}
&\ex\left(\sum_{n=1}^N\sum_{i=1}^d \sum_{j=1}^d\sum_{k=1}^d(z_iZ_k w_j) x^{(i)}_n x_n^{(j)}x^{(k)}_q \right)\left(\sum_{n=1}^N\sum_{i=1}^d \sum_{j=1}^d\sum_{k=1}^d(Z_kw_iw_j)x^{(i)}_n x_n^{(j)}x^{(k)}_q\right) =0.
\end{align*}
Third,
\begin{align*}
&\ex\left(\sum_{n=1}^N\sum_{i=1}^d \sum_{j=1}^d\sum_{k=1}^d(vw_iU_{jk}) x^{(i)}_n x_n^{(j)}x^{(k)}_q \right)\left(\sum_{n=1}^N\sum_{i=1}^d \sum_{j=1}^d\sum_{k=1}^d(Z_kw_iw_j)x^{(i)}_n x_n^{(j)}x^{(k)}_q\right) =0.
\end{align*}
These holds due to the symmetry of the $w$ variables. 
And the fourth, 
\begin{align*}
&\ex\left(\sum_{n=1}^N\sum_{i=1}^d \sum_{j=1}^d\sum_{k=1}^d(vZ_kw_iw_j) w_jx^{(i)}_n x_n^{(j)}x^{(k)}_q \right)\left(\sum_{n=1}^N\sum_{i=1}^d \sum_{j=1}^d\sum_{k=1}^d(Z_kw_iw_j)x^{(i)}_n x_n^{(j)}x^{(k)}_q\right) 
\\=&
\ex\left(\sum_{n=1}^N\sum_{i=1}^d \sum_{j=1}^d\sum_{k=1}^d\sum_{n'=1}^N\sum_{i'=1}^d \sum_{j'=1}^d\sum_{k'=1}^dvZ_kw_iw_jx^{(i)}_n x_n^{(j)}x^{(k)}_q (Z_{k'}w_{i'}w_{j'})x^{(i')}_{n'} x_{n'}^{(j')}x^{(k')}_q\right)
\\=&
\ex\left(\sum_{n=1}^N\sum_{i=1}^d \sum_{j=i}\sum_{k=1}^d\sum_{n'=1}^N\sum_{i'=1, i'\neq i}^d \sum_{j'=i'}\sum_{k'=1}^dvZ_kw_iw_jx^{(i)}_n x_n^{(j)}x^{(k)}_q (Z_{k'}w_{i'}w_{j'})x^{(i')}_{n'} x_{n'}^{(j')}x^{(k')}_q\right)
\\&+
\ex\left(\sum_{n=1}^N\sum_{i=1}^d \sum_{j=i}\sum_{k=1}^d\sum_{n'=1}^N\sum_{i'=i}\sum_{j'=i}\sum_{k'=1}^dvZ_kw_iw_jx^{(i)}_n x_n^{(j)}x^{(k)}_q (Z_{k'}w_{i'}w_{j'})x^{(i')}_{n'} x_{n'}^{(j')}x^{(k')}_q\right)
\\\stackrel{(a)}{=}&
\ex\left(\sum_{n=1}^N \sum_{i=1}^d \sum_{k=1}^d\sum_{n'=1}^N\sum_{i'=1, i'\neq i}^d \sum_{k'=1}^dvZ_kx^{(i)}_n x_n^{(i)}x^{(k)}_q Z_{k'}x^{(i')}_{n'} x_{n'}^{(i')}x^{(k')}_q\right)
\\&+
3\ex\left(\sum_{n=1}^N\sum_{i=1}^d \sum_{k=1}^d\sum_{n'=1}^N\sum_{k'=1}^dvZ_kx^{(i)}_n x_n^{(i)}x^{(k)}_q (Z_{k'})x^{(i)}_{n'} x_{n'}^{(i)}x^{(k')}_q\right)
\\=&
\underbrace{\ex\left(\sum_{n=1}^N \sum_{i=1}^d \sum_{k=1}^d\sum_{n'=1}^N\sum_{i'=1}^d \sum_{k'=1}^dvZ_kx^{(i)}_n x_n^{(i)}x^{(k)}_q Z_{k'}x^{(i')}_{n'} x_{n'}^{(i')}x^{(k')}_q\right)}_{J_1}
\\&+
\underbrace{2\ex\left(\sum_{n=1}^N\sum_{i=1}^d \sum_{k=1}^d\sum_{n'=1}^N\sum_{k'=1}^dvZ_kx^{(i)}_n x_n^{(i)}x^{(k)}_q (Z_{k'})x^{(i)}_{n'} x_{n'}^{(i)}x^{(k')}_q\right)}_{J_2}.
\end{align*}
Here, (a) results from a fourth moment calculation of $w$ variables, with the Isserlis' theorem applied.
\begin{align*}
J_1=&\ex\left(\sum_{n=1}^N \sum_{i=1}^d \sum_{k=1}^d\sum_{n'=1}^N\sum_{i'=1}^d \sum_{k'=1}^dvZ_kx^{(i)}_n x_n^{(i)}x^{(k)}_q Z_{k'}x^{(i')}_{n'} x_{n'}^{(i')}x^{(k')}_q\right)
\\=&\ex\left(\sum_{n=1}^N \sum_{i=1}^d \sum_{k=1}^d\sum_{i'=1}^d \sum_{k'=1}^dvZ_kx^{(i)}_n x_n^{(i)}x^{(k)}_q Z_{k'}x^{(i')}_{n} x_{n}^{(i')}x^{(k')}_q\right)
\\ &+\ex\left(\sum_{n=1}^N \sum_{i=1}^d \sum_{k=1}^d\sum_{n'=1, n'\neq n}^N\sum_{i'=1}^d \sum_{k'=1}^dvZ_kx^{(i)}_n x_n^{(i)}x^{(k)}_q Z_{k'}x^{(i')}_{n'} x_{n'}^{(i')}x^{(k')}_q\right)
\\=& N \sum_{i=1}^d \sum_{k=1}^d\sum_{i'=1}^d \sum_{k'=1}^dvZ_k \Lam_{kk'}(\Lam_{ii}\Lam_{i'i'} + 2\Lam_{ii'}\Lam_{ii'})+ \\ &\quad + N(N-1)\sum_{i=1}^d \sum_{k=1}^d\sum_{i'=1}^d \sum_{k'=1}^dvZ_k \Lam_{kk'}Z_{k'}\Lam_{ii}\Lam_{i'i'}
\\=& 2N \sum_{i=1}^d \sum_{k=1}^d\sum_{i'=1}^d \sum_{k'=1}^dvZ_k \Lam_{kk'} Z_{k'}\Lam_{ii'}\Lam_{ii'} + N^2\sum_{i=1}^d \sum_{k=1}^d\sum_{i'=1}^d \sum_{k'=1}^dvZ_k \Lam_{kk'}Z_{k'}\Lam_{ii}\Lam_{i'i'}
\\=& 2NvZ^\top \Lam Z \tr(\Lam^2)+ N^2 vZ^\top \Lam Z \tr^2(\Lam)
\end{align*}
\begin{align*}
J_2=&2\ex\left(\sum_{n=1}^N\sum_{i=1}^d \sum_{k=1}^d\sum_{n'=1}^N\sum_{k'=1}^dvZ_kx^{(i)}_n x_n^{(i)}x^{(k)}_q (Z_{k'})x^{(i)}_{n'} x_{n'}^{(i)}x^{(k')}_q\right)
\\
=&2\ex\left(\sum_{n=1}^N\sum_{i=1}^d \sum_{k=1}^d\sum_{k'=1}^dvZ_kx^{(i)}_n x_n^{(i)}x^{(k)}_q (Z_{k'})x^{(i)}_{n} x_{n}^{(i)}x^{(k')}_q\right)
\\&+
2\ex\left(\sum_{n=1}^N\sum_{i=1}^d \sum_{k=1}^d\sum_{n'=1, n'\neq n}^N\sum_{k'=1}^dvZ_kx^{(i)}_n x_n^{(i)}x^{(k)}_q (Z_{k'})x^{(i)}_{n'} x_{n'}^{(i)}x^{(k')}_q\right)
\\=&6N\ex\left(\sum_{i=1}^d \sum_{k=1}^d\sum_{k'=1}^dvZ_k(\Lam_{ii})^2x^{(k)}_q (Z_{k'})x^{(k')}_q\right)
\\&+
2N(N-1)\ex\left(\sum_{i=1}^d \sum_{k=1}^d\sum_{k'=1}^dvZ_k\Lam_{ii}x^{(k)}_q (Z_{k'})\Lam_{ii}x^{(k')}_q\right)
\\=&6N\tr(\Lam^2)vZ^\top \Lam Z+2N(N-1) \tr(\Lam^2) vZ^\top \Lam Z
\\=&2N(N+2)\tr(\Lam^2)vZ^\top \Lam Z.
\end{align*}
Therefore, $J_1+J_2=[2N(N+3)\tr(\Lam^2)+N^2\tr^2(\Lam)]vZ^\top \Lam Z$.

\begin{align*}
&\ex\left(\sum_{i=1}^d \sum_{j=1}^d\sum_{k=1}^dz_iU_{jk}x^{(i)}_q x_q^{(j)}x^{(k)}_q \right)\left(\sum_{n=1}^N\sum_{i=1}^d \sum_{j=1}^d\sum_{k=1}^d(Z_kw_iw_j)x^{(i)}_n x_n^{(j)}x^{(k)}_q\right)\\=&\ex\sum_{i=1}^d \sum_{j=1}^d\sum_{k=1}^d \sum_{n=1}^N\sum_{i'=1}^d \sum_{j'=1}^d\sum_{k'=1}^dz_iU_{jk}x^{(i)}_q x_q^{(j)}x^{(k)}_q(Z_{k'}w_{i'}w_{j'})x^{(i')}_n x_n^{(j')}x^{(k')}_q\\=&\ex\sum_{i=1}^d \sum_{j=1}^d\sum_{k=1}^d \sum_{n=1}^N\sum_{i'=1}^d \sum_{k'=1}^dz_iU_{jk}x^{(i)}_q x_q^{(j)}x^{(k)}_q(Z_{k'})x^{(i')}_n x_n^{(i')}x^{(k')}_q
\\=&N\tr(\Lam)\ex\sum_{i=1}^d \sum_{j=1}^d\sum_{k=1}^d \sum_{k'=1}^dz_iU_{jk}Z_{k'}x^{(i)}_q x_q^{(j)}x^{(k)}_qx^{(k')}_q
\\=&N\tr(\Lam)\sum_{i=1}^d \sum_{j=1}^d\sum_{k=1}^d \sum_{k'=1}^dz_iU_{jk}Z_{k'}(\Lam_{ij}\Lam_{kk'}+\Lam_{ik}\Lam_{jk'}+\Lam_{ik'}\Lam_{jk})\\=&
N\tr(\Lam)[z^\top\Lam(U+U^\top)\Lam Z + z^\top\Lam Z\tr(U\Lam)].
\end{align*}
The separated term,
\begin{align*}
&\ex\left[\sum_{m=1}^d w_mx_q^{(m)} \left(\sum_{n=1}^N\sum_{i=1}^d \sum_{j=1}^d\sum_{k=1}^d(w_iU_{jk}+Z_kw_iw_j)x^{(i)}_n x_n^{(j)}x^{(k)}_q\right)\right]\\ 
=&\ex\left[\sum_{m=1}^d  \sum_{n=1}^N\sum_{i=1}^d \sum_{j=1}^d\sum_{k=1}^dw_iw_mx_q^{(m)}U_{jk}x^{(i)}_n x_n^{(j)}x^{(k)}_q\right]+
\\&\quad +\ex\left[\sum_{m=1}^d \sum_{n=1}^N\sum_{i=1}^d \sum_{j=1}^d\sum_{k=1}^dw_mx_q^{(m)}Z_kw_iw_jx^{(i)}_n x_n^{(j)}x^{(k)}_q\right]\\ 
=&\ex\left[  \sum_{n=1}^N\sum_{i=1}^d \sum_{j=1}^d\sum_{k=1}^dU_{jk}x_q^{(i)}x^{(i)}_n x_n^{(j)}x^{(k)}_q\right]
\\=&N\sum_{i=1}^d \sum_{j=1}^d\sum_{k=1}^d\Lam_{ij}U_{jk}\Lam_{ki}=N\tr(\Lam U\Lam).
\end{align*}

We would like to point out that the rest of calculations follow the same principle we elaborated and very similar procedure that are demonstrated here. So, they are provided in the supplement materials. 

\section{Supplementary  Materials}

\subsection{Detailed Calculations for $\frac{\partial L}
{\partial z}$}
\label{sec:Dz}

Recall that, for any $\ell=1,2,\ldots, d$, 
\begin{align*}
\frac{\partial L}{\partial z_\ell}= &\ex\left[\frac{1}{N}\sum_{n=1}^N\sum_{i=1}^d \sum_{j=1}^d\sum_{k=1}^d(z_i  U_{jk}+z_iZ_k w_j+vw_iU_{jk}+vZ_kw_iw_j)x^{(i)}_n x_n^{(j)}x^{(k)}_q \right.\\ &\left. +\frac{1}{N}\sum_{i=1}^d\sum_{j=1}^d \sum_{k=1}^d (z_iU_{jk}) x_q^{(j)}x_q^{(i)}  x_q^{(k)}-\sum_{m=1}^d w_mx_q^{(m)}\right]\\  &\times \left(\frac{1}{N}\sum_{n=1}^N \sum_{j=1}^d\sum_{k=1}^d(U_{jk}+Z_kw_j)x^{(\ell)}_n x_n^{(j)}x^{(k)}_q +\frac{1}{N}\sum_{j=1}^d \sum_{k=1}^d U_{jk} x_q^{(j)}x_q^{(\ell)}  x_q^{(k)}\right).
\end{align*}
Calculations below gives us, 
\begin{align*}
\left(\frac{\partial L}{\partial z}\right)_\ell = & 
\frac{1}{N}(z^\top\Lam)_\ell\tr(U\Lam U^\top \Lam)+ \frac{N+2}{N} (z^\top\Lam U\Lam U^\top \Lam)_\ell \nonumber
\\ &+\frac{1}{N}(z^\top\Lam U \Lam)_\ell \tr(U\Lam)+ \frac{1}{N}(z^\top\Lam U \Lam U\Lam)_\ell \nonumber
\\&+\frac{1}{N} (z^\top \Lam)_\ell\tr(\Lam) (Z^\top\Lam Z)+\frac{N+1}{N}(z^\top \Lam^2)_\ell (Z^\top\Lam Z) \nonumber
\\&+\frac{2}{N} 
(vZ^\top\Lam U^\top \Lam)_\ell \tr(\Lam)+\frac{N+1}{N} (vZ^\top\Lam U^\top\Lam^2)_\ell \nonumber
\\ &+\frac{N+1}{N} ( vZ^\top\Lam U^\top\Lam^2)_\ell\nonumber
\\ &+
\frac{1}{N}(vZ^\top\Lam)_\ell \tr(\Lam) \tr (U\Lam)+\frac{2}{N}(vZ^\top\Lam U\Lam)_\ell \tr(\Lam)  \nonumber
\\&+\frac{1}{N}(z^\top \Lam U\Lam U^\top\Lam)_\ell+\frac{1}{N}(z^\top \Lam U^\top\Lam U^\top\Lam)_\ell+ \frac{1}{N}(z^T\Lam U^\top\Lam^\top)_\ell \tr(U\Lam) \nonumber
\\&+\frac{1}{N^2}[2(z^\top \Lam U\Lam)_\ell \tr(U\Lam) + 2(z^\top \Lam U \Lam U\Lam)_\ell+ 2(z^\top \Lam U \Lam U^\top \Lam)_\ell \nonumber
\\ 
&+ 2(z^\top \Lam U^\top \Lam)_\ell \tr(U\Lam) + 2(z^\top \Lam U^\top  \Lam U\Lam)_\ell+2(z^\top \Lam U^\top \Lam U^\top\Lam)_\ell \nonumber
\\ 
&+ (z^\top \Lam)_\ell \tr(U\Lam U\Lam ) +(z^\top \Lam)_\ell \tr(U\Lam U^\top \Lam ) +(z^\top \Lam)_\ell \tr(U\Lam)\tr(U\Lam)] \nonumber
\\ &-Z^\top\Lam^2.
\end{align*}

To see this, let us calculate each term. First,
\begin{align*}
&\ex\left(\sum_{n=1}^N\sum_{i=1}^d \sum_{j=1}^d\sum_{k=1}^dz_iU_{jk}x^{(i)}_n x_n^{(j)}x^{(k)}_q \right)\left(\sum_{n=1}^N\sum_{j=1}^d\sum_{k=1}^dU_{jk}x^{(\ell)}_n x_n^{(j)}x^{(k)}_q\right)\\=&\ex\left(\sum_{n=1}^N\sum_{i=1}^d \sum_{j=1}^d\sum_{k=1}^d\sum_{n'=1}^N\sum_{j'=1}^d\sum_{k'=1}^dz_iU_{jk}x^{(i)}_n x_n^{(j)}x^{(k)}_q U_{j'k'}x^{(\ell)}_{n'} x_{n'}^{(j')}x^{(k')}_q\right)
\\=&\ex\left(\sum_{n=1}^N\sum_{i=1}^d \sum_{j=1}^d\sum_{k=1}^d\sum_{n'=1}^N\sum_{j'=1}^d\sum_{k'=1}^dz_iU_{jk}x^{(i)}_n x_n^{(j)} U_{j'k'}x^{(\ell)}_{n'} x_{n'}^{(j')}\Lam_{kk'}\right)
\\=&\ex\left(\sum_{n=1}^N\sum_{i=1}^d \sum_{j=1}^d\sum_{k=1}^d\sum_{j'=1}^d\sum_{k'=1}^dz_iU_{jk}x^{(i)}_n x_n^{(j)} U_{j'k'}x^{(\ell)}_{n} x_{n}^{(j')}\Lam_{kk'}\right)
\\&+\ex\left(\sum_{n=1}^N\sum_{i=1}^d \sum_{j=1}^d\sum_{k=1}^d\sum_{n'=1,n'\neq n}^N\sum_{j'=1}^d\sum_{k'=1}^dz_iU_{jk}x^{(i)}_n x_n^{(j)} U_{j'k'}x^{(\ell)}_{n'} x_{n'}^{(j')}\Lam_{kk'}\right)
\\=&N\ex\left(\sum_{i=1}^d \sum_{j=1}^d\sum_{k=1}^d\sum_{j'=1}^d\sum_{k'=1}^dz_iU_{jk} U_{j'k'}(\Lam_{ij}\Lam_{\ell j'}+\Lam_{i\ell}\lam_{jj'}+\Lam_{ij'}\Lam_{j\ell})\Lam_{kk'}\right)
\\&+N(N-1)\ex\left(\sum_{i=1}^d \sum_{j=1}^d\sum_{k=1}^d\sum_{j'=1}^d\sum_{k'=1}^dz_iU_{jk}\Lam_{ij} U_{j'k'}\Lam_{\ell j'}\Lam_{kk'}\right)
\\=&N(z^\top\Lam U\Lam U^\top \Lam^\top)_\ell+N (z^\top\Lam)_\ell \tr(U\Lam U^\top \Lam^\top)+
\\&\quad 
+N(z^\top\Lam U \Lam^\top U^\top \Lam)_\ell+ N(N-1) (z^\top\Lam U\Lam U^\top \Lam^\top)_\ell
\\=&N (z^\top\Lam)_\ell\tr(U\Lam U^\top \Lam^\top)+ N(N+1) (z^\top\Lam U\Lam U^\top \Lam)_\ell.
\end{align*}
Again, this is the result of taking expectation the multivariate normal random variables separately, with the help of the Isserlis theorem. 

\begin{align*}
&\ex\left(\sum_{n=1}^N\sum_{i=1}^d \sum_{j=1}^d\sum_{k=1}^dz_iU_{jk}x^{(i)}_n x_n^{(j)}x^{(k)}_q \right)\left(\sum_{n=1}^N\sum_{j=1}^d\sum_{k=1}^dZ_kw_jx^{(\ell)}_n x_n^{(j)}x^{(k)}_q\right)=0,
\end{align*}
due to the symmetry of random variable $w$.

\begin{align*}
&\ex\left(\sum_{n=1}^N\sum_{i=1}^d \sum_{j=1}^d\sum_{k=1}^dz_iU_{jk}x^{(i)}_n x_n^{(j)}x^{(k)}_q \right)\left(\sum_{j=1}^d\sum_{k=1}^dU_{jk}x^{(\ell)}_q x_q^{(j)}x^{(k)}_q\right)
\\&= \ex\left(\sum_{n=1}^N\sum_{i=1}^d \sum_{j=1}^d\sum_{k=1}^d\sum_{j'=1}^d\sum_{k'=1}^dz_iU_{jk}x^{(i)}_n x_n^{(j)}x^{(k)}_q U_{j'k'}x^{(\ell)}_q x_q^{(j')}x^{(k')}_q\right)
\\&= N\ex\left(\sum_{i=1}^d \sum_{j=1}^d\sum_{k=1}^d\sum_{j'=1}^d\sum_{k'=1}^dz_iU_{jk}U_{j'k'}\Lam_{ij}x^{(k)}_q x^{(\ell)}_q x_q^{(j')}x^{(k')}_q\right)
\\&= N\ex\left(\sum_{i=1}^d \sum_{j=1}^d\sum_{k=1}^d\sum_{j'=1}^d\sum_{k'=1}^dz_iU_{jk}U_{j'k'}\Lam_{ij}(\Lam_{k\ell}\Lam_{j'k'}+\Lam_{kj'}\Lam_{\ell k'}+\Lam_{kk'}\Lam_{\ell j'})\right)
\\&= N(z^\top\Lam U \Lam)_\ell \tr(U\Lam^\top)+ N(z^\top\Lam U \Lam U\Lam^\top)_\ell+N(z^\top\Lam U\Lam U^\top \Lam^\top)_\ell.
\end{align*}

\begin{align*}
&\ex\left(\sum_{n=1}^N\sum_{i=1}^d \sum_{j=1}^d\sum_{k=1}^dz_iZ_kw_jx^{(i)}_n x_n^{(j)}x^{(k)}_q \right)\left(\sum_{n=1}^N\sum_{j=1}^d\sum_{k=1}^dU_{jk}x^{(\ell)}_n x_n^{(j)}x^{(k)}_q\right)=0.
\end{align*}

\begin{align*}
&\ex\left(\sum_{n=1}^N\sum_{i=1}^d \sum_{j=1}^d\sum_{k=1}^dz_iZ_kw_jx^{(i)}_n x_n^{(j)}x^{(k)}_q \right)\left(\sum_{n=1}^N\sum_{j=1}^d\sum_{k=1}^dZ_kw_jx^{(\ell)}_n x_n^{(j)}x^{(k)}_q\right)
\\=& \ex\left(\sum_{n=1}^N\sum_{i=1}^d \sum_{j=1}^d\sum_{k=1}^dz_iZ_kw_jx^{(i)}_n x_n^{(j)}x^{(k)}_q \sum_{n'=1}^N\sum_{j'=1}^d\sum_{k'=1}^dZ_{k'}w_{j'}x^{(\ell)}_{n'} x_{n'}^{(j')}x^{(k')}_q\right)
\\=& \ex\left(\sum_{n=1}^N\sum_{i=1}^d \sum_{j=1}^d\sum_{k=1}^dz_iZ_kx^{(i)}_n x_n^{(j)}x^{(k)}_q \sum_{n'=1}^N\sum_{k'=1}^dZ_{k'}x^{(\ell)}_{n'} x_{n'}^{(j)}x^{(k')}_q\right)
\\=& \ex\left(\sum_{n=1}^N\sum_{i=1}^d \sum_{j=1}^d\sum_{k=1}^d\sum_{n'=1}^N\sum_{k'=1}^dz_iZ_k Z_{k'}x^{(i)}_n x_n^{(j)}x^{(\ell)}_{n'} x_{n'}^{(j)}\Lam_{kk'}\right)
\\= & \ex\left(\sum_{n=1}^N\sum_{i=1}^d \sum_{j=1}^d\sum_{k=1}^d\sum_{k'=1}^dz_iZ_k Z_{k'}x^{(i)}_n x_n^{(j)}x^{(\ell)}_{n} x_{n}^{(j)}\Lam_{kk'}\right)
\\&+ \ex\left(\sum_{n=1}^N\sum_{i=1}^d \sum_{j=1}^d\sum_{k=1}^d\sum_{n'=1, n'\neq n}^N\sum_{k'=1}^dz_iZ_k Z_{k'}x^{(i)}_n x_n^{(j)}x^{(\ell)}_{n'} x_{n'}^{(j)}\Lam_{kk'}\right)
\\=& N\ex\left(\sum_{i=1}^d \sum_{j=1}^d\sum_{k=1}^d\sum_{n'=1}^N\sum_{k'=1}^dz_iZ_k Z_{k'}(2\Lam_{ij}\Lam_{\ell j}+ \Lam_{i\ell}\Lam_{j j})\Lam_{kk'}\right) +
\\ &\quad + N(N-1) \sum_{i=1}^d \sum_{j=1}^d\sum_{k=1}^d\sum_{k'=1}^dz_iZ_k Z_{k'} \Lam_{ij}\Lam_{\ell j}\Lam_{kk'}
\\=& 2(z^\top \Lam^2)_\ell (Z^\top\Lam Z)+N (z^\top \Lam)_\ell\tr(\Lam) (Z^\top\Lam Z)+N(N-1)(z^\top \Lam^2)_\ell (Z^\top\Lam Z)
\\=& N (z^\top \Lam)_\ell\tr(\Lam) (Z^\top\Lam Z)+N(N+1)(z^\top \Lam^2)_\ell (Z^\top\Lam Z).
\end{align*}

\begin{align*}
&\ex\left(\sum_{n=1}^N\sum_{i=1}^d \sum_{j=1}^d\sum_{k=1}^dz_iZ_kw_jx^{(i)}_n x_n^{(j)}x^{(k)}_q \right)\left(\sum_{j=1}^d\sum_{k=1}^dU_{jk}x^{(\ell)}_qx_q^{(j)}x^{(k)}_q\right)=0.
\end{align*}

\begin{align*}
&\ex\left(\sum_{n=1}^N\sum_{i=1}^d \sum_{j=1}^d\sum_{k=1}^dvw_iU_{jk}x^{(i)}_n x_n^{(j)}x^{(k)}_q \right)\left(\sum_{n=1}^N\sum_{j=1}^d\sum_{k=1}^dU_{jk}x^{(\ell)}_n x_n^{(j)}x^{(k)}_q\right)=0.
\end{align*}

\begin{align*}
&\ex\left(\sum_{n=1}^N\sum_{i=1}^d \sum_{j=1}^d\sum_{k=1}^dvw_iU_{jk}x^{(i)}_n x_n^{(j)}x^{(k)}_q \right)\left(\sum_{n=1}^N\sum_{j=1}^d\sum_{k=1}^dZ_kw_jx^{(\ell)}_n x_n^{(j)}x^{(k)}_q\right)
\\=&\ex\left(\sum_{n=1}^N\sum_{i=1}^d \sum_{j=1}^d\sum_{k=1}^dvw_iU_{jk}x^{(i)}_n x_n^{(j)}x^{(k)}_q \sum_{n'=1}^N\sum_{j'=1}^d\sum_{k'=1}^dZ_{k'}w_{j'}x^{(\ell)}_{n'} x_{n'}^{(j')}x^{(k')}_q\right)
\\=&\ex\left(\sum_{n=1}^N\sum_{i=1}^d \sum_{j=1}^d\sum_{k=1}^dvU_{jk}x^{(i)}_n x_n^{(j)}x^{(k)}_q \sum_{n'=1}^N\sum_{k'=1}^dZ_{k'}x^{(\ell)}_{n'} x_{n'}^{(i)}x^{(k')}_q\right)
\\=&\ex\left(\sum_{n=1}^N\sum_{i=1}^d \sum_{j=1}^d\sum_{k=1}^dvU_{jk}x^{(i)}_n x_n^{(j)} \sum_{n'=1}^N\sum_{k'=1}^dZ_{k'}x^{(\ell)}_{n'} x_{n'}^{(i)}\Lam_{kk'}\right)
\\=&\ex\left(\sum_{n=1}^N\sum_{i=1}^d \sum_{j=1}^d\sum_{k=1}^d\sum_{k'=1}^dvU_{jk}x^{(i)}_n x_n^{(j)} Z_{k'}x^{(\ell)}_{n} x_{n}^{(i)}\Lam_{kk'}\right)
\\&+\ex\left(\sum_{n=1}^N\sum_{i=1}^d \sum_{j=1}^d\sum_{k=1}^d\sum_{n'=1,n'\neq n}^N\sum_{k'=1}^dvU_{jk}x^{(i)}_n x_n^{(j)} Z_{k'}x^{(\ell)}_{n'} x_{n'}^{(i)}\Lam_{kk'}\right)
\\=&N\ex\left(\sum_{i=1}^d \sum_{j=1}^d\sum_{k=1}^d\sum_{k'=1}^dvU_{jk}Z_{k'}(2\Lam_{ij}\Lam_{\ell i}+\Lam_{ii}\Lam_{\ell j})\Lam_{kk'}\right)
\\&+N(N-1)\ex\left(\sum_{i=1}^d \sum_{j=1}^d\sum_{k=1}^d\sum_{k'=1}^dvU_{jk}\Lam_{ij} \Lam_{\ell i}Z_{k'}\Lam_{kk'}\right)
\\=&2N(vZ^\top\Lam U^\top \Lam^2)_\ell+ N(vZ^\top\Lam U^\top \Lam)_\ell \tr(\Lam)+N(N-1) (vZ^\top\Lam U\Lam^2)_\ell
\\=& N(vZ^\top\Lam U^\top \Lam)_\ell \tr(\Lam)+N(N+1) (vZ^\top\Lam U^\top\Lam^2)_\ell.
\end{align*}

\begin{align*}
&\ex\left(\sum_{n=1}^N\sum_{i=1}^d \sum_{j=1}^d\sum_{k=1}^dvw_iU_{jk}x^{(i)}_n x_n^{(j)}x^{(k)}_q \right)\left(\sum_{j=1}^d\sum_{k=1}^dU_{jk}x^{(\ell)}_qx_q^{(j)}x^{(k)}_q\right)=0.
\end{align*}

\begin{align*}
&\ex\left(\sum_{n=1}^N\sum_{i=1}^d \sum_{j=1}^d\sum_{k=1}^dvZ_kw_iw_jx^{(i)}_n x_n^{(j)}x^{(k)}_q \right)\left(\sum_{n=1}^N\sum_{j=1}^d\sum_{k=1}^dU_{jk}x^{(\ell)}_n x_n^{(j)}x^{(k)}_q\right)
\\=&\ex\left(\sum_{n=1}^N\sum_{i=1}^d \sum_{j=1}^d\sum_{k=1}^dvZ_kw_iw_jx^{(i)}_n x_n^{(j)}x^{(k)}_q \sum_{n'=1}^N\sum_{j'=1}^d\sum_{k'=1}^dU_{j'k'}x^{(\ell)}_{n'} x_{n'}^{(j')}x^{(k')}_q\right)
\\=&\ex\left(\sum_{n=1}^N\sum_{i=1}^d \sum_{k=1}^dvZ_kx^{(i)}_n x_n^{(i)}x^{(k)}_q \sum_{n'=1}^N\sum_{j'=1}^d\sum_{k'=1}^dU_{j'k'}x^{(\ell)}_{n'} x_{n'}^{(j')}x^{(k')}_q\right)
\\=&\ex\left(\sum_{n=1}^N\sum_{i=1}^d \sum_{k=1}^d\sum_{n'=1}^N\sum_{j'=1}^d\sum_{k'=1}^dvZ_kx^{(i)}_n x_n^{(i)}U_{j'k'}x^{(\ell)}_{n'} x_{n'}^{(j')}\Lam_{kk'}\right)
\\=&N(N+1)( vZ^\top\Lam U^\top\Lam^2)_\ell+ N(vZ^\top U^\top\Lam)_\ell \tr(\Lam).
\end{align*}

\begin{align*}
&\ex\left(\sum_{n=1}^N\sum_{i=1}^d \sum_{j=1}^d\sum_{k=1}^dvZ_kw_iw_jx^{(i)}_n x_n^{(j)}x^{(k)}_q \right)\left(\sum_{n=1}^N\sum_{j=1}^d\sum_{k=1}^dZ_kw_jx^{(\ell)}_nx_n^{(j)}x^{(k)}_q\right)=0
\end{align*}

\begin{align*}
&\ex\left(\sum_{n=1}^N\sum_{i=1}^d \sum_{j=1}^d\sum_{k=1}^dvZ_kw_iw_jx^{(i)}_n x_n^{(j)}x^{(k)}_q \right)\left(\sum_{j=1}^d\sum_{k=1}^dU_{jk}x^{(\ell)}_q x_q^{(j)}x^{(k)}_q\right)
\\=&\ex\left(\sum_{n=1}^N\sum_{i=1}^d \sum_{j=1}^d\sum_{k=1}^dvZ_kw_iw_jx^{(i)}_n x_n^{(j)}x^{(k)}_q \sum_{j'=1}^d\sum_{k'=1}^dU_{j'k'}x^{(\ell)}_q x_q^{(j')}x^{(k')}_q\right)
\\=&\ex\left(\sum_{n=1}^N\sum_{i=1}^d \sum_{k=1}^dvZ_kx^{(i)}_n x_n^{(i)}x^{(k)}_q \sum_{j'=1}^d\sum_{k'=1}^dU_{j'k'}x^{(\ell)}_q x_q^{(j')}x^{(k')}_q\right)
\\=&N\ex\left(\sum_{i=1}^d \sum_{k=1}^d\sum_{j'=1}^d\sum_{k'=1}^dvZ_k\Lam_{ii} U_{j'k'}x^{(k)}_qx^{(\ell)}_q x_q^{(j')}x^{(k')}_q\right)
\\=&N\sum_{i=1}^d \sum_{k=1}^d\sum_{j'=1}^d\sum_{k'=1}^dvZ_k\Lam_{ii} U_{j'k'}
(\Lam_{k\ell}\Lam_{j'k'}+\Lam_{kj'}\Lam_{\ell k'}+\Lam_{kk'}\Lam_{\ell j'})
\\=&N(vZ^\top\Lam)_\ell \tr(\Lam) \tr (U\Lam)+N(vZ^\top\Lam U\Lam)_\ell \tr(\Lam) +N(vZ^\top\Lam U^\top \Lam)_\ell \tr(\Lam)
\end{align*}

\newpage

\begin{align*}
&\ex\left(\sum_{i=1}^d \sum_{j=1}^d\sum_{k=1}^dz_iU_{jk}x^{(i)}_q x_q^{(j)}x^{(k)}_q \right)\left(\sum_{n=1}^N\sum_{j=1}^d\sum_{k=1}^dU_{jk}x^{(\ell)}_n x_n^{(j)}x^{(k)}_q\right)
\\=&\ex\left(\sum_{i=1}^d \sum_{j=1}^d\sum_{k=1}^dz_iU_{jk}x^{(i)}_q x_q^{(j)}x^{(k)}_q \sum_{n=1}^N\sum_{j'=1}^d\sum_{k'=1}^dU_{j'k'}x^{(\ell)}_n x_n^{(j')}x^{(k')}_q\right)
\\=&N\ex\left(\sum_{i=1}^d \sum_{j=1}^d\sum_{k=1}^dz_iU_{jk}x^{(i)}_q x_q^{(j)}x^{(k)}_q \sum_{j'=1}^d\sum_{k'=1}^dU_{j'k'}\Lam_{\ell j'}x^{(k')}_q\right)
\\=&N\ex\left(\sum_{i=1}^d \sum_{j=1}^d\sum_{k=1}^d\sum_{j'=1}^d\sum_{k'=1}^dz_iU_{jk} U_{j'k'}\Lam_{\ell j'}x^{(i)}_q x_q^{(j)}x^{(k)}_qx^{(k')}_q\right)
\\=&N\sum_{i=1}^d \sum_{j=1}^d\sum_{k=1}^d\sum_{j'=1}^d\sum_{k'=1}^dz_iU_{jk} U_{j'k'}\Lam_{\ell j'}(\Lam_{ij}\Lam_{kk'}+\Lam_{ik}\Lam_{jk'}+\Lam_{ik'}\Lam_{kj})
\\=&N(z^\top \Lam U\Lam U^\top\Lam)_\ell+N(z^\top \Lam U^\top\Lam U^\top\Lam)_\ell+ N(z^T\Lam U^\top\Lam^\top)_\ell \tr(U\Lam).
\end{align*}

\begin{align*}
&\ex\left(\sum_{i=1}^d \sum_{j=1}^d\sum_{k=1}^dz_iU_{jk}x^{(i)}_q x_q^{(j)}x^{(k)}_q \right)\left(\sum_{n=1}^N\sum_{j=1}^d\sum_{k=1}^dZ_kw_jx^{(\ell)}_nx_n^{(j)}x^{(k)}_q\right)=0.
\end{align*}

\begin{align*}
&\ex\left(\sum_{i=1}^d \sum_{j=1}^d\sum_{k=1}^dz_iU_{jk}x^{(i)}_q x_q^{(j)}x^{(k)}_q \right)\left(\sum_{j=1}^d\sum_{k=1}^dU_{jk}x^{(\ell)}_q x_q^{(j)}x^{(k)}_q\right)
\\=
&\ex\left(\sum_{i=1}^d \sum_{j=1}^d\sum_{k=1}^dz_iU_{jk}x^{(i)}_q x_q^{(j)}x^{(k)}_q\sum_{j'=1}^d\sum_{k'=1}^dU_{j'k'}x^{(\ell)}_q x_q^{(j')}x^{(k')}_q\right)
\\=
&\ex\left(\sum_{i=1}^d \sum_{j=1}^d\sum_{k=1}^d\sum_{j'=1}^d\sum_{k'=1}^dz_iU_{jk}U_{j'k'}x^{(i)}_q x_q^{(j)}x^{(k)}_qx^{(\ell)}_q x_q^{(j')}x^{(k')}_q\right)
\\ &=\sum_{i=1}^d \sum_{j=1}^d\sum_{k=1}^d\sum_{j'=1}^d\sum_{k'=1}^dz_iU_{jk}U_{j'k'}
[\Lam_{ij}(\Lam_{k\ell}\Lam_{j'k'}+\Lam_{kj'}\Lam_{\ell k'}+\Lam_{kk'}\Lam_{j'\ell})
\\ &+\Lam_{ik}(\Lam_{j\ell}\Lam_{j'k'}+\Lam_{jj'}\Lam_{\ell k'}+\Lam_{jk'}\Lam_{j'\ell})+ \Lam_{ik'}(\Lam_{j\ell}\Lam_{j'k}+\Lam_{jj'}\Lam_{\ell k}+\Lam_{jk}\Lam_{j'\ell})
\\ &+\Lam_{i\ell}(\Lam_{jk'}\Lam_{j'k}+\Lam_{jj'}\Lam_{k' k}+\Lam_{jk}\Lam_{j'k'})+\Lam_{ij'}(\Lam_{jk'}\Lam_{\ell k}+\Lam_{j\ell}\Lam_{k' k} +\Lam_{jk}\Lam_{\ell k'})]
\\
&
= (z^\top \Lam U\Lam)_\ell \tr(U\Lam) + (z^\top \Lam U \Lam U\Lam)_\ell+ (z^\top \Lam U \Lam U^\top \Lam)_\ell
\\ 
&+ (z^\top \Lam U^\top \Lam)_\ell \tr(U\Lam) + (z^\top \Lam U^\top  \Lam U\Lam)_\ell+(z^\top \Lam U^\top \Lam U^\top\Lam)_\ell 
\\ 
&+ (z^\top \Lam U^\top \Lam U^\top\Lam)_\ell +(z^\top \Lam U^\top \Lam U\Lam)_\ell +(z^\top \Lam U^\top \Lam)_\ell \tr(U\Lam)
\\ 
&+ (z^\top \Lam)_\ell \tr(U\Lam U\Lam ) +(z^\top \Lam)_\ell \tr(U\Lam U^\top \Lam ) +(z^\top \Lam)_\ell \tr(U\Lam)\tr(U\Lam)
\\ 
& + (z^\top \Lam U \Lam U \Lam)_\ell + (z^\top \Lam U \Lam U^\top\Lam)_\ell+ (z^\top \Lam U \Lam)_\ell \tr(U\Lam).
\end{align*}
Here, The Isserlis theorem is applied to the sixth moment of $x_q$.

\begin{align*}
&\ex\left[\sum_{m=1}^d w_mx_q^{(m)} \left(\sum_{n=1}^N \sum_{j=1}^d\sum_{k=1}^d(U_{jk}+Z_kw_j)x^{(\ell)}_n x_n^{(j)}x^{(k)}_q +\sum_{j=1}^d \sum_{k=1}^d U_{jk} x_q^{(j)}x_q^{(\ell)}  x_q^{(k)}\right)\right]\\ 
=&\ex\left[\sum_{m=1}^d \sum_{n=1}^N \sum_{j=1}^d\sum_{k=1}^dZ_kw_j w_mx_q^{(m)}x^{(\ell)}_n x_n^{(j)}x^{(k)}_q \right]
\\=&\ex\left[ \sum_{n=1}^N \sum_{j=1}^d\sum_{k=1}^dZ_kx_q^{(j)}x^{(\ell)}_n x_n^{(j)}x^{(k)}_q \right]
\\=& N\sum_{j=1}^d\sum_{k=1}^dZ_k\Lam_{\ell j}\Lam_{kj}=N(Z^\top\Lam^2)_\ell.
\end{align*}

\subsection{Detailed Calculations of $\frac{\partial L}{\partial Z}$}
\label{sec:DZ}

Recall that, For any $\ell=1,2,\ldots, d$, 
\begin{align*}
\frac{\partial L}{\partial Z_\ell}= &\ex\left[\frac{1}{N}\sum_{n=1}^N\sum_{i=1}^d \sum_{j=1}^d\sum_{k=1}^d(z_i  U_{jk}+z_iZ_k w_j+vw_iU_{jk}+vZ_kw_iw_j)x^{(i)}_n x_n^{(j)}x^{(k)}_q \right.\\ &\left. +\frac{1}{N}\sum_{i=1}^d\sum_{j=1}^d \sum_{k=1}^d (z_iU_{jk}) x_q^{(j)}x_q^{(i)}  x_q^{(k)}-\sum_{m=1}^d w_mx_q^{(m)}\right]\\  &\times \left(\frac{1}{N}\sum_{n=1}^N \sum_{i=1}^d\sum_{j=1}^d(z_iw_j+vw_iw_j)x^{(i)}_n x_n^{(j)}x^{(\ell)}_q \right).
\end{align*}

Calculations below gives us, 
\begin{align*}
\left(\frac{\partial L}{\partial Z}\right)_\ell = & 
\frac{2}{N} (vz^T\Lam^2U\Lam)_\ell+(vz^\top \Lam U\Lam)_\ell \tr(\Lam)\nonumber
    \\ &+ \frac{1}{N}(z^\top\Lam z Z^\top\Lam)_\ell\tr(\Lam)+ \frac{N+1}{N} (z^\top \Lam^2 z Z^\top\Lam)_\ell \nonumber
    \\& +\frac{2}{N}(vz^\top\Lam U\Lam)_\ell \tr(\Lam)+ \frac{N+1}{N}(vz^\top\Lam^2U\Lam)_\ell \nonumber
    \\ &+\frac{N+4}{N}(v^2 Z^\top\Lam)_\ell (\tr(\Lam))^2+2\frac{N+1}{N}(v^2Z^\top \Lam)_\ell \tr(\Lam^2) \nonumber
    \\&+ \frac{1}{N}(vz^\top \Lam U^\top \Lam)_\ell\tr(\Lam)+ \frac{1}{N}(vz^\top\Lam)_\ell\tr(U\Lam)\tr(\Lam)-(z^\top\Lam^2)_\ell.
\end{align*}
Below, we will collect all the calculations, with the same arguments that have already been introduced. 
\begin{align*}
&\ex\left(\sum_{n=1}^N\sum_{i=1}^d \sum_{j=1}^d\sum_{k=1}^dz_iU_{jk}x^{(i)}_n x_n^{(j)}x^{(k)}_q \right)\left(\sum_{n=1}^N\sum_{i=1}^d\sum_{j=1}^dz_iw_jx^{(i)}_nx_n^{(j)}x^{(\ell)}_q\right)=0. 
\end{align*}

\begin{align*}
&\ex\left(\sum_{n=1}^N\sum_{i=1}^d \sum_{j=1}^d\sum_{k=1}^dz_iU_{jk}x^{(i)}_n x_n^{(j)}x^{(k)}_q \right)\left(\sum_{n=1}^N\sum_{i=1}^d\sum_{j=1}^dvw_iw_jx^{(i)}_nx_n^{(j)}x^{(\ell)}_q\right)\\=&\ex\left(\sum_{n=1}^N\sum_{i=1}^d \sum_{j=1}^d\sum_{k=1}^dz_iU_{jk}x^{(i)}_n x_n^{(j)}x^{(k)}_q \sum_{n'=1}^N\sum_{i'=1}^d\sum_{j'=1}^dvw_{i'}w_{j'}x^{(i')}_{n'}x_{n'}^{(j')}x^{(\ell)}_q\right)
\\=&\ex\left(\sum_{n=1}^N\sum_{i=1}^d \sum_{j=1}^d\sum_{k=1}^dz_iU_{jk}x^{(i)}_n x_n^{(j)}x^{(k)}_q \sum_{n'=1}^N\sum_{i'=1}^dvx^{(i')}_{n'}x_{n'}^{(i')}x^{(\ell)}_q\right)
\\=&\ex\left(\sum_{n=1}^N\sum_{i=1}^d \sum_{j=1}^d\sum_{k=1}^d\sum_{n'=1}^N\sum_{i'=1}^dz_iU_{jk}x^{(i)}_n x_n^{(j)} vx^{(i')}_{n'}x_{n'}^{(i')}\Lam_{k\ell}\right)
\\=&\ex\left(\sum_{n=1}^N\sum_{i=1}^d \sum_{j=1}^d\sum_{k=1}^d\sum_{i'=1}^dz_iU_{jk}x^{(i)}_n x_n^{(j)} vx^{(i')}_{n}x_{n}^{(i')}\Lam_{k\ell}\right)\\&+\ex\left(\sum_{n=1}^N\sum_{i=1}^d \sum_{j=1}^d\sum_{k=1}^d\sum_{n'=1, n'\neq n}^N\sum_{i'=1}^dz_iU_{jk}x^{(i)}_n x_n^{(j)} vx^{(i')}_{n'}x_{n'}^{(i')}\Lam_{k\ell}\right)
\\=&N\sum_{i=1}^d \sum_{j=1}^d\sum_{k=1}^d\sum_{i'=1}^dz_iU_{jk} v\Lam_{k\ell}(\Lam_{ij}\Lam_{i'i'}+2\Lam_{ii'}\Lam_{ji'})\\&+N(N-1)\sum_{i=1}^d \sum_{j=1}^d\sum_{k=1}^d\sum_{i'=1}^dz_iU_{jk}\Lam_{ij}\Lam_{i'i'} v\Lam_{k\ell}
\\&=N(vz^\top \Lam U\Lam)_\ell \tr(\Lam) + 2N(vz^T\Lam^2U\Lam)_\ell+N(N-1)(vz^\top \Lam U\Lam)_\ell \tr(\Lam)
\\&=2N(vz^T\Lam^2U\Lam)_\ell+N^2(vz^\top \Lam U\Lam)_\ell \tr(\Lam).
\end{align*}

\begin{align*}
&\ex\left(\sum_{n=1}^N\sum_{i=1}^d \sum_{j=1}^d\sum_{k=1}^dz_iZ_kw_jx^{(i)}_n x_n^{(j)}x^{(k)}_q \right)\left(\sum_{n=1}^N\sum_{i=1}^d\sum_{j=1}^dz_iw_jx^{(i)}_nx_n^{(j)}x^{(\ell)}_q\right)
\\=&\ex\left(\sum_{n=1}^N\sum_{i=1}^d \sum_{j=1}^d\sum_{k=1}^dz_iZ_kw_jx^{(i)}_n x_n^{(j)}x^{(k)}_q \sum_{n'=1}^N\sum_{i'=1}^d\sum_{j'=1}^dz_{i'}w_{j'}x^{(i')}_{n'}x_{n'}^{(j')}x^{(\ell)}_q\right)
\\=&\ex\left(\sum_{n=1}^N\sum_{i=1}^d \sum_{j=1}^d\sum_{k=1}^dz_iZ_kx^{(i)}_n x_n^{(j)}x^{(k)}_q \sum_{n'=1}^N\sum_{i'=1}^dz_{i'}x^{(i')}_{n'}x_{n'}^{(j)}x^{(\ell)}_q\right)
\\=&\ex\left(\sum_{n=1}^N\sum_{i=1}^d \sum_{j=1}^d\sum_{k=1}^d\sum_{n'=1}^N\sum_{i'=1}^dz_iz_{i'}Z_kx^{(i)}_n x_n^{(j)} x^{(i')}_{n'}x_{n'}^{(j)}\Lam_{k\ell}\right)
\\=&\ex\left(\sum_{n=1}^N\sum_{i=1}^d \sum_{j=1}^d\sum_{k=1}^d\sum_{i'=1}^dz_iz_{i'}Z_kx^{(i)}_n x_n^{(j)} x^{(i')}_{n}x_{n}^{(j)}\Lam_{k\ell}\right)
\\&+\ex\left(\sum_{n=1}^N\sum_{i=1}^d \sum_{j=1}^d\sum_{k=1}^d\sum_{n'=1, n'\neq n}^N\sum_{i'=1}^dz_iz_{i'}Z_kx^{(i)}_n x_n^{(j)} x^{(i')}_{n'}x_{n'}^{(j)}\Lam_{k\ell}\right)
\\=&N\sum_{i=1}^d \sum_{j=1}^d\sum_{k=1}^d\sum_{i'=1}^dz_iz_{i'}Z_k\Lam_{k\ell}(2\Lam_{ij}\Lam_{i'j}+\Lam_{ii'}\Lam_{jj})
\\&+N(N-1)\sum_{i=1}^d \sum_{j=1}^d\sum_{k=1}^d\sum_{i'=1}^dz_iz_{i'}Z_k \Lam_{ij}\Lam_{i'j}\Lam_{k\ell}
\\=&N\sum_{i=1}^d \sum_{j=1}^d\sum_{k=1}^d\sum_{i'=1}^dz_iz_{i'}Z_k\Lam_{k\ell}\Lam_{ii'}\Lam_{jj}
\\&+N(N+1)\sum_{i=1}^d \sum_{j=1}^d\sum_{k=1}^d\sum_{i'=1}^dz_iz_{i'}Z_k\Lam_{ij}\Lam_{i'j}\Lam_{k\ell}
\\=& N(z^\top\Lam z Z^\top\Lam)_\ell\tr(\Lam)+ N(N+1) (z^\top \Lam^2 z Z^\top\Lam)_\ell.
\end{align*}

\begin{align*}
&\ex\left(\sum_{n=1}^N\sum_{i=1}^d \sum_{j=1}^d\sum_{k=1}^dz_iZ_kw_jx^{(i)}_n x_n^{(j)}x^{(k)}_q \right)\left(\sum_{n=1}^N\sum_{i=1}^d\sum_{j=1}^dvw_iw_jx^{(i)}_nx_n^{(j)}x^{(\ell)}_q\right)=0.
\end{align*}

\begin{align*}
&\ex\left(\sum_{n=1}^N\sum_{i=1}^d \sum_{j=1}^d\sum_{k=1}^dvw_iU_{jk}x^{(i)}_n x_n^{(j)}x^{(k)}_q \right)\left(\sum_{n=1}^N\sum_{i=1}^d\sum_{j=1}^dz_iw_jx^{(i)}_nx_n^{(j)}x^{(\ell)}_q\right)
\\=&\ex\left(\sum_{n=1}^N\sum_{i=1}^d \sum_{j=1}^d\sum_{k=1}^dvw_iU_{jk}x^{(i)}_n x_n^{(j)}x^{(k)}_q \sum_{n'=1}^N\sum_{i'=1}^d\sum_{j'=1}^dz_{i'}w_{j'}x^{(i')}_{n'}x_{n'}^{(j')}x^{(\ell)}_q\right)
\\=&\ex\left(\sum_{n=1}^N\sum_{i=1}^d \sum_{j=1}^d\sum_{k=1}^dvU_{jk}x^{(i)}_n x_n^{(j)}x^{(k)}_q \sum_{n'=1}^N\sum_{i'=1}^dz_{i'}x^{(i')}_{n'}x_{n'}^{(i)}x^{(\ell)}_q\right)
\\=&\ex\left(\sum_{n=1}^N\sum_{i=1}^d \sum_{j=1}^d\sum_{k=1}^d\sum_{n'=1}^N\sum_{i'=1}^dvU_{jk}x^{(i)}_n x_n^{(j)} z_{i'}x^{(i')}_{n'}x_{n'}^{(i)}\Lam_{k\ell}\right)
\\=&\ex\left(\sum_{n=1}^N\sum_{i=1}^d \sum_{j=1}^d\sum_{k=1}^d\sum_{i'=1}^dvU_{jk}x^{(i)}_n x_n^{(j)} z_{i'}x^{(i')}_{n}x_{n}^{(i)}\Lam_{k\ell}\right)
\\&+\ex\left(\sum_{n=1}^N\sum_{i=1}^d \sum_{j=1}^d\sum_{k=1}^d\sum_{n'=1,n'\neq n}^N\sum_{i'=1}^dvU_{jk}x^{(i)}_n x_n^{(j)} z_{i'}x^{(i')}_{n'}x_{n'}^{(i)}\Lam_{k\ell}\right)
\\=&N\sum_{i=1}^d \sum_{j=1}^d\sum_{k=1}^d\sum_{i'=1}^dvU_{jk}(2\Lam_{ij}\Lam_{i'i}+\Lam_{ii}\Lam_{i'j}) z_{i'}\Lam_{k\ell}
\\&+N(N-1)\sum_{i=1}^d \sum_{j=1}^d\sum_{k=1}^d\sum_{i'=1}^dvU_{jk}z_{i'}\Lam_{ij}\Lam_{i'i}\Lam_{k\ell} 
\\=&N\sum_{i=1}^d \sum_{j=1}^d\sum_{k=1}^d\sum_{i'=1}^dvU_{jk}\Lam_{ii}\Lam_{i'j}z_{i'}\Lam_{k\ell}+N(N+1)\sum_{i=1}^d \sum_{j=1}^d\sum_{k=1}^d\sum_{i'=1}^dvU_{jk}z_{i'}\Lam_{ij}\Lam_{i'i}\Lam_{k\ell} 
\\=&N(vz^\top\Lam U\Lam)_\ell \tr(\Lam)+ N(N+1)(vz^\top\Lam^2U\Lam)_\ell.
\end{align*}

\begin{align*}
&\ex\left(\sum_{n=1}^N\sum_{i=1}^d \sum_{j=1}^d\sum_{k=1}^dvw_iU_{jk}x^{(i)}_n x_n^{(j)}x^{(k)}_q \right)\left(\sum_{n=1}^N\sum_{i=1}^d\sum_{j=1}^dvw_iw_jx^{(i)}_nx_n^{(j)}x^{(\ell)}_q\right)=0.
\end{align*}

\begin{align*}
&\ex\left(\sum_{n=1}^N\sum_{i=1}^d \sum_{j=1}^d\sum_{k=1}^dvZ_kw_iw_jx^{(i)}_n x_n^{(j)}x^{(k)}_q \right)\left(\sum_{n=1}^N\sum_{i=1}^d\sum_{j=1}^dz_iw_jx^{(i)}_nx_n^{(j)}x^{(\ell)}_q\right)=0
\end{align*}

\begin{align*}
&\ex\left(\sum_{n=1}^N\sum_{i=1}^d \sum_{j=1}^d\sum_{k=1}^dvZ_kw_iw_jx^{(i)}_n x_n^{(j)}x^{(k)}_q \right)\left(\sum_{n=1}^N\sum_{i=1}^d\sum_{j=1}^dvw_iw_jx^{(i)}_nx_n^{(j)}x^{(\ell)}_q\right)\\
=&\ex\left(\sum_{n=1}^N\sum_{i=1}^d \sum_{j=1}^d\sum_{k=1}^dvZ_kw_iw_jx^{(i)}_n x_n^{(j)}x^{(k)}_q \sum_{n'=1}^N\sum_{i'=1}^d\sum_{j'=1}^dvw_{i'}w_{j'}x^{(i')}_{n'}x_{n'}^{(j')}x^{(\ell)}_q\right)\\
=&\ex\left(\sum_{n=1}^N\sum_{i=1}^d \sum_{j=1}^d\sum_{k=1}^d\sum_{n'=1}^N\sum_{i'=1}^d\sum_{j'=1}^dvZ_k (\delta_{i,j}\delta_{i',j'}+\delta_{i,j'}\delta_{i',j}+ \delta_{i,i'}\delta_{j,j'})x^{(i)}_n x_n^{(j)}x^{(k)}_q vx^{(i')}_{n'}x_{n'}^{(j')}x^{(\ell)}_q\right)
\\
=&\ex\left(\sum_{n=1}^N\sum_{i=1}^d \sum_{j=1}^d\sum_{k=1}^d\sum_{n'=1}^N\sum_{i'=1}^d\sum_{j'=1}^dvZ_k \delta_{i,j}\delta_{i',j'}x^{(i)}_n x_n^{(j)}x^{(k)}_q vx^{(i')}_{n'}x_{n'}^{(j')}x^{(\ell)}_q\right)
\\&+\ex\left(\sum_{n=1}^N\sum_{i=1}^d \sum_{j=1}^d\sum_{k=1}^d\sum_{n'=1}^N\sum_{i'=1}^d\sum_{j'=1}^dvZ_k \delta_{i,j'}\delta_{i',j}x^{(i)}_n x_n^{(j)}x^{(k)}_q vx^{(i')}_{n'}x_{n'}^{(j')}x^{(\ell)}_q\right)
\\
&+\ex\left(\sum_{n=1}^N\sum_{i=1}^d \sum_{j=1}^d\sum_{k=1}^d\sum_{n'=1}^N\sum_{i'=1}^d\sum_{j'=1}^dvZ_k \delta_{i,i'}\delta_{j,j'}x^{(i)}_n x_n^{(j)}x^{(k)}_q vx^{(i')}_{n'}x_{n'}^{(j')}x^{(\ell)}_q\right)
\\
=&\underbrace{\ex\left(\sum_{n=1}^N\sum_{i=1}^d \sum_{k=1}^d\sum_{n'=1}^N\sum_{i'=1}^dvZ_k x^{(i)}_n x_n^{(i)}x^{(k)}_q vx^{(i')}_{n'}x_{n'}^{(i')}x^{(\ell)}_q\right)}_{L_1}
\\&+\underbrace{\ex\left(\sum_{n=1}^N\sum_{i=1}^d \sum_{j=1}^d\sum_{k=1}^d\sum_{n'=1}^N\sum_{i'=1}^dvZ_k x^{(i)}_n x_n^{(i')}x^{(k)}_q vx^{(i')}_{n'}x_{n'}^{(i)}x^{(\ell)}_q\right)}_{L_2}
\\
&+\underbrace{\ex\left(\sum_{n=1}^N\sum_{i=1}^d \sum_{j=1}^d\sum_{k=1}^d\sum_{n'=1}^NvZ_k x^{(i)}_n x_n^{(j)}x^{(k)}_q vx^{(i)}_{n'}x_{n'}^{(j)}x^{(\ell)}_q\right)}_{L3}.
\end{align*}
\begin{align*}
L_1=&\ex\left(\sum_{n=1}^N\sum_{i=1}^d \sum_{k=1}^d\sum_{n'=1}^N\sum_{i'=1}^dvZ_k x^{(i)}_n x_n^{(i)}x^{(k)}_q vx^{(i')}_{n'}x_{n'}^{(i')}x^{(\ell)}_q\right)\\=&\ex\left(\sum_{n=1}^N\sum_{i=1}^d \sum_{k=1}^d\sum_{n'=1}^N\sum_{i'=1}^dv^2Z_k x^{(i)}_n x_n^{(i)} x^{(i')}_{n'}x_{n'}^{(i')}\Lam_{k\ell}\right)
\\=&\ex\left(\sum_{n=1}^N\sum_{i=1}^d \sum_{k=1}^d\sum_{i'=1}^dv^2Z_k x^{(i)}_n x_n^{(i)} x^{(i')}_{n}x_{n}^{(i')}\Lam_{k\ell}\right)
\\&+\ex\left(\sum_{n=1}^N\sum_{i=1}^d \sum_{k=1}^d\sum_{n'=1,n'\neq n}^N\sum_{i'=1}^dv^2Z_k x^{(i)}_n x_n^{(i)} x^{(i')}_{n'}x_{n'}^{(i')}\Lam_{k\ell}\right)
\\=&N\sum_{i=1}^d \sum_{k=1}^d\sum_{i'=1}^dv^2Z_k (\Lam_{ii}\Lam_{i'i'}+2\Lam_{ii'}\Lam_{ii'})\Lam_{k\ell}
+N(N-1)\sum_{i=1}^d \sum_{k=1}^d\sum_{i'=1}^dv^2Z_k \Lam_{ii}\Lam_{i'i'}\Lam_{k\ell}
\\=&2N\sum_{i=1}^d \sum_{k=1}^d\sum_{i'=1}^dv^2Z_k \Lam_{ii'}\Lam_{ii'}\Lam_{k\ell}
+N^2\sum_{i=1}^d \sum_{k=1}^d\sum_{i'=1}^dv^2Z_k \Lam_{ii}\Lam_{i'i'}\Lam_{k\ell}
\\=&2N(v^2Z^\top \Lam)_\ell \tr(\Lam^2)+ N^2(v^2 Z^\top \Lam)_\ell(\tr(\Lam))^2.
\end{align*}
\begin{align*}
L_2=&
\ex\left(\sum_{n=1}^N\sum_{i=1}^d \sum_{j=1}^d\sum_{k=1}^d\sum_{n'=1}^N\sum_{i'=1}^dv^2Z_k x^{(i)}_n x_n^{(i')}x^{(k)}_q x^{(i')}_{n'}x_{n'}^{(i)}x^{(\ell)}_q\right) 
\\=&
\ex\left(\sum_{n=1}^N\sum_{i=1}^d \sum_{j=1}^d\sum_{k=1}^d\sum_{n'=1}^N\sum_{i'=1}^dv^2Z_k x^{(i)}_n x_n^{(i')}  x^{(i')}_{n'}x_{n'}^{(i)}\Lam_{k\ell}\right) 
\\=&
\ex\left(\sum_{n=1}^N\sum_{i=1}^d \sum_{j=1}^d\sum_{k=1}^d\sum_{i'=1}^dv^2Z_k x^{(i)}_n x_n^{(i')}  x^{(i')}_{n}x_{n}^{(i)}\Lam_{k\ell}\right) 
\\&+
\ex\left(\sum_{n=1}^N\sum_{i=1}^d \sum_{j=1}^d\sum_{k=1}^d\sum_{n'=1, n'\neq n}^N\sum_{i'=1}^dv^2Z_k x^{(i)}_n x_n^{(i')}  x^{(i')}_{n'}x_{n'}^{(i)}\Lam_{k\ell}\right) 
\\=&N\sum_{i=1}^d \sum_{j=1}^d\sum_{k=1}^d\sum_{i'=1}^dv^2Z_k (2\Lam_{ii'}\Lam_{i'i}+\Lam_{ii}\Lam_{i'i'})
\Lam_{k\ell}+N(N-1)\sum_{i=1}^d \sum_{j=1}^d\sum_{k=1}^d\sum_{i'=1}^dv^2Z_k \Lam_{ii'}\Lam_{i'i}
\Lam_{k\ell}
\\=&N\sum_{i=1}^d \sum_{j=1}^d\sum_{k=1}^d\sum_{i'=1}^dv^2Z_k \Lam_{ii}\Lam_{i'i'}
\Lam_{k\ell}+N(N+1)\sum_{i=1}^d \sum_{j=1}^d\sum_{k=1}^d\sum_{i'=1}^dv^2Z_k \Lam_{ii'}\Lam_{i'i}
\Lam_{k\ell}
\\=&N(v^2 Z^\top\Lam)_\ell (\tr(\Lam))^2+N(N+1)(v^2Z^\top \Lam)_\ell \tr(\Lam^2).
\end{align*}

\begin{align*}
L_3=&\ex\left(\sum_{n=1}^N\sum_{i=1}^d \sum_{j=1}^d\sum_{k=1}^d\sum_{n'=1}^Nv^2Z_k x^{(i)}_n x_n^{(j)}x^{(k)}_q x^{(i)}_{n'}x_{n'}^{(j)}x^{(\ell)}_q\right)
\\=&\ex\left(\sum_{n=1}^N\sum_{i=1}^d \sum_{j=1}^d\sum_{k=1}^d\sum_{n'=1}^Nv^2Z_k x^{(i)}_n x_n^{(j)}x^{(i)}_{n'}x_{n'}^{(j)}\Lam_{k\ell}\right)
\\=&\ex\left(\sum_{n=1}^N\sum_{i=1}^d \sum_{j=1}^d\sum_{k=1}^dv^2Z_k x^{(i)}_n x_n^{(j)}x^{(i)}_{n}x_{n}^{(j)}\Lam_{k\ell}\right)
\\&+\ex\left(\sum_{n=1}^N\sum_{i=1}^d \sum_{j=1}^d\sum_{k=1}^d\sum_{n'=1,n'\neq n}^Nv^2Z_k x^{(i)}_n x_n^{(j)}x^{(i)}_{n'}x_{n'}^{(j)}\Lam_{k\ell}\right)
\\=& N\sum_{i=1}^d \sum_{j=1}^d\sum_{k=1}^dv^2Z_k(2\Lam_{ij}\Lam_{ij}+\Lam_{ii}\Lam_{jj})\Lam_{k\ell}+N(N-1)\sum_{i=1}^d \sum_{j=1}^d\sum_{k=1}^dv^2Z_k\Lam_{ij}\Lam_{ij}\Lam_{k\ell}
\\=& N\sum_{i=1}^d \sum_{j=1}^d\sum_{k=1}^dv^2Z_k\Lam_{ii}\Lam_{jj}\Lam_{k\ell}+N(N+1)\sum_{i=1}^d \sum_{j=1}^d\sum_{k=1}^dv^2Z_k\Lam_{ij}\Lam_{ij}\Lam_{k\ell}
\\=&N(v^2 Z^\top\Lam)_\ell (\tr(\Lam))^2+N(N+1)(v^2Z^\top \Lam)_\ell \tr(\Lam^2).
\end{align*}

Therefore, $L_1+L_2+L_3=N(N+4)(v^2 Z^\top\Lam)_\ell (\tr(\Lam))^2+2N(N+1)(v^2Z^\top \Lam)_\ell \tr(\Lam^2)$.

\begin{align*}
&\ex\left(\sum_{i=1}^d \sum_{j=1}^d\sum_{k=1}^dz_iU_{jk}x^{(i)}_q x_q^{(j)}x^{(k)}_q \right)\left(\sum_{n=1}^N\sum_{i=1}^d\sum_{j=1}^dz_iw_jx^{(i)}_nx_n^{(j)}x^{(\ell)}_q\right)=0.
\end{align*}

\begin{align*}
&\ex\left(\sum_{i=1}^d \sum_{j=1}^d\sum_{k=1}^dz_iU_{jk}x^{(i)}_q x_q^{(j)}x^{(k)}_q \right)\left(\sum_{n=1}^N\sum_{i=1}^d\sum_{j=1}^dvw_iw_jx^{(i)}_nx_n^{(j)}x^{(\ell)}_q\right)\\=&\ex\left(\sum_{i=1}^d \sum_{j=1}^d\sum_{k=1}^dz_iU_{jk}x^{(i)}_q x_q^{(j)}x^{(k)}_q \sum_{n'=1}^N\sum_{i'=1}^d\sum_{j'=1}^dvw_{i'}w_{j'}x^{(i')}_{n'}x_{n'}^{(j')}x^{(\ell)}_q\right)
\\=&\ex\left(\sum_{i=1}^d \sum_{j=1}^d\sum_{k=1}^dz_iU_{jk}x^{(i)}_q x_q^{(j)}x^{(k)}_q \sum_{n'=1}^N\sum_{i'=1}^dvx^{(i')}_{n'}x_{n'}^{(i')}x^{(\ell)}_q\right)
\\=&N\ex\left(\sum_{i=1}^d \sum_{j=1}^d\sum_{k=1}^d\sum_{i'=1}^dvz_iU_{jk}x^{(i)}_q x_q^{(j)}x^{(k)}_q x^{(\ell)}_q\Lam_{i'i'}\right)
\\=&N\ex\left(\sum_{i=1}^d \sum_{j=1}^d\sum_{k=1}^d\sum_{i'=1}^dvz_iU_{jk}(\Lam_{ij}\Lam_{k\ell}+\Lam_{ik}\Lam_{j\ell}+\Lam_{i\ell}\Lam_{jk})\Lam_{i'i'}\right)
\\=& N(vz^\top\Lam U\Lam)_\ell\tr(\Lam)+ N(vz^\top U^\top \Lam)_\ell\tr(\Lam)+ N(vz^\top\Lam)_\ell\tr(U\Lam)\tr(\Lam).
\end{align*}

\begin{align*}
&\ex\left(\sum_{m=1}^d w_mx_q^{(m)}\sum_{n=1}^N \sum_{i=1}^d\sum_{j=1}^d(z_iw_j+vw_iw_j)x^{(i)}_n x_n^{(j)}x^{(\ell)}_q \right)\\ =&\ex\left(\sum_{m=1}^d \sum_{n=1}^N \sum_{i=1}^d\sum_{j=1}^dz_iw_jw_mx_q^{(m)}x^{(i)}_n x_n^{(j)}x^{(\ell)}_q \right)
\\ =&\ex\left( \sum_{n=1}^N \sum_{i=1}^d\sum_{j=1}^dz_ix_q^{(j)}x^{(i)}_n x_n^{(j)}x^{(\ell)}_q \right)
\\ =&N\sum_{i=1}^d\sum_{j=1}^dz_i\Lam_{j\ell}\Lam_{ij}=N(z^\top\Lam^2)_\ell.
\end{align*}

\subsection{Detailed Calculations of $\frac{\partial L}{\partial U}$}
\label{sec:DU}

Recall that, for $\ell, p=1,2,\ldots d$, we have, 
\begin{align*}
\frac{\partial L}{\partial U_{\ell p}}= &\ex\left[\frac{1}{N}\sum_{n=1}^N\sum_{i=1}^d \sum_{j=1}^d\sum_{k=1}^d(z_i  U_{jk}+z_iZ_k w_j+vw_iU_{jk}+vZ_kw_iw_j)x^{(i)}_n x_n^{(j)}x^{(k)}_q \right.\\ &\left. +\frac{1}{N}\sum_{i=1}^d\sum_{j=1}^d \sum_{k=1}^d (z_iU_{jk}) x_q^{(j)}x_q^{(i)}  x_q^{(k)}-\sum_{m=1}^d w_mx_q^{(m)}\right]\\  &\times \frac{1}{N}\left(\sum_{n=1}^N \sum_{i=1}^d(z_i+vw_i)x^{(i)}_n x^{(\ell)}_n x_q^{(p)}+ \sum_{i=1}^dz_i x_q^{(i)}  x_q^{(\ell)}x_q^{(p)}\right).
\end{align*}

Calculations below gives us, 
\begin{align*}
\frac{\partial L}{\partial U_{\ell p}}= &\frac{1}{N}(z^\top \Lam z)(\Lam U\Lam)_{\ell p}+ \frac{N+3}{N}(z^\top\Lam)_\ell (z^\top \Lam U\Lam)_p \nonumber
\\
 &+\frac{1}{N}(z^\top\Lam U\Lam z)(\Lam)_{\ell P}+\frac{1}{N}(z^\top\Lam U\Lam)_\ell (z^\top \Lam)_p  \nonumber
 \\ &+
 \frac{N+1}{N}v\tr(\Lam) (z^\top \Lam)_\ell (Z^\top\Lam)_p+ \frac{N+3}{N}v(z^\top \Lam^2)_\ell (Z^\top \Lam)_p \nonumber
 \\&+
 \frac{1}{N} v^2(\Lam U\Lam)_{\ell p}\tr(\Lam)+\frac{N+1}{N} v^2 (\Lam ^2U\Lam)_{\ell p} \nonumber
 \\ &+
 \frac{1}{N}v [(z^\top\Lam Z) \Lam_{\ell p} + 2(Z^\top \Lam )_\ell(z^\top \Lam)_p]\tr(\Lam) \nonumber
\\&+ \frac{1}{N}(z^\top\Lam)_\ell(z^\top\Lam U^\top\Lam)_p+ \frac{1}{N}(z^\top\Lam)_\ell(z^\top\Lam)_p\tr(U\Lam) \nonumber
\\&+\frac{1}{N^2}[(z^\top \Lam U\Lam z)\Lam_{\ell p}+ 2(z^\top \Lam U\Lam)_\ell (z^\top \Lam)_p+ 2(z^\top \Lam)_\ell(z^\top \Lam U\Lam)_p \nonumber
\\&+ (z^\top \Lam U^\top \Lam z) \Lam_{\ell p} + 2(z^\top \Lam)_\ell(z^\top \Lam U^\top\Lam)_p \nonumber
\\&+ (z^\top \Lam z) \tr(U\Lam) \Lam_{\ell p}+  (z^\top \Lam z)(\Lam U^\top \Lam)_{\ell p} + (z^\top \Lam z)(\Lam U \Lam)_{\ell p}\nonumber
\\ &+ 2(z^\top \Lam U^\top \Lam)_\ell (z^\top \Lam)_p + 2(z^\top \Lam)_\ell \tr(U \Lam) (z^\top \Lam)_p] \nonumber
\\ &-v(\Lam^2)_{\ell p}.
\end{align*}

Below, we will collect all the calculations, with the same arguments that have already been introduced. 
\begin{align*}
&\ex\left(\sum_{n=1}^N\sum_{i=1}^d \sum_{j=1}^d\sum_{k=1}^dz_iU_{jk}x^{(i)}_n x_n^{(j)}x^{(k)}_q \right)\left(\sum_{n=1}^N\sum_{i=1}^dz_ix^{(i)}_n x^{(\ell)}_n x_q^{(p)}\right)
\\=&\ex\left(\sum_{n=1}^N\sum_{i=1}^d \sum_{j=1}^d\sum_{k=1}^dz_iU_{jk}x^{(i)}_n x_n^{(j)}x^{(k)}_q \sum_{n'=1}^N\sum_{i'=1}^dz_{i'}x^{(i')}_{n'} x^{(\ell)}_{n'} x_q^{(p)}\right)
\\=&\ex\left(\sum_{n=1}^N\sum_{i=1}^d \sum_{j=1}^d\sum_{k=1}^d\sum_{n'=1}^N\sum_{i'=1}^dz_iz_{i'}U_{jk}x^{(i)}_n x_n^{(j)}x^{(i')}_{n'} x^{(\ell)}_{n'} \Lam_{kp}\right)
\\=&\ex\left(\sum_{n=1}^N\sum_{i=1}^d \sum_{j=1}^d\sum_{k=1}^d\sum_{i'=1}^dz_iz_{i'}U_{jk}x^{(i)}_n x_n^{(j)}x^{(i')}_{n} x^{(\ell)}_{n} \Lam_{kp}\right)
\\&+\ex\left(\sum_{n=1}^N\sum_{i=1}^d \sum_{j=1}^d\sum_{k=1}^d\sum_{n'=1, n'\neq n}^N\sum_{i'=1}^dz_iz_{i'}U_{jk}x^{(i)}_n x_n^{(j)}x^{(i')}_{n'} x^{(\ell)}_{n'} \Lam_{kp}\right)
\\=&N\sum_{i=1}^d \sum_{j=1}^d\sum_{k=1}^d\sum_{i'=1}^dz_iz_{i'}U_{jk}(\Lam_{ij}\Lam_{i'\ell}+\Lam_{ii'}\Lam_{j\ell}+\Lam_{i\ell}\Lam_{ji'}) \Lam_{kp}
\\&+N(N-1)\sum_{n=1}^N\sum_{i=1}^d \sum_{j=1}^d\sum_{k=1}^d\sum_{i'=1}^dz_iz_{i'}U_{jk}\Lam_{ij}\Lam_{i'\ell} \Lam_{kp}
\\=&N\sum_{i=1}^d \sum_{j=1}^d\sum_{k=1}^d\sum_{i'=1}^dz_iz_{i'}U_{jk}(\Lam_{ii'}\Lam_{j\ell}+\Lam_{i\ell}\Lam_{ji'}) \Lam_{kp}
\\&+N^2\sum_{n=1}^N\sum_{i=1}^d \sum_{j=1}^d\sum_{k=1}^d\sum_{i'=1}^dz_iz_{i'}U_{jk}\Lam_{ij}\Lam_{i'\ell} \Lam_{kp}
\\=&N(z^\top \Lam z)(\Lam U\Lam)_{\ell p}+ N(N+1)(z^\top\Lam)_\ell (z^\top \Lam U\Lam)_p.
\end{align*}

\begin{align*}
&\ex\left(\sum_{n=1}^N\sum_{i=1}^d \sum_{j=1}^d\sum_{k=1}^dz_iU_{jk}x^{(i)}_n x_n^{(j)}x^{(k)}_q \right)\left(\sum_{n=1}^N\sum_{i=1}^dvw_ix^{(i)}_n x^{(\ell)}_n x_q^{(p)}\right)=0.
\end{align*}

\newpage

\begin{align*}
&\ex\left(\sum_{n=1}^N\sum_{i=1}^d \sum_{j=1}^d\sum_{k=1}^dz_iU_{jk}x^{(i)}_n x_n^{(j)}x^{(k)}_q \right)\left(\sum_{i=1}^dz_ix^{(i)}_q x^{(\ell)}_qx_q^{(p)}\right)
\\=&\ex\left(\sum_{n=1}^N\sum_{i=1}^d \sum_{j=1}^d\sum_{k=1}^dz_iU_{jk}x^{(i)}_n x_n^{(j)}x^{(k)}_q \sum_{i'=1}^dz_{i'}x^{(i')}_q x^{(\ell)}_qx_q^{(p)}\right)
\\=&N\ex\left(\sum_{i=1}^d \sum_{j=1}^d\sum_{k=1}^d\sum_{i'=1}^dz_iz_{i'}U_{jk}\Lam_{ij}x^{(k)}_q x^{(i')}_q x^{(\ell)}_qx_q^{(p)}\right)
\\=&N\sum_{i=1}^d \sum_{j=1}^d\sum_{k=1}^d\sum_{i'=1}^dz_iz_{i'}U_{jk}\Lam_{ij}(\Lam_{ki'}\Lam_{\ell p}+\Lam_{k\ell}\Lam_{i' p}+\Lam_{kp}\Lam_{\ell i'})
\\=&N(z^\top\Lam U\Lam z)(\Lam)_{\ell P}+N(z^\top\Lam U\Lam)_\ell (z^\top \Lam)_p +N(z^\top \Lam)_\ell(z^\top\Lam U\Lam)_p.
\end{align*}

\begin{align*}
&\ex\left(\sum_{n=1}^N\sum_{i=1}^d \sum_{j=1}^d\sum_{k=1}^dz_iZ_kw_jx^{(i)}_n x_n^{(j)}x^{(k)}_q \right)\left(\sum_{n=1}^N\sum_{i=1}^dz_ix^{(i)}_n x^{(\ell)}_n x_q^{(p)}\right)=0.
\end{align*}

\begin{align*}
&\ex\left(\sum_{n=1}^N\sum_{i=1}^d \sum_{j=1}^d\sum_{k=1}^dz_iZ_kw_jx^{(i)}_n x_n^{(j)}x^{(k)}_q \right)\left(\sum_{n=1}^N\sum_{i=1}^dvw_ix^{(i)}_n x^{(\ell)}_n x_q^{(p)}\right)
\\=&\ex\left(\sum_{n=1}^N\sum_{i=1}^d \sum_{j=1}^d\sum_{k=1}^dz_iZ_kw_jx^{(i)}_n x_n^{(j)}x^{(k)}_q \sum_{n'=1}^N\sum_{i'=1}^dvw_{i'}x^{(i')}_{n'} x^{(\ell)}_{n'} x_q^{(p)}\right)
\\=&\ex\left(\sum_{n=1}^N\sum_{i=1}^d \sum_{j=1}^d\sum_{k=1}^d \sum_{n'=1}^Nvz_iZ_kx^{(i)}_n x_n^{(j)}x^{(k)}_qx^{(j)}_{n'} x^{(\ell)}_{n'} x_q^{(p)}\right)
\\=&\ex\left(\sum_{n=1}^N\sum_{i=1}^d \sum_{j=1}^d\sum_{k=1}^d \sum_{n'=1}^Nvz_iZ_kx^{(i)}_n x_n^{(j)}x^{(j)}_{n'} x^{(\ell)}_{n'} \Lam_{kp}\right)
\\=&\ex\left(\sum_{n=1}^N\sum_{i=1}^d \sum_{j=1}^d\sum_{k=1}^d vz_iZ_kx^{(i)}_n x_n^{(j)}x^{(j)}_{n} x^{(\ell)}_{n} \Lam_{kp}\right)
\\&+\ex\left(\sum_{n=1}^N\sum_{i=1}^d \sum_{j=1}^d\sum_{k=1}^d \sum_{n'=1, n'\neq n}^Nvz_iZ_kx^{(i)}_n x_n^{(j)}x^{(j)}_{n'} x^{(\ell)}_{n'} \Lam_{kp}\right)
\\=&N\sum_{i=1}^d \sum_{j=1}^d\sum_{k=1}^d vz_iZ_k(2\Lam_{ij}\Lam_{j\ell}+ \Lam_{i\ell}\Lam_{jj})\Lam_{kp}
+N(N-1)\sum_{i=1}^d \sum_{j=1}^d\sum_{k=1}^d vz_iZ_k\Lam_{ij}\Lam_{j\ell}\Lam_{kp}
\\=&N\sum_{i=1}^d \sum_{j=1}^d\sum_{k=1}^d vz_iZ_k \Lam_{i\ell}\Lam_{jj}\Lam_{kp}
+N(N+1)\sum_{i=1}^d \sum_{j=1}^d\sum_{k=1}^d vz_iZ_k\Lam_{ij}\Lam_{j\ell}\Lam_{kp}
\\=& Nv\tr(\Lam) (z^\top \Lam)_\ell (Z^\top\Lam)_p+ N(N+1)v(z^\top \Lam^2)_\ell (Z^\top \Lam)_p.
\end{align*}

\begin{align*}
&\ex\left(\sum_{n=1}^N\sum_{i=1}^d \sum_{j=1}^d\sum_{k=1}^dz_iZ_kw_jx^{(i)}_n x_n^{(j)}x^{(k)}_q \right)\left(\sum_{i=1}^dz_ix^{(i)}_q x^{(\ell)}_qx_q^{(p)}\right)=0.
\end{align*}

\begin{align*}
&\ex\left(\sum_{n=1}^N\sum_{i=1}^d \sum_{j=1}^d\sum_{k=1}^dvw_iU_{jk}x^{(i)}_n x_n^{(j)}x^{(k)}_q \right)\left(\sum_{n=1}^N\sum_{i=1}^dz_ix^{(i)}_n x^{(\ell)}_n x_q^{(p)}\right)=0.
\end{align*}

\begin{align*}
&\ex\left(\sum_{n=1}^N\sum_{i=1}^d \sum_{j=1}^d\sum_{k=1}^dvw_iU_{jk}x^{(i)}_n x_n^{(j)}x^{(k)}_q \right)\left(\sum_{n=1}^N\sum_{i=1}^dvw_ix^{(i)}_n x^{(\ell)}_n x_q^{(p)}\right)
\\=&\ex\left(\sum_{n=1}^N\sum_{i=1}^d \sum_{j=1}^d\sum_{k=1}^dvw_iU_{jk}x^{(i)}_n x_n^{(j)}x^{(k)}_q \sum_{n'=1}^N\sum_{i'=1}^dvw_{i'}x^{(i')}_{n'} x^{(\ell)}_{n'} x_q^{(p)}\right)
\\=&\ex\left(\sum_{n=1}^N\sum_{i=1}^d \sum_{j=1}^d\sum_{k=1}^d\sum_{n'=1}^Nv^2U_{jk}x^{(i)}_n x_n^{(j)}x^{(k)}_q x^{(i)}_{n'} x^{(\ell)}_{n'} x_q^{(p)}\right)
\\=&\ex\left(\sum_{n=1}^N\sum_{i=1}^d \sum_{j=1}^d\sum_{k=1}^d\sum_{n'=1}^Nv^2U_{jk}x^{(i)}_n x_n^{(j)} x^{(i)}_{n'} x^{(\ell)}_{n'} \Lam_{kp}\right)
\\=&\ex\left(\sum_{n=1}^N\sum_{i=1}^d \sum_{j=1}^d\sum_{k=1}^dv^2U_{jk}x^{(i)}_n x_n^{(j)} x^{(i)}_{n} x^{(\ell)}_{n} \Lam_{kp}\right)\\&+\ex\left(\sum_{n=1}^N\sum_{i=1}^d \sum_{j=1}^d\sum_{k=1}^d\sum_{n'=1, n'\neq n}^Nv^2U_{jk}x^{(i)}_n x_n^{(j)} x^{(i)}_{n'} x^{(\ell)}_{n'} \Lam_{kp}\right)
\\=& N\sum_{i=1}^d \sum_{j=1}^d\sum_{k=1}^dv^2U_{jk}(2\Lam_{ij}\Lam_{i\ell}+\Lam_{ii}\Lam_{j\ell}) \Lam_{kp}+N(N-1)\sum_{i=1}^d \sum_{j=1}^d\sum_{k=1}^dv^2U_{jk}\Lam_{ij}\Lam_{i\ell} \Lam_{kp}
\\=& N\sum_{i=1}^d \sum_{j=1}^d\sum_{k=1}^dv^2U_{jk}\Lam_{ii}\Lam_{j\ell} \Lam_{kp}+N(N+1)\sum_{i=1}^d \sum_{j=1}^d\sum_{k=1}^dv^2U_{jk}\Lam_{ij}\Lam_{i\ell} \Lam_{kp}
\\=& N v^2(\Lam U\Lam)_{\ell p}\tr(\Lam)+N(N+1) v^2 (\Lam ^2U\Lam)_{\ell p}.
\end{align*}

\begin{align*}
&\ex\left(\sum_{n=1}^N\sum_{i=1}^d \sum_{j=1}^d\sum_{k=1}^dvw_iU_{jk}x^{(i)}_n x_n^{(j)}x^{(k)}_q \right)\left(\sum_{i=1}^dz_ix^{(i)}_q x^{(\ell)}_qx_q^{(p)}\right)=0.
\end{align*}

\begin{align*}
&\ex\left(\sum_{n=1}^N\sum_{i=1}^d \sum_{j=1}^d\sum_{k=1}^dvZ_kw_iw_jx^{(i)}_n x_n^{(j)}x^{(k)}_q \right)\left(\sum_{n=1}^N\sum_{i=1}^dz_ix^{(i)}_n x^{(\ell)}_n x_q^{(p)}\right)
\\=&\ex\left(\sum_{n=1}^N\sum_{i=1}^d \sum_{j=1}^d\sum_{k=1}^dvZ_kw_iw_jx^{(i)}_n x_n^{(j)}x^{(k)}_q \sum_{n'=1}^N\sum_{i'=1}^dz_{i'}x^{(i')}_{n'} x^{(\ell)}_{n'} x_q^{(p)}\right)
\\=&\ex\left(\sum_{n=1}^N\sum_{i=1}^d \sum_{k=1}^d\sum_{n'=1}^N\sum_{i'=1}^dz_{i'}vZ_kx^{(i)}_n x_n^{(i)}x^{(k)}_q x^{(i')}_{n'} x^{(\ell)}_{n'} x_q^{(p)}\right)
\\=&\ex\left(\sum_{n=1}^N\sum_{i=1}^d \sum_{k=1}^d\sum_{n'=1}^N\sum_{i'=1}^dz_{i'}vZ_kx^{(i)}_n x_n^{(i)}x^{(i')}_{n'} x^{(\ell)}_{n'} \Lam_{kp}\right)
\\=&\ex\left(\sum_{n=1}^N\sum_{i=1}^d \sum_{k=1}^d\sum_{i'=1}^dz_{i'}vZ_kx^{(i)}_n x_n^{(i)}x^{(i')}_{n} x^{(\ell)}_{n} \Lam_{kp}\right)
\\&+\ex\left(\sum_{n=1}^N\sum_{i=1}^d \sum_{k=1}^d\sum_{n'=1, n'\neq n}^N\sum_{i'=1}^dz_{i'}vZ_kx^{(i)}_n x_n^{(i)}x^{(i')}_{n'} x^{(\ell)}_{n'} \Lam_{kp}\right)
\\=& N\sum_{i=1}^d \sum_{k=1}^d\sum_{i'=1}^dz_{i'}vZ_k(\Lam_{ii}\Lam_{i'\ell}+2\Lam_{ii'}\Lam_{i\ell})\Lam_{kp}
+N(N-1)\sum_{i=1}^d \sum_{k=1}^d\sum_{n'=1, n'\neq n}^N\sum_{i'=1}^dz_{i'}vZ_k\Lam_{ii}\Lam_{i'\ell}\Lam_{kp} 
\\=& 2 N\sum_{i=1}^d \sum_{k=1}^d\sum_{i'=1}^dz_{i'}vZ_k\Lam_{ii'}\Lam_{i\ell}\Lam_{kp}
+N^2\sum_{i=1}^d \sum_{k=1}^d\sum_{n'=1, n'\neq n}^N\sum_{i'=1}^dz_{i'}vZ_k\Lam_{ii}\Lam_{i'\ell}\Lam_{kp} 
\\=&2N v (z^\top\Lam^2)_\ell (Z^\top\Lam)_p+ N^2v(z^\top \Lam)_\ell (Z^\top\Lam)_p \tr(\Lam).
\end{align*}

\begin{align*}
&\ex\left(\sum_{n=1}^N\sum_{i=1}^d \sum_{j=1}^d\sum_{k=1}^dvZ_kw_iw_jx^{(i)}_n x_n^{(j)}x^{(k)}_q \right)\left(\sum_{n=1}^N\sum_{i=1}^dvw_ix^{(i)}_n x^{(\ell)}_n x_q^{(p)}\right)=0.
\end{align*}

\begin{align*}
&\ex\left(\sum_{n=1}^N\sum_{i=1}^d \sum_{j=1}^d\sum_{k=1}^dvZ_kw_iw_jx^{(i)}_n x_n^{(j)}x^{(k)}_q \right)\left(\sum_{i=1}^dz_ix^{(i)}_q x^{(\ell)}_qx_q^{(p)}\right)
\\=&\ex\left(\sum_{n=1}^N\sum_{i=1}^d \sum_{j=1}^d\sum_{k=1}^dvZ_kw_iw_jx^{(i)}_n x_n^{(j)}x^{(k)}_q \sum_{i'=1}^dz_{i'}x^{(i')}_q x^{(\ell)}_qx_q^{(p)}\right)
\\=&\ex\left(\sum_{n=1}^N\sum_{i=1}^d \sum_{k=1}^d\sum_{i'=1}^dv^2z_{i'}Z_kx^{(i)}_n x_n^{(i)}x^{(k)}_q x^{(i')}_q x^{(\ell)}_qx_q^{(p)}\right)
\\=&N\tr(\Lam)\ex\left(\sum_{k=1}^d\sum_{i'=1}^dvz_{i'}Z_kx^{(k)}_q x^{(i')}_q x^{(\ell)}_qx_q^{(p)}\right)
\\=&N\tr(\Lam)\sum_{k=1}^d\sum_{i'=1}^dvz_{i'}Z_k(\Lam_{ki'}\Lam_{\ell p}+\Lam_{k\ell}\Lam_{i' p}+\Lam_{kp}\Lam_{\ell i'})
\\=&N\tr(\Lam)v ((z^\top\Lam Z) \Lam_{\ell p} + 2(Z^\top \Lam )_\ell(z^\top \Lam)_p).
\end{align*}

\begin{align*}
&\ex\left(\sum_{i=1}^d \sum_{j=1}^d\sum_{k=1}^dz_iU_{jk}x^{(i)}_qx_q^{(j)}x^{(k)}_q \right)\left(\sum_{n=1}^N\sum_{i=1}^dz_ix^{(i)}_n x^{(\ell)}_n x_q^{(p)}\right)
\\=&\ex\left(\sum_{i=1}^d \sum_{j=1}^d\sum_{k=1}^dz_iU_{jk}x^{(i)}_qx_q^{(j)}x^{(k)}_q \sum_{n=1}^N\sum_{i'=1}^dz_{i'}x^{(i')}_{n} x^{(\ell)}_n x_q^{(p)}\right)
\\=&N\ex\left(\sum_{i=1}^d \sum_{j=1}^d\sum_{k=1}^d\sum_{i'=1}^dz_{i'}z_iU_{jk}x^{(i)}_qx_q^{(j)}x^{(k)}_q  x_q^{(p)}\Lam_{i'\ell}\right)
\\=&N\sum_{i=1}^d \sum_{j=1}^d\sum_{k=1}^d\sum_{i'=1}^dz_{i'}z_iU_{jk}(\Lam_{ij}\Lam_{kp}+\Lam_{ik}\Lam_{jp}+\Lam_{ip}\Lam_{jk})\Lam_{i'\ell}
\\=&N(z^\top\Lam)_\ell(z^\top\Lam U\Lam)_p+ N(z^\top\Lam)_\ell(z^\top\Lam U^\top\Lam)_p+ N(z^\top\Lam)_\ell(z^\top\Lam)_p\tr(U\Lam).
\end{align*}

\begin{align*}
&\ex\left(\sum_{i=1}^d \sum_{j=1}^d\sum_{k=1}^dz_iU_{jk}x^{(i)}_qx_q^{(j)}x^{(k)}_q \right)\left(\sum_{n=1}^N\sum_{i=1}^dvw_ix^{(i)}_n x^{(\ell)}_n x_q^{(p)}\right)=0.
\end{align*}

\begin{align*}
&\ex\left(\sum_{i=1}^d \sum_{j=1}^d\sum_{k=1}^dz_iU_{jk}x^{(i)}_qx_q^{(j)}x^{(k)}_q \right)\left(\sum_{i=1}^dz_ix^{(i)}_q x^{(\ell)}_qx_q^{(p)}\right)
\\=&\ex\left(\sum_{i=1}^d \sum_{j=1}^d\sum_{k=1}^dz_iU_{jk}x^{(i)}_qx_q^{(j)}x^{(k)}_q \sum_{i'=1}^dz_{i'}x^{(i')}_q x^{(\ell)}_qx_q^{(p)}\right)
\\=&\ex\left(\sum_{i=1}^d \sum_{j=1}^d\sum_{k=1}^d\sum_{i'=1}^dz_{i'}z_iU_{jk}x^{(i)}_qx_q^{(j)}x^{(k)}_q x^{(i')}_q x^{(\ell)}_qx_q^{(p)}\right)
\\=&\sum_{i=1}^d \sum_{j=1}^d\sum_{k=1}^d\sum_{i'=1}^dz_{i'}z_iU_{jk}[\Lam_{ij}(\Lam_{ki'}\Lam_{\ell p}+\Lam_{k\ell}\Lam_{i' p}+\Lam_{kp}\Lam_{\ell i'})\\ &+ \Lam_{ik}(\Lam_{ji'}\Lam_{\ell p}+\Lam_{j\ell}\Lam_{i' p}+\Lam_{jp}\Lam_{\ell i'})+\Lam_{ii'}(\Lam_{kj}\Lam_{\ell p}+\Lam_{k\ell}\Lam_{j p}+\Lam_{kp}\Lam_{\ell j})
\\&+\Lam_{i\ell}(\Lam_{ki'}\Lam_{j p}+\Lam_{kj}\Lam_{i' p}+\Lam_{kp}\Lam_{j i'})+\Lam_{ip}(\Lam_{ki'}\Lam_{\ell j}+\Lam_{k\ell}\Lam_{i' j}+\Lam_{kj}\Lam_{\ell i'})]
\\= & (z^\top \Lam U\Lam z)\Lam_{\ell p}+ (z^\top \Lam U\Lam)_\ell (z^\top \Lam)_p+ (z^\top \Lam)_\ell(z^\top \Lam U\Lam)_p
\\&+ (z^\top \Lam U^\top \Lam z) \Lam_{\ell p}+ (z^\top\Lam U^\top \Lam)_\ell(z^\top \Lam)_p + (z^\top \Lam)_\ell(z^\top \Lam U^\top\Lam)_p
\\&+ (z^\top \Lam z) \tr(U\Lam) \Lam_{\ell p}+  (z^\top \Lam z)(\Lam U^\top \Lam)_{\ell p} + (z^\top \Lam z)(\Lam U \Lam)_{\ell p}\\ &+ (z^\top \Lam)_\ell(z^\top\Lam U^\top\Lam)_p +(z^\top \Lam)_\ell(z^\top\Lam)_p\tr(U\Lam)_p+  (z^\top \Lam)_\ell(z^\top\Lam U\Lam)_p
\\&+(z^\top \Lam U^\top \Lam)_\ell (z^\top \lam)_p + (z^\top \Lam U \Lam)_\ell (z^\top \lam)_p + (z^\top \Lam)_\ell \tr(U \Lam) (z^\top \lam)_p.
\end{align*}

\begin{align*}
&\ex\left[\sum_{m=1}^d w_mx_q^{(m)}\left(\sum_{n=1}^N \sum_{i=1}^d(z_i+vw_i)x^{(i)}_n x^{(\ell)}_n x_q^{(p)}+ \sum_{i=1}^dz_i x_q^{(i)}  x_q^{(\ell)}x_q^{(p)}\right)\right]
\\=&\ex\left[\sum_{m=1}^d\sum_{n=1}^N \sum_{i=1}^d w_mx_q^{(m)}vw_ix^{(i)}_n x^{(\ell)}_n x_q^{(p)}\right]
\\=&\ex\left[\sum_{n=1}^N \sum_{i=1}^dx_q^{(i)}vx^{(i)}_n x^{(\ell)}_n x_q^{(p)}\right]
\\=& N \sum_{i=1}^dv\Lam_{i\ell}\Lam_{ip}=Nv(\Lam^2)_{\ell p}.
\end{align*}

\end{document}